\documentclass[12pts]{amsart}
\usepackage{amsthm}
\usepackage{amsfonts}

\usepackage{graphics,psfrag,graphicx,subfigure,epsfig, xypic}
\usepackage{amsmath,amsfonts,amssymb,amstext,amsthm,amscd,mathrsfs}
\usepackage[spanish,english]{babel}
\usepackage[all]{xy}
\usepackage{multicol}
\usepackage{rotating}

\newtheorem{theorem}{Theorem}[section]
\newtheorem{corollary}[theorem]{Corollary}
\newtheorem{definition}[theorem]{Definition}
\newtheorem{lemma}[theorem]{Lemma}
\newtheorem{proposition}[theorem]{Proposition}
\newtheorem{remark}[theorem]{Remark}

\newcommand{\noi}{{\noindent}}

\newcommand{\C}{\mathbb{C}}
\newcommand{\R}{\mathbb{R}}
\newcommand{\Z}{\mathbb{Z}}

\newcommand{\q} {\mathbf{q}}
\newcommand{\ii} {\mathbf{i}}
\newcommand{\jj} {\mathbf{j}}
\newcommand{\kk} {\mathbf{k}}
\newcommand{\pp} {\mathbf{p}}
\newcommand{\Hy} {\mathbf{H}_{\mathbb{H}}^1}

\usepackage{latexsym} \usepackage{amsmath} \usepackage{amssymb}

\begin{document}
\title[Modular Groups and Orbifolds over Quaternions]{Quaternionic Kleinian Modular Groups and Arithmetic Hyperbolic Orbifolds over the Quaternions.}

\author[J.P. Diaz -- A. Verjovsky -- F. Vlacci]
{Juan  Pablo D\'iaz, Alberto Verjovsky and 
Fabio Vlacci.}
 \address{J.P. Diaz Gonz\'alez, 
A. Verjovsky --
Instituto de Matem\'aticas, Unidad Cuernavaca, Universidad Nacional Aut\'onoma de M\'exico
 Av. Universidad s/n. Col. Lomas de Chamilpa
C\'odigo Postal 62210, Cuernavaca, Morelos,  M\'exico}

\email{juanpablo@matcuer.mx}
\email{alberto@matcuer.mx}
\address{F. Vlacci -- Dipartimento
di Matematica e Informatica DiMaI ``U. Dini'', 
Universit\`a di Firenze
Viale Morgagni 67/A, 50134 - Firenze, Italy.} 

\email{vlacci@math.unifi.it}
 \maketitle 
 
\begin{center}\emph{In memory of Marco Brunella (1964-2012)}
\end{center}

\begin{abstract} Using the rings of Lipschitz and Hurwitz integers $\mathbb{H}(\mathbb{Z})$ and  $\mathbb{H}ur(\mathbb{Z})$
in the quaternion division algebra $\mathbb{H}$, we define several Kleinian discrete subgroups of $PSL(2,\mathbb{H})$.
We define first a Kleinian subgroup $PSL(2,\mathfrak{L})$ of
$PSL(2,\mathbb{H}(\mathbb{Z}))$. This
group is a generalization of the modular group $PSL(2,\Z)$. Next we define a discrete subgroup $PSL(2,\mathfrak{H})$ of
$PSL(2,\mathbb{H})$ which is obtained by using Hurwitz integers and
in particular the subgroup of order 24 consisting of Hurwitz units. It
contains as a subgroup $PSL(2,\mathfrak{L})$. In analogy
with the classical modular case, these groups act properly and
discontinuously on the hyperbolic half space
$\mathbf{H}^1_{\mathbb{H}}:=\{\q\in\mathbb{H}\,:\, \Re(\mathbf{q})>0   \}$. We exhibit 
fundamental domains of the actions of these groups and determine the
isotropy groups of the fixed points and describe the orbifold quotients
$\mathbf{H}_{\mathbb{H}}^1/PSL(2,\mathfrak{L})$ and $\mathbf{H}_{\mathbb{H}}^1/PSL(2,\mathfrak{H})$ which are quaternionic
versions of the classical modular orbifold and they are of finite
volume.   We give a thorough study
of the Iwasawa decompositions, affine subgroups, and
their descriptions by Lorentz transformations in the Lorentz-Minkowski
model of hyperbolic 4-space. We give abstract finite presentations of these modular groups in terms of
generators and relations via the Cayley graphs associated to the
fundamental domains. We also describe a set of Selberg covers (corresponding to finite-index subgroups acting freely)
which are quaternionic hyperbolic manifolds of finite volume with
cusps whose sections are 3-tori. These hyperbolic arithmetic 4-manifolds
are topologically the complement of linked 2-tori in the 4-sphere, in
analogy with the complement in the 3-sphere of the Borromean rings and are related to the ubiquitous hyperbolic 24-cell. Finally we study the Poincar\'e
extensions of these Kleinian groups to arithmetic Kleinian groups acting on hyperbolic 5-space and described in the quaternionic
setting. In particular $PSL(2,\mathbb{H}(\mathbb{Z}))$ and $PSL(2,\mathbb{H}ur(\mathbb{Z}))$ are discrete subgroups of isometries of 
$\mathbf{H}_\mathbb{R}^5$ and $\mathbf{H}_\mathbb{R}^5/PSL(2,\mathbb{H}(\mathbb{Z}))$,  
$\mathbf{H}_\mathbb{R}^5/PSL(2,\mathbb{H}ur(\mathbb{Z}))$ are examples of arithmetic 5-dimensional hyperbolic orbifolds of finite volume.

\end{abstract}

\noi {\bf Keywords:} Quaternionic modular groups, arithmetic hyperbolic 4-manifolds and orbifolds.

\noi  {\bf AMS subject classification:} Primary 20H10, 57S30, 11F06. Secondary 30G35, 30F45.

\tableofcontents 

\section{Introduction}

\noi  Since the time of Carl Friedrich Gauss and the foundational papers by Richard Dedekind and Felix Klein the classical modular group $PSL(2,\mathbb Z)$ and its action on
the hyperbolic (complex) upper half plane $\{z\in\mathbb{C}\ :\ \Im(z)>0\}$
have played a central role in different branches
of Mathematics and Physics like Number Theory, Riemann surfaces, Elliptic Curves,
Hyperbolic Geometry, Cristallography, String Theory and others.  Similarly, discrete
subgroups of $PSL(2,\mathbb C)$ are very important in the
construction of lattices to study of arithmetic hyperbolic
3-orbifolds (see, for instance, \cite{MR}) and many other fields of Mathematics. Another very important subject derived from the modular group is the theory of modular forms and automorphic forms in general  (see for example, \cite{BGHZ}, \cite{FK}, \cite{Ma}, \cite{Ser}). 

 \noi  In this paper, we introduce two generalizations of the modular group in
  the setting of the quaternions and the rings of Lipschitz and Hurwitz integers (\cite{H1} \cite{H2} \cite{L}) and then focus our attention to their
  actions on hyperbolic (quaternionic) half space
  $\mathbf{H}_{\mathbb{H}}^1:=\{\mathbf{q}\in\mathbb{H}\ :\ \Re(\mathbf{q})>0\}$
  with metric $\frac{d\mathbf|{q}|^2}{(\Re{\mathbf{q})}^2}$. This can
  be done thanks to a recent algebraic characterization for linear
  transformations to leave invariant the hyperbolic half space (\cite{BisGen}).  We
  then explore new results which give a very detailed description of
  the orbifolds associated to the corresponding quaternionic modular groups.
  Generators of these groups are obtained in complete analogy with the
  classical modular group $PSL(2,\mathbb{Z})$ generated by translations by purely imaginary Lipschitz
  integers and the inversion given by $T(\mathbf{q})={\mathbf{q}}^{-1}=\overline{\mathbf{q}}/|\mathbf{q}|^2$; furthermore, 
  concrete (finite) representations of the groups are exhibited via the
  Cayley graphs. Secondly we define the group generated by translations by the imaginary parts of
  Lipschitz integers, the inversion $T$ and a finite group associated to the Hurwitz units. We give a thorough description of its properties and the corresponding orbifolds.
  
  \noi Furthermore, using some results due to Ratcliffe and
  Tschantz, we describe several Selberg covers of the orbifolds which
  turn out to be the complements of links of knotted 2-tori in the
  4-sphere. These are 4-dimensional noncompact hyperbolic manifolds
  with finite volume.  These results, beside its own geometrical
  importance and interest, have great consequences in many other
  research subjects such as Number Theory and Representation
  Theory. We also give a generalization of Iwasawa decomposition in
  the quaternionic case, a geometric description of the fundamental
  domains for the actions of the quaternionic modular groups and a
  detailed analysis of the topology and of the isotropy groups of the
  singularities of the orbifolds introduced. 
  \noi We study the corresponding modular Lipschitz and Hurwitz groups in the Lorentz-Minkowski model
  and the corresponding Iwasawa decomposition. We describe a congruence group and a principal congruence subgroup
  which both consist of Lorentz matrices in $SO_+(4,1)$ with integer coefficients. 
  Finally, using the notion of {\em Dieudonn\'e determinant} we describe the Poincar\'e extensions of the different modular groups to hyperbolic 5-spaces and
  describe the corresponding honeycombs. To make this paper self-contained we also provide 
a basic introduction to the geometric properties of orbifolds in the Appendix.

\section{The Quaternionic hyperbolic 4-space $\mathbf{H}^{1}_{\mathbb{H}}$ and its isometries.}

\subsection{Group of Quaternionic  Automorphisms}

\noi Let $GL(2,\mathbb{H})$ denote the group of $2\times 2$
invertible\footnote{By this we mean that a $2\times 2$ quaternionic
  matrix $A$ has a right and left inverse; in \cite{BisGen} it is
  shown that this is equivalent for $A$ to have non zero Dieudonn\'e
  determinant.}  matrices with entries in the quaternions:
\[\mathbb{H}=\{x_0+x_1\mathbf{i}+x_2\mathbf{j}+x_3\mathbf{k}\,  :\,
x_n\in\mathbb{R},\, n=0,1,2,3,\mathbf{i}^2=\mathbf{j}^2=\mathbf{k}^2=-1, \mathbf{ij=-ji=k}\}.
\]  
\noi If
$\mathbf{q}=x_0+x_1\mathbf{i}+x_2\mathbf{j}+x_3\mathbf{k}\in\mathbb{H}$ then $\overline
{\mathbf{q}}:=x_0-x_1\mathbf{i}-x_2\mathbf{j}-x_3\mathbf{k}\in\mathbb{H}$ 
and $|\mathbf{q}|^2:=\mathbf{q}\overline{\mathbf{q}}\in\mathbb{R}^+$.
We notice here that the multiplication in $\mathbb{H}$ is not commutative.

\begin{definition}
\noi A {\em Lipschitz quaternion} (or Lipschitz integer) is a quaternion whose
 components are all integers. The ring of all Lipschitz quaternions $\mathbb{H}(\mathbb{Z})$ is the subset of quaternions with integer coefficients:

 $$  \mathbb{H}(\mathbb{Z}) := \left\{a+b\mathbf{i}+c\mathbf{j}+d\mathbf{k} \in \mathbb{H} \ :\ a,b,c,d \in \mathbb{Z}\right\}$$
\end{definition}
\noi This is a subring of the  ring of {\em Hurwitz quaternions}: 

$$ \mathbb{H}ur(\mathbb{Z}) := \left\{a+b\mathbf{i}+c\mathbf{j}+d\mathbf{k} \in \mathbb{H}\ :\ a,b,c,d \in \mathbb{Z}
 \;\mbox{ or }\, a,b,c,d \in \mathbb{Z} + \dfrac{1}{2}\right\}.$$ 
Indeed it can be proven that $\mathbb{H}ur(\mathbb{Z})$ is closed under quaternion multiplication and
 addition, which makes it a subring of the ring of all quaternions
 $\mathbb{H}$.

\noi As a group, $\mathbb{H}ur(\mathbb{Z})$ is free abelian with generators ${1/2(1+\mathbf{i+j+k), i,
  j, k}}$. Therefore $\mathbb{H}ur(\mathbb{Z})$ forms a lattice in $\mathbb{R}^4$. This lattice is known as
the $\mathcal{F}_4$ lattice since it is the root lattice of the semisimple Lie
algebra  $\mathcal{F}_4$. The Lipschitz quaternions $\mathbb{H}(\mathbb{Z})$ form an index 2 sublattice of
$\mathbb{H}ur(\mathbb{Z})$ and it is a subring of the ring of quaternions.

\noi 
 \begin{definition}
For any $A=\begin{pmatrix}
a&b\\
c&d
\end{pmatrix}
\in GL(2,\mathbb{H})$, the associated real analytic function

$$F_A:\mathbb{H}\cup \{\infty\}\to \mathbb{H}\cup \{\infty\}$$
defined by 
\begin{equation}
F_A(\mathbf{q})=(a\mathbf{q}+b)\cdot{(c\mathbf{q}+d)^{-1}}
\end{equation}
is called the $\mathrm{linear \ fractional \ transformation}$ associated to $A$. We set $F_A(\infty)=\infty$ if $c=0$, 
$F_A(\infty)=ac^{-1} $ if $c\neq 0$ and $F_A(-c^{-1}d)=\infty$. 

\noi Let $\mathbb{F}:=\{F_A \quad{\mathrm with\ } A {\ \mathrm in\ } GL(2,\mathbb{H})\}$ the set of linear  fractional  transformations.

\end{definition}

\noi  Since $\mathbb{H}\times\mathbb{H}=\mathbb H^2:= \big \{(\mathbf{q}_0,\mathbf{q}_1)\ :\ \mathbf{q}_0,\,\mathbf{q}_1 \in \mathbb H \big \}$ as a real vector space is
$\mathbb{R}^8$, the group $GL(2,\mathbb{H})$ can be thought as a
subgroup of $GL(8,\mathbb{R})$. Using this identification we define:

\begin{definition} $SL(2,\mathbb{H}):=SL(8,\mathbb{R})\cap{GL(2,\mathbb{H})}$ and 
$$PSL(2,\mathbb{H}):=(SL(8,\mathbb{R})\cap{GL(2,\mathbb{H})})/\{\pm \mathcal{I}\},$$
where $\mathcal{I }$ denotes the identity matrix. 
\end{definition}

\noi The next result can be found in  \cite{BisGen}:
\begin{theorem}\label{Thphi} 
The set $\mathbb{F}$ is a group with
respect to the composition operation and the map

$$\Phi: GL(2,\mathbb{H})\to \mathbb{F}$$ 
defined as $\Phi(A)=F_A$ is a surjective group antihomomorphism with
$\ker(\Phi)=\{t\mathcal{I }\ :\ t\in\mathbb{R}\setminus\{0\}\}$.
Furthermore, the restriction of $\Phi$ to the special linear group 
$SL(2,\mathbb{H})$ is still surjective and has kernel $\{\pm\mathcal{I}\}$.

\end{theorem}

\subsection{Isometries in the disk model of the hyperbolic 4-space.}\label{ISOD}
\noi Let $\mathbf{B}$ denote the open unit ball in $\mathbb{H}$ and
$\mathcal{M}_{\mathbf{B}}$ the 
set of linear  (\textit{M\"obius})
transformations that leave invariant $\mathbf{B}$. 
This is the set
$$\mathcal{M}_{\mathbf{B}}:=\{F\in\mathbb{F}\ :
F(\mathbf{B})=\mathbf{B}\}.$$ 
In \cite{BisGen} there is an interesting characterization of these
transformations 

\begin{theorem}
Given $A\in GL(2,\mathbb{H})$, then the linear fractional transformation
$F_A\in\mathcal{M}_{\mathbf{B}}$ if and only if there
exist $u,v\in\partial \mathbf{B}$, $q_0\in \mathbf{B}$  (i.e., $|u|=|v|=1$ and $|q_0|<1$) such that
\begin{equation}
F_A(\mathbf{q})=v(\mathbf{q}-q_0)(1-\overline{q_0}\mathbf{q})^{-1}u
\end{equation}
for $\mathbf{q}\in\mathbf{B}$. In particular, the antihomomorphism $\Phi$
defined in Theorem \ref{Thphi} can be restricted to a surjective 
group antihomomorphism $\Phi:Sp(1,1)\to\mathcal{M}_{\mathbf{B}}$ whose kernel is 
$\{\pm\mathcal{I}\}$.
\end{theorem}

\begin{proof} We give a slightly different proof from the one  in
  \cite{BisGen}; our proof is simply based on the evaluation of real
  dimensions of the parameters involved in the description of $F_A$.  
A simple calculation implies that when $|\mathbf{q}|=1$ one has $|F_A(\mathbf{q})|=1$
and since $|q_0|<1$, $F_A$ preserves $\mathbf B$.  The elements of the
form (2) are parametrized by $(u,v, {q}_0)$ with $|u|=|v|=1$, and
$|q_0|<1$, i.e.  as a manifold ${\mathcal M}_{\mathbf B}$ is
diffeomorphic to $\mathbb S^3\times \mathbb S^3\times \mathbf B$ which
depends on 10 real parameters. On the other hand the group of isometries of
$\mathbf B$ with the hyperbolic metric is diffeomorphic to the
10-dimensional manifold ${SO(4)\times\mathbf B}=\mathbb
S^3\times{SO(3)\times\mathbf B}$ since it acts simply and transitively
on the frame bundle of $\mathbf B$ which is ${SO(4)\times\mathbf B}$. The kernel
of $\Phi$ corresponds to the triples $(1,1,0)$ and $(-1,-1,0)$.
\end{proof} 

\noi As observed,  $\mathcal{M}_{\mathbf{B}}$
depends on 10 real parameters. 
The compactification $\widehat{\mathbb H}:= \mathbb{H}\cup \{\infty\}$  of $\mathbb{H}$
can be identified with $\mathbb{S}^4$ via the stereographic projection.  The elements of $\mathcal{M}_{\mathbf{B}}$
act conformally on the 4-sphere with respect to the standard metric
and they also preserve orientation. Therefore we conclude that

$$\mathcal{M}_{\mathbf{B}}\subset{Conf_+(\mathbb{S}^4
)},$$
where $Conf_+(\mathbb{S}^4)$ is the group of conformal and orientation-preserving 
diffeomorphisms of the 4-sphere $\mathbb{S}^{4}$.
 As a manifold, $Conf_+(\mathbb S^4)$ is diffeomorphic to
 $SO(5) \times
\mathbf{H}^5_{\R}$ with
$\mathbf{H}^5_{\R}
=\{(x_1,x_2,x_3,x_4,x_5)\in\mathbb{R}^5\ :\ x_1>0\}$, 
so  $Conf_+(\mathbb S^4)$ has real  dimension 15.
Let us give a different description of this group.  We recall that  $\mathbb S^4$ 
can be thought of as being the projective quaternionic line
 $\mathbf{P}_{\mathbb H}^1 \cong {\mathbb S}^4$. This is the space of {\it right
quaternionic lines}  in $\mathbb H^2$, {\it i.e.},  subspaces 
 of the form 
$$\mathcal{L}_{\mathbf{q}} := \{\mathbf{q}\lambda \,: \, \lambda \in \mathbb H\}\,\,\,, \, \,\, 
\mathbf{q} \, \in \,\mathbb H^2\setminus\{(0,0)\}.$$

\noi We recall that $\mathbb H^2$ is a right module over $ \mathbb H$ and the action 
of $GL(2,\mathbb H)$ on $\mathbb H^2$ commutes with multiplication on
the right, i.e. for every $\lambda \in \mathbb H$ and $A \in GL(2,\mathbb H)$ one has,
$$
A R_{\lambda} = R_{\lambda} {A}
$$
where $R_{\lambda}$ is the multiplication on the right by
 $\lambda\in\mathbb{H}$.  Thus $GL(2,\mathbb H)$  carries right quaternionic 
lines into right quaternionic lines, and in this way
an action of $GL(2,\mathbb H)$  on $\mathbf{P}_{\mathbb H}^1$ is defined.
 
\noi Any $F_A\in\mathbb{F}$ lifts canonically to an automorphism 
$\widetilde {F_A}$ of $\mathbf{P}^3_{\mathbb{C}}$, the complex projective 3-space and the map $\Psi: F_A\mapsto \widetilde {F_A}$ 
injects $\mathbb{F}$ into the complex projective group $PSL(4,\mathbb{C}):=
SL(4, \mathbb{C})/\{\pm \mathcal{I}\}$.
Hence (see \cite{A1, A2}), we conclude that

$$\mathcal{M}_{\mathbf{B}}\subset {PSL(2,\mathbb{H})}:=GL(2,\mathbb H)/\{t{\mathcal{I}}, \, t\neq0\}
\simeq Conf_+(\mathbb{S}^4
).$$

\noi We also notice that 
the  map 
$$\mathbf{q}\mapsto \overline {\mathbf{q}}$$ 
is conformal and maps the unit ball $\mathbf{B}$ onto itself but it
reverses the orientation of $\mathbf{B}$.

\noi We consider now the group $\mathbb{M}_{\mathbf{B}}$ of {\em extended
M\"obius transformations} defined as the union of all the 
M\"obius transformations $\mathcal{M}_{\mathbf{B}}$ and all maps $\phi$ which reverse the orientation obtained
as $\phi(\mathbf{q})=F_A(\overline{\mathbf{q}})=(a\overline{\mathbf{q}}+b)(c\overline{\mathbf{q}}+d)^{-1}$ for some  $F_A\in\mathcal{M}_{\mathbf{B}}$.

\noi In \cite{BisGen} the {\em Poincar\'e distance} $d_{\mathbf{B}}$ is defined in
$\mathbf{B}$ in terms of the quaternionic cross--ratio and this distance coincides with the standard hyperbolic metric $\frac{4|d\mathbf{q}|^2}{(1-|\mathbf{q}|^2)^2}$; furthermore, in \cite{BisGen} it is also
proved the following:
\begin{proposition}
The Poincar\'e distance $d_{\mathbf{B}}$ in
$\mathbf{B}$ is invariant under the action of 
the group $\mathbb{M}_{\mathbf{B}}$ of  extended 
M\"obius transformations. 

\noi In other words: 

$$\mathbb{M}_{\mathbf{B}}= Isom_{ d_{\mathbf{B}}}(\mathbf{B}).$$ 
\end{proposition}
\subsection{Isometries in the half-space model of the hyperbolic 4-space $\mathbf{H}^{1}_{\mathbb{H}}$.}

\noi Let $\mathbf{H}^{1}_{\mathbb{H}}$ be the half-space
model of the \emph{one-dimensional quaternionic hyperbolic space}

$$
\mathbf{H}^{1}_{\mathbb{H}}:=\{\mathbf{q}\in\mathbb{H}\ :\ \Re({\mathbf{q}})>0\}.
$$

\noi Since the unit ball $\mathbf{B}$ in
$\mathbb{H}$ can be identified with the lower hemisphere of $\mathbb{S}^4
$ and
any transformation $F_A\in\mathcal{M}_{\mathbf{B}}$ is conformal and preserves
orientation, we conclude that  (see also \cite{A1, A2}) 

$$\mathcal{M}_{\mathbf{B}}\simeq Conf_+(\mathbf{H}^{1}_{\mathbb{H}})$$
where $Conf_+(\mathbf{H}^{1}_{\mathbb{H}})$ represents the group of
conformal diffeomorphisms orientation--preserving of the half-space
model $\mathbf{H}^{1}_{\mathbb{H}}$.

\begin{remark}
The choice of describing $\mathbf{H}^{1}_{\mathbb{H}}$ using the
quaternionic notation is motivated by the main point of view of this
paper;  however this set precisely is isometric to the hyperbolic
real space (in four dimensions), namely
\[\mathbf{H}^{1}_{\mathbb{H}} \cong\mathbf{H}^{4}_{\mathbb{R}}=\{(x_0,x_1,x_2,x_3)\in\mathbb{R}^4 :\ x_0>0\}\]
with the element  of hyperbolic metric given by $(ds)^2=\frac{(dx_0)^2+(dx_1)^2+(dx_2)^2+(dx_3)^2}{x_0^2}$ where $s$ measures length along a parametrized curve. Even though the (natural) algebraic structures carried by the two sets are deeply different.
\end{remark}

\noi Via the Cayley transformation $\Psi:\mathbf{B}\to \mathbf{H}^{1}_{\mathbb{H}}$
defined as $\Psi(\mathbf{q})=(1+\mathbf{q})(1-\mathbf{q})^{-1}$ one
can show explicitely (see \cite{BisGen}) that the unit ball
$\mathbf{B}$ of $\mathbb{H}$ is diffeomorphic to
$\mathbf{H}^{1}_{\mathbb{H}}$ and introduce a Poincar\'e distance in
$\mathbf{H}^{1}_{\mathbb{H}}$ in such a way that the Cayley
transformation $\Psi:\mathbf{B}\to \mathbf{H}^{1}_{\mathbb{H}}$ is an
isometry; moreover the Poincar\'e distance in
$\mathbf{H}^{1}_{\mathbb{H}}$ is invariant under the action of the
group
$\Psi\mathbb{M}_{\mathbf{B}} \Psi^{-1}$ which is denoted by $\mathbb{M}_{\mathbf{H}^{1}_{\mathbb{H}}}$.

\noi Finally we recall the conditions found by Bisi and Gentili (see
\cite{BisGen} again):

\begin{proposition} \label{GB} $\mathbf{[Conditions \, (BG)]}$
The subgroup of $PSL(2,{\mathbb{H}})$ whose elements are
associated to invertible linear fractional transformations which preserve $\mathbf{H}^{1}_{\mathbb{H}}$ can be characterized as the group induced by matrices which satisfy
one of the following (equivalent) conditions:

\begin{equation*}\label{BG}
\mathrm{(BG)}\quad\left\{\begin{array}{l}
\left\{A=\begin{pmatrix}a&b\\c&d\end{pmatrix} a,b,c,d\in\mathbb{H}\ :\  \ \overline{A}^tKA=K\right\}
\mathrm{with}\
K=\begin{pmatrix}
0&1\\
1&0
\end{pmatrix}\\\\
 \left\{A=\begin{pmatrix}a&b\\c&d\end{pmatrix} a,b,c,d\in\mathbb{H}\ :\  \ \Re(a\overline{c})=0,\ \Re(b\overline{d})=0,\ \overline{b}c+\overline{d}a=1
\right\}\\
\\\left\{A=\begin{pmatrix}a&b\\c&d\end{pmatrix} a,b,c,d\in\mathbb{H}\ :\  \ \Re(c\overline{d})=0,\ \Re(a\overline{b})=0,\ a\overline{d}+b\overline{c}=1
\right\}.
\end{array}
\right.
\end{equation*}
\end{proposition}

\begin{definition}
The group of invertible linear transformations satisfying (BG) conditions consists of orientation preserving hyperbolic
isometries of $\mathbf{H}^1_{\mathbb{H}}$, therefore it be 
denoted $Isom_+(\mathbf{H}^1_{\mathbb{H}})$; however,
in analogy with the previous notations, we will also called it the \it{M}\"o\it{bius group of} $\mathbf{H}^{1}_{\mathbb{H}}$ and we denote alternatively this group
as $\mathcal{M}_{\mathbf{H}^{1}_{\mathbb{H}}}:=\Psi\mathcal{M}_{\mathbf
  B}\Psi^{-1}$. 
  
  \noi Moreover, we have $Isom(\mathbf{H}^1_{\mathbb{H}}) =\mathbb{M}_{\mathbf{H}^{1}_{\mathbb{H}}}$, where $Isom(\mathbf{H}^1_{\mathbb{H}})$ is the full group of isometries of $\mathbf{H}^1_{\mathbb{H}}$. \end{definition}

\begin{remark} The group $\mathcal{M}_{\mathbf{H}^{1}_{\mathbb{H}}}$ acts by orientation-preserving conformal transformations on the
sphere at infinity of the hyperbolic 4-space defined as follows 
$$
\mathbb{S}^3=\partial\mathbf{H}^{1}_{\mathbb{H}}:=\{\mathbf{q}\in\mathbb
H\,\,:\,\Re(\mathbf{q})=0\}\cup\{\infty\}.$$
\noi In other words $\mathcal{M}_{\mathbf{H}^{1}_{\mathbb{H}}} \cong{Conf_+({\mathbb{S}}^3})$.
\end{remark}

\noi Among the elements of $\mathcal{M}_{\mathbf{H}^{1}_{\mathbb{H}}}$ and 
of $\mathbb{M}_{\mathbf{H}^{1}_{\mathbb{H}}}$,  translations,
(hyperbolic) rotations and inversion
 will be of fundamental importance in the sequel and will be
 specifically studied in a dedicated section.
 
 \begin{remark}  Since the only ambiguity is a change of sign, by abuse of language, in all that follows we will sometime identify a matrix with quaternionic entries satisfying (BG) conditions with the induced  M$\ddot{o}$bius transformation.
 \end{remark}

\subsection{The  affine subgroup $ \mathcal{A}(\mathbb{H})$ of the isometries of $\mathbf{H}^{1}_{\mathbb{H}}$.}

\vskip 0.5  cm
Consider now the \emph{affine subgroup} $\mathcal{A}(\mathbb{H})$ of $PSL(2, \mathbb{H})$  consisting of transformations which are induced by matrices of the form

$$\begin{pmatrix}
\lambda{a}&b\\
0&\lambda^{-1}{a}
\end{pmatrix} \quad \textit{i.e.} \quad 
\mathbf{q}\mapsto ((\lambda{a}){\mathbf{q}}+b)(\lambda^{-1}{a})^{-1}
$$

\noi with $|a|=1, \, \lambda>0 $ and $\Re(\overline{b}{a})=0$. 
Such matrices  satisfy (BG) conditions of Proposition \ref{GB} and
therefore are in $\mathcal{M}_{\mathbf{H}^{1}_{\mathbb{H}}}$.  The group  $ \mathcal{A}(\mathbb{H})$ is the maximal subgroup of  $\mathcal{M}_{\mathbf{H}^{1}_{\mathbb{H}}}$ which fixes the point at infinity.

\noi The group $ \mathcal{A}(\mathbb{H})$ is a Lie group of real dimension 7 and each
matrix in $ \mathcal{A}(\mathbb{H})$ acts as a conformal transformation on the
hyperplane at infinity $\partial
\mathbf{H}^{1}_{\mathbb{H}}.$

\noi Therefore $ \mathcal{A}(\mathbb{H})$ is the
group  of conformal and orientation preserving transformation 
acting on the space of pure imaginary quaternions at
infinity which can be identified with $\R^3$ so that this group is isomorphic to the conformal group
$Conf_+(\mathbb{R}^3)$.

\subsection{Iwasawa decomposition of the isometries of $\mathbf{H}^{1}_{\mathbb{H}}$.}

\noi In analogy with the complex and real case, 
we can state a generalization of Iwasawa
 decomposition for any element  of $\mathcal{M}_{\mathbf{H}^{1}_{\mathbb{H}}}$   as follows

\begin{proposition} \label{Iwasawa} Every element of  $\mathcal{M}_{\mathbf{H}^{1}_{\mathbb{H}}}$ i.e., elements in $PSL(2,\mathbb{H})$ which
  satisfies (BG) conditions
 and which is represented by the matrix $M=\begin{pmatrix}
a&b\\
c&d 
\end{pmatrix}$ can be written in an unique way as follows 

\begin{equation}\label{Iwa} 
M=\begin{pmatrix}
\lambda&0\\
0&{\lambda}^{-1} 
\end{pmatrix} \begin{pmatrix}
1&\omega\\
0&1 
\end{pmatrix} 
\begin{pmatrix}
\alpha&\beta\\
\beta&\alpha 
\end{pmatrix},
\end{equation}
with $\lambda>0$, $\Re{(\omega)}=0$,
$|\alpha|^2+|\beta|^2=1$ and $\Re{(\alpha\overline\beta)}=0$. 
\end{proposition}
\begin{proof}
We'll give  explicit expressions for  $\alpha,\ \beta,\ \lambda $ and $\omega$ in terms of $a,b,c,$ and $d$.
Indeed, from direct computations, one easily obtains that
$\lambda d=\alpha$ and $\lambda c=\beta$;  therefore, from the equations 
$ a=\lambda^2(d+\omega c) \quad b=\lambda^2(c+\omega d)$ 
it is a matter of calculations to conclude that
$$\lambda=\dfrac{1}{\sqrt{|c|^2+|d|^2}}\quad\mathrm{and}\quad \omega=a\overline{c}+b\overline{d}.$$
Therefore, from  (BG) conditions of Proposition \ref{GB}, 
it follows that $\Re{(\omega)}=0$ and $\Re{(\alpha\overline\beta)}=0$.
\end{proof}

\subsection{Isotropy subgroup of the isometries of $\Hy$ which fixed one point.}

\noi We notice
that the set $\mathcal{M}_{\mathbf{H}^{1}_{\mathbb{H}}}$ of matrices inducing elements in $PSL(2,\mathbb{H})$ satisfying (BG) conditions has real
dimension ten.

\noi Let $\mathcal{K}:=\Big\{ \begin{pmatrix}
\alpha&\beta\\
\beta&\alpha 
\end{pmatrix} \in \mathcal{M}_{\mathbf{H}^{1}_{\mathbb{H}}} \Big\}$ be the subgroup of symmetric matrices in $\mathcal{M}_{\mathbf{H}^{1}_{\mathbb{H}}} $. For the matrix $\begin{pmatrix}
\alpha&\beta\\
\beta&\alpha 
\end{pmatrix}$ the conditions  $|\alpha|^2+|\beta|^2=1$ and $\Re{(\alpha\overline\beta)}=0$ are equivalent to 
(BG) conditions in Proposition \ref{GB}:
$$
\overline{\begin{pmatrix}
\alpha&\beta\\
\beta&\alpha 
\end{pmatrix}}\ \begin{pmatrix}
0&1\\
1&0 
\end{pmatrix} \begin{pmatrix}
\alpha&\beta\\
\beta&\alpha 
\end{pmatrix}= \begin{pmatrix}
0&1\\
1&0 
\end{pmatrix}. 
$$

\noi We have the following:

\begin{proposition}
 The group  $\mathcal{K}$ is a compact Lie group isomorphic to the special orthogonal group $SO(4)$.
Indeed this group is precisely the isotropy subgroup at $1\in\mathbf H_{\mathbb H}^1$ of the action of $PSL(2,\mathbb{H})$  by orientation preserving
isometries on $\mathbf {H}_{\mathbb H}^1$. Moreover the group $\mathcal{K}$ is a maximal compact subgroup of $\mathcal{M}_{\mathbf{H}^{1}_{\mathbb{H}}}$. 
\end{proposition}
\begin{proof}
\noi Let $\begin{pmatrix}
a&b\\
c&d 
\end{pmatrix}$ be a matrix satisfying (BG) conditions and fixing 1. Then $a+b=c+d$ or $a-d=c-b$ and 
$$
|a-d|^2=(c-b)(\overline{a}-\overline{d})=c\overline{a}-c\overline{d}-b\overline{a}+b\overline{d}
$$

$$
|c-b|^2=(a-d)(\overline{c}-\overline{b})=a\overline{c}-a\overline{b}-d\overline{c}+d\overline{b}.
$$

\noi Then (BG) conditions imply:  
$$
|a-d|^2+|c-b|^2=(c\overline{a}+a\overline{c})+(b\overline{d}+d\overline{b})-((c\overline{d}+d\overline{c})+(b\overline{a}+a\overline{b}))=0.
$$  
\noi Finally, $|a-d|^2=|c-b|^2=0$ implies that $a=d$ and $c=b$ and again (BG) conditions imply $\Re (\overline{a}b)=0$ and $|a|^2+|b|^2=1$. 
\end{proof}

\noi Let $\mathcal{D}:=\Big\{ \begin{pmatrix}
\alpha&0\\
0&\alpha 
\end{pmatrix} \in \mathcal{M}_{\mathbf{H}^{1}_{\mathbb{H}}} \Big\}$ be the subgroup of $\mathcal{K}$ whose elements are diagonal matrices in
$\mathcal{M}_{\mathbf{H}_\mathbb{H}}^1$. Then the  (BG) conditions imply that $|\alpha|=1$. The action at infinity is given by $\q\mapsto{\alpha\q\bar\alpha}$, which is the usual action of $SO(3)$ on the purely imaginary quaternions.

\noi Therefore:

\begin{corollary} 
The group $\mathcal{D}$ is isomorphic to $SO(3)$. \end{corollary} 

\section{The quaternionic modular groups.}
\noi In this section we investigate a class of linear transformations 
which will play a crucial role in the definition of the quaternionic
modular groups.
\subsection{Quaternionic Translations}

\noi  We recall that a translation $\tau_{\omega}: \mathbf{H}^{1}_{\mathbb{H}} \to
\mathbf{H}^{1}_{\mathbb{H}}$  defined as  $\mathbf{q}\mapsto \mathbf{q}+\omega$ 
is a  hyperbolic isometry in $\mathbf{H}^{1}_{\mathbb{H}}$ 
if it is a transformation associated to the matrix $\left(\begin{array}{cc} 1 &
  \omega \\0 & 1\end{array}\right) \in\mathcal{M}_{\mathbf{H}^{1}_{\mathbb{H}}}$, i.e. if it is such that
  $\Re(\omega)=0$. 
  
  \noi In what follows  we consider translations where $\omega$ is the imaginary part
  of a Lipschitz or Hurwitz integer. We remark that the imaginary part of a Lipschitz integer is still a 
  Lipschitz integer but the imaginary part of a Hurwitz integer is not necessarily a  Hurwitz integer.

\begin{definition} An {\em imaginary Lipschitz quaternion} (or imaginary Lipschitz integer) is the imaginary part of a Lipschitz quaternion, a quaternion whose real part is 0 and the others components are all integers. The set of all imaginary Lipschitz quaternions is

 $$\Im \mathbb{H}(\mathbb{Z}) = \left\{b\mathbf{i}+c\mathbf{j}+d\mathbf{k} \in \mathbb{H} \ :\ b,c,d \in \mathbb{Z}\right\}.$$

\end{definition}

\noi We denote by $\mathcal{T}_{\Im \mathbb{H}(\mathbb{Z})}$ the abelian group of translations 
by the imaginary Lipschitz group $\Im \mathbb{H}(\mathbb{Z})$, i.e. such that
$\mathbf{q}\mapsto \mathbf{q}+\omega$, $\omega=n_2\mathbf{i}+n_3\mathbf{j}+n_4\mathbf{k}$
where the $n$'s are all integers; 
equivalently $\mathbf{q}\mapsto \mathbf{q}+\omega$ belongs to $\mathcal{T}_{\Im \mathbb{H}(\mathbb{Z})}$
if and only if $\omega\in \Im \mathbb{H}(\mathbb{Z})$. Thus the elements of  $\mathcal{T}_{\Im \mathbb{H}(\mathbb{Z})}$
are given by matrices of the form:
\[\left(\begin{array}{cc} 1 & \omega \\0 & 1\end{array}\right)\,\,\, \,  \rm{with} \,\, \Re(\omega)=0 \]

\noi The group  $\mathcal{T}_{\Im \mathbb{H}(\mathbb{Z})}$ acts freely on $ \mathbf{H}^{1}_{\mathbb{H}}$ as a representation of the abelian group with 3 free generators $\Z \oplus \Z \oplus \Z$. A fundamental domain is the following set $\{\q=x_0+x_1\mathbf{i}+x_2\mathbf{j}+x_3\mathbf{k}\in \mathbf{H}^{1}_{\mathbb{H}}\,\,:
\,\,|x_n|\leq1/2,\,n=1,\cdots,3\}.$ This set is referred as the \emph{chimney} in figure 1 in section 5. It has two ends, one of finite volume which is asymptotic at the point at infinity. The other end has infinite volume, it is called a \emph{hyperbolic trumpet}.   
 
\subsection{Inversion} 

\noi Let us consider now  \[T(\mathbf{q})=\mathbf{q}^{-1}=\dfrac{\overline
  {\mathbf{q}}}{|\mathbf{q}|^2}.\] 
Clearly $T$ is
a linear fractional tranformation of $\mathbf{H}^{1}_{\mathbb{H}}$
and its representative
 matrix is $\left(\begin{array}{ccc}0 & 1 \\1 &
  0 \end{array}\right)$.
  
\noi The only fixed point of $T$ in $\mathbf{H}^{1}_{\mathbb{H}}$ is 1
since the other fixed point of $T$ in $\mathbb{H}$ is $-1$ which is
not in $\mathbf{H}^{1}_{\mathbb{H}}$. We also notice here that in the topological closure of
$\mathbf{H}^{1}_{\mathbb{H}}$ (denoted by $\overline{\mathbf{H}^{1}_{\mathbb{H}}}$) the points 0
 and $\infty$ are periodic (of period 2) for $T$. Furthermore 
$T$ is an isometric involution \footnote {In the following sense; $T$ sends every point of a hyperbolic  geodesic
  parametrized by arc length $\gamma(s)$,
 passing through 1 at time 0 (i.e. such that $\gamma(0)=1$), to its opposite
$\gamma(-s)$. In other words, $T$ is a hyperbolic symmetry around 1.}  of $\mathbf{H}^{1}_{\mathbb{H}}$ because it satisfies 
(BG) conditions of Proposition \ref{GB}. 
In particular $T$ is an inversion on $\mathbb{S}^3$ which becomes the
antipodal map on any copy of $\mathbb{S}^2$  obtained as intersection of 
$\mathbb{S}^3$ with a plane perpendicular to the line passing through
0 and 1.
Finally, this
isometry $T$ leaves invariant 
the hemisphere (which is a hyperbolic 3-dimensional hyperplane)
$\Pi:=\{\mathbf{q}\in \mathbf{H}^{1}_{\mathbb{H}}\,\,\,\, :\,\,\,\, |\mathbf{q}|=1\}.$
Each point of $\Pi$  different from 1 (which is fixed by $T$) is a periodic point of $T$ of period 2.

\begin{definition} Let 
$$\mathcal{C}=\{\q=x_0+x_1\mathbf{i}+x_2\mathbf{j}+x_3\mathbf{k}\in \mathbf{H}^{1}_{\mathbb{H}}\,\,:
\,\,|\q|=1,\,|x_n|\leq1/2,\,n=1,\cdots,3 \}.
$$
\end{definition}
\noi Then,  $\mathcal{C}$ is a regular hyperbolic cube in $\Pi$. The points of the form
$ \frac{1}{2}\pm \frac{1}{2} \mathbf{i}\pm \frac{1}{2} \mathbf{j}\pm \frac{1}{2} \mathbf{k},$ are the vertices of $\mathcal{C}$ and in particular they are periodic of period 2 for $T$. These eight points
are Hurwitz units (but not Lipschitz units).

 \subsection{Composition of translations and inversion} 

\noi We observe that if $\tau_{\omega}(\mathbf{q}):=\mathbf{q}+\omega$, $\omega\in\mathbb{H}$,
then $L_{\omega}:=\tau_{\omega} T 
$ has as corresponding matrix
$$\begin{pmatrix}
\omega&1\\
1&0
\end{pmatrix}=\begin{pmatrix}
1&\omega\\
0&1
\end{pmatrix} \begin{pmatrix}
0&1\\
1&0
\end{pmatrix};
$$
\noi similarly $R_{\omega}:=T \tau_{\omega}$ has as corresponding matrix
$$\begin{pmatrix}
0&1\\
1&\omega
\end{pmatrix}=\begin{pmatrix}
0&1\\
1&0
\end{pmatrix} \begin{pmatrix}
1&\omega\\
0&1
\end{pmatrix}.$$ 
 
\noi Therefore $R_{\omega}$ is represented by interchanging the elements on the diagonal 
of the matrix which represents $L_{\omega}$.
In the following table, we list the matrices associated to iterates of 
$L_\omega=\tau_\omega   T$ with  suitable choices of $\omega$.  \emph{In the table all possible choices of signs are allowed}.

\vskip.5cm

\begin{center}
\noi \begin{tabular}{|c|c|c|}\hline\label {table}

$\omega=\pm \mathbf{i} $   or $\omega= \pm \mathbf{j}$ or $\omega=\pm \mathbf{k}$ &  $\omega=\pm \mathbf{i \pm j}$ or $\omega=\pm \mathbf{i
  \pm k}$&  
$\omega=\mathbf{\pm i \pm j \pm k}$ \\
&or $\omega=\mathbf{\pm j \pm k}$ &  \\ 
 $\omega^2=-1$ &  $\omega^2=-2$ &  $\omega^2=-3$
\\\hline  
& & \\
 $L_\omega^2=\begin{pmatrix}
0&\omega\\
\omega&1
\end{pmatrix}$&  
$L_\omega^2=\begin{pmatrix}
-1&\omega\\
\omega&1
\end{pmatrix}$
& $L_\omega^2=\begin{pmatrix}
-2&\omega\\
\omega&1
\end{pmatrix}$ \\  
& & \\\hline & & \\
 $L_\omega^3=\begin{pmatrix}
\omega&0\\
0&\omega
\end{pmatrix}$ & $L_\omega^3=\begin{pmatrix}
0&-1\\
-1&\omega
\end{pmatrix}$ & $L_\omega^3=\begin{pmatrix}
-\omega&-2\\
-2&\omega
\end{pmatrix}$ \\ & &  \\\hline  & &\\ $L_\omega^4=\begin{pmatrix}
-1&\omega\\
\omega&0
\end{pmatrix}$ & $L_\omega^4=\begin{pmatrix}
-1&0\\
0&-1
\end{pmatrix}$ & $L_\omega^4=\begin{pmatrix}
1&-\omega\\
-\omega&-2
\end{pmatrix}$ \\   
& & \\\hline & & \\
 $L_\omega^5=\begin{pmatrix}
0&-1\\
-1&\omega
\end{pmatrix}$ &  & $L_\omega^5=\begin{pmatrix}
0&1\\
1&-\omega
\end{pmatrix}$ \\ & & \\\hline & & \\
 $L_\omega^6=\begin{pmatrix}
-1&0\\
0&-1
\end{pmatrix}$ &  & $L_\omega^6=\begin{pmatrix}
1&0\\
0&1
\end{pmatrix}$ \\ & & \\\hline \end{tabular}

\end{center}

\vskip1cm

\noi We can see that the order of $L_\omega$ depends on
$\omega$; in particular, each of the six transformations $L_\omega$ with $\omega=\mathbf{\pm i, \pm j, \pm k}$, has
order 6 but when restricted to the plane
$S_\omega:=\{\mathbf{q}=x_1+x_{\mathbf{i}}\mathbf{i}+x_{\mathbf{j}}\mathbf{j}+x_{\mathbf{k}}\mathbf{k}
\in\mathbf{H}^{1}_{\mathbb{H}} \ : \ x_{\alpha}= 0\ \mathrm{if} \ \alpha\neq \omega, 0\}$,
with $\omega=\mathbf{i,j,k}$ has order 3.  \noi Furthermore $\mathbf{q}_0$
is a fixed point for $L_{\omega}= \tau_\omega   T$ with $\omega=0,\mathbf{\pm i,\pm
  j,\pm k}$, if and only if $\mathbf{q}_0$ is a root of $
\mathbf{q}^2-\omega \mathbf{q}-1=0$. If $\omega=0$ there is only one root in
$\mathbf{H}^{1}_{\mathbb{H}}$ (and so only one fixed point for $T$),
namely $\mathbf{q}_0=1$.  If $\omega=\mathbf{\pm i,\pm j,\pm k}$, then it
is easily verified that if $\alpha$ and $\beta$ are two roots of
$
\mathbf{q}^2-\omega \mathbf{q}-1=0$, it follows that $\Re (\alpha+\beta)=0$
or $\Re (\alpha)=-\Re(\beta)$.  Since a root of $
\mathbf{q}^2-\omega \mathbf{q}-1=0$  is $\alpha=\frac{\sqrt{3}}{2}+\frac{\omega}{2}$
($\alpha=\frac{\sqrt{3}}{2}-\frac{\omega}{2}$) any other possible root
$\beta$ of the above given equation would not sit in
$\mathbf{H}^{1}_{\mathbb{H}}$. In the same way $\mathbf{q}_0$
is a fixed point for $R_{\omega}=T  \tau_\omega$ with $\omega=0,\mathbf{\pm i,\pm
  j,\pm k}$, if and only if $\mathbf{q}_0$ is a root of $
\mathbf{q}^2+\mathbf{q}\omega-1=0$. If $\omega=0$ there is only one root in
$\mathbf{H}^{1}_{\mathbb{H}}$ (and so only one fixed point for $T$),
namely $\mathbf{q}_0=1$.  If $\omega=\mathbf{\pm i,\pm j,\pm k}$, then it
is easily verified that if $\alpha$ and $\beta$ are two roots of
$
\mathbf{q}^2+\mathbf{q}\omega-1=0$, it follows that $\Re (\alpha)=-\Re(\beta)$.  Since a root of $
\mathbf{q}^2+\mathbf{q}\omega-1=0$ is $\alpha=\frac{\sqrt{3}}{2}-\frac{\omega}{2}$ any other possible root
$\beta$ of the above given equation would not sit in
$\mathbf{H}^{1}_{\mathbb{H}}$.

\noi Briefly, the only fixed point of $L_{\omega}$ in $\mathbf{H}^{1}_{\mathbb{H}}$ is $\frac{\sqrt{3}}{2}+\frac{\omega}{2}$ and the only fixed point of $R_{\omega}$ is $\frac{\sqrt{3}}{2}-\frac{\omega}{2}$. 

\subsection{The Lipschitz Quaternionic modular group $PSL(2,\mathfrak{L})$.} 
\noi We are now in the position of introducing the following:

\begin{definition} 
The {\em Lipschitz quaternionic modular group} is the group
generated by the inversion $T$
and the translations $\mathcal{T}_{\Im \mathbb{H}(\mathbb{Z})}$.
 It will be denoted by $PSL(2, \mathfrak{L}) $.
\end{definition}

\begin{remark} The group $PSL(2,\mathfrak{L})$ is obviously a discrete subgroup of  $PSL(2,\mathbb{H})$ (more precisely, of $\mathcal{M}_{\mathbf{H}^{1}_{\mathbb{H}}}$). It is important to emphasize that the quaternionic
  modular group is a proper subgroup of $PSL(2, \mathbb{H}(\mathbb{Z}))$;
indeed, the subgroup  generated by (proper) translations and by the inversion $T$
 in $\mathbf{H}_\mathbb{H}^1$ has elements which are represented by matrices 
 with Lipschitz  integers as entries, but  
in general an arbitrary element in  $PSL(2, \mathbb{H}(\mathbb{Z}))$ 
 does not satisfy (BG) conditions of Proposition 1.5 and 
therefore it does not preserve $\mathbf{H}_\mathbb{H}^1$. 

\end{remark}

\vskip 0.2  cm

\subsection{Lipschitz unitary and  affine subgroups of   $PSL(2,\mathfrak{L})$.}

\noi Let $\mathfrak{L}_u$ the group (of order 8) of \textit{Lipschitz  units}  
   $$\mathfrak{L}_u:=
\{\pm 1,\ \pm \mathbf{i},\ \pm\mathbf{j},\pm\mathbf{k} :\ \mathbf{i}^2=\mathbf{j}^2=\mathbf{k}^2=\mathbf{ijk}=-1 \}.$$

\noi This group is the {\it quaternion group} which is a non-abelian group of order eight. Moreover, its elements are the 8 vertices of a 16-cell in the 3-sphere $\mathbb{S}^3$ and the 8 barycentres of the faces of its dual polytope which is a hypercube also called 8-cell.

\begin{definition}
The subgroup $\mathcal{U}(\mathfrak{L})$ of $PSL(2,\mathfrak{L})$ whose elements are
the 4 diagonal matrices
$$D_\mathbf{u}:=\begin{pmatrix}
\mathbf{u}&0\\
0&\mathbf{u}
\end{pmatrix}$$
with $\mathbf{u}$ a Lipschitz unit is called {\em Lipschitz unitary group}.
\end{definition}

\noi The  Lipschitz unitary group is isomorphic to the so called \emph{Klein group} of order 4 which is isomorphic to
$\Z/2\Z\oplus \Z/2\Z$, since $\mathbf{ij=k}$. 
Moreover, we observe that the action on  $\mathbf{H}^{1}_{\mathbb{H}}$ of the transformation
associated to 
$$\begin{pmatrix}
\mathbf{u}&0\\
0&\mathbf{u}
\end{pmatrix},
$$
where $\mathbf{u}=\mathbf{i},\mathbf{j}$ or $\mathbf{k}$ is for conjugation and sends a quaternion $\mathbf{q}\in \mathbf{H}^{1}_{\mathbb{H}}$ to $\mathbf{uqu}^{-1}$. It acts as a rotation of angle $\pi$ with axis the vertical hyperbolic 2-plane
$$
S_{\mathbf{u}}=\{ x+y\mathbf{u}\,\, :\,\, x,y \in \R, x>0 \}
.$$

\begin{definition} The Lipschitz affine subgroup (or the Lipschitz parabolic subgroup)  $\mathcal{A}(\mathfrak{L})$ is
the group generated by the unitary  group $\mathcal{U}(\mathfrak{L})$ 
and the group of translations $\mathcal{T}_{\Im \mathbb{H}(\mathbb{Z})}$. Equivalently

\begin{subequations}
\begin{align}
\mathcal{A}(\mathfrak{L}) &=\left\{\begin{pmatrix}
\mathbf{u}&\mathbf{u}b\\
0&\mathbf{u}
\end{pmatrix}
:\ \mathbf{u}\in{\mathfrak{L}_u}   ,\ \Re(b)=0  \right\}\\
&=\left\{\begin{pmatrix}
\mathbf{u}&b\mathbf{u}\\
0&\mathbf{u}
\end{pmatrix}
:\ \mathbf{u}\in{\mathfrak{L}_u}   ,\ \Re(b)=0  \right\}.
\end{align}
\end{subequations}
\end{definition}

\begin{remark} The Lipschitz affine subgroup $\mathcal{A}(\mathfrak{L})$ is the maximal Lipschitz parabolic
  subgroup of $PSL(2,\mathfrak{L})$. Moreover $\mathcal{A}(\mathfrak{L}) \subset PSL(2, \mathfrak{L}) \cap \mathcal{A}(\mathbb{H})$.
Furthermore, this subgroup leaves invariant  the horizontal horospheres $\Re(\mathbf{q})=x_0>0$ and also
the horoball  $\Re(\mathbf{q})>x_0>0$. 
\end{remark}
  
 \noi Evidently $\mathcal{A}(\mathfrak{L})$  is a subgroup  of $PSL(2,\mathfrak{L})$ and, since $\mathbf{ij=k}$, 
it is generated by hyperbolic isometries associated to the matrices

$$\begin{pmatrix}
\mathbf{i}&0\\
0&\mathbf{i}
\end{pmatrix}, \quad \begin{pmatrix}
\mathbf{j}&0\\
0&\mathbf{j}
\end{pmatrix}, \quad\begin{pmatrix}
1&\mathbf{u}\\
0&1
\end{pmatrix}
$$
where $\mathbf{u}=\mathbf{i},\mathbf{j}$ and $\mathbf{k}$. 

\noi 

\begin{remark}
In particular,  since the transformation represented by the matrix $\begin{pmatrix}
\mathbf{u}&0\\
0&\mathbf{u}
\end{pmatrix}$
is a rotation of angle $\pi$ which keeps fixed each point of the plane 
 $S_\mathbf{u}$ (the ``axis of rotation''),
 the combination of such a rotation and the inversion leads
 to a transformation represented by the matrix
$\begin{pmatrix}
0&\mathbf{u}\\
\mathbf{u}&0
\end{pmatrix}$
with $\mathbf{u}=\mathbf{i,j,k}$.
For these trasformations the plane 
 $S_\mathbf{u},$ with $\mathbf{u}=\mathbf{i,j,k}$ is invariant.
Both rotations and inversion composed with a rotation of the plane leave invariant the sphere $\Pi$ and have 1 as a fixed point. 
\end{remark}

\noi We have the following properties:

 \begin{enumerate}

   \item The inverse of a matrix $\left(\begin{array}{cc}a & b \\ 0 &
     d\end{array}\right) \in \mathcal{A}(\mathbb{H})$ is the matrix 

$$\left(\begin{array}{cc}a^{-1} &  -a^{-1}bd^{-1}  \\ 0 & d^{-1} \end{array}\right) \in \mathcal{A}(\mathbb{H}).$$
   
\item  If we consider the group $\mathfrak{L}_u$ of order 8 of Lipschitz units, then the map 
$$
\mathcal{A}(\mathfrak{L}) \to
\mathfrak{L}_u\, , \,\, \,\,\,
 \left(\begin{array}{cc}\mathbf{u} & \mathbf{u}b \\ 0 & \mathbf{u}\end{array}\right) \mapsto \mathbf{u}$$
  is an epimorphism whose kernel is:
$$
\mathcal{T}_{\Im \mathbb{H}(\mathbb{Z})}=\left\{ \left(\begin{array}{cc}1 & \omega \\ 0 & 1\end{array}\right)\, :\,\, \omega \in \Im\mathbb{H}(\mathbb{Z})  \right\}
$$

  \end{enumerate}

\noi Thus we have the exact sequence 
$$
0\longrightarrow{ \mathcal{T}_{\Im \mathbb{H}(\mathbb{Z})} }\longrightarrow{ \mathcal{A} }(\mathfrak{L})\longrightarrow\mathcal{U}(\mathfrak{L})=\Z/2\Z\oplus \Z/2\Z\longrightarrow0
$$
\noi This sequence splits and the group $\mathcal{A}(\mathfrak{L})$ is the semi-direct product of $\mathcal{T}_{\Im \mathbb{H}(\mathbb{Z})}$ with $\mathcal{U}(\mathfrak{L})$.

\subsection{A congruence subgroup of  $PSL(2,\mathfrak{L})$.}

\noi Let $\mathcal{A}(2,\mathfrak{L})$ denote the finite-index subgroup of
$\mathcal{A}(\mathfrak{L})$ generated by the 12 translations
$$
\{\tau_{\mathbf{u}+\mathbf{v}}\quad:\quad \mathbf{u}\neq \mathbf{v} \quad \rm{and} \quad  \mathbf{u,v}= \pm \mathbf{i}, \pm\mathbf{j}, \pm \mathbf{k} \}. 
$$

\noi In fact we only need the three translations $\tau_{\mathbf{i}+\mathbf{j}},\tau_{\mathbf{i}+\mathbf{k}}$ and $\tau_{\mathbf{j}+\mathbf{k}}$ to generate $\mathcal{A}(2,\mathfrak{L})$ but the twelve  translations are important to describe its fundamental domain (see section 10). 

\noi Thus $\mathcal{A}(2,\mathfrak{L})$ consists of elements corresponding to matrices in $PSL(2,\mathfrak{L})$ associated to the general Lipschitz translation
$\tau_{x,y,z}=\left(\begin{array}{cc}1 &
  x\mathbf{i}+y\mathbf{j}+z\mathbf{k} \\0 & 1 \end{array}\right)$ such that $x+y+z \equiv 0$ (mod 2).
  
 \begin{definition} Let $\Gamma(2,\mathfrak{L})$ be the group generated by
 $\mathcal{A}(2,\mathfrak{L})$ and the inversion $T$.
 \end{definition}

 \noi This group plays the role of a congruence group modulo two and in fact it is a subgroup of index two of $PSL(2, \mathfrak{L})$. It corresponds to a subgroup of Lorentz transformations with integer entries which will become particularly important in section \ref{Lor}.
 
 \section{The Hurwitz modular group and its unitary and affine subgroups.}

\noi In analogy with the introduction of the unitary, affine and modular groups in
the Lipschitz integers setting of the previous sections, we give the following generalization

\begin{definition} Let $\mathfrak{H}_u$ be the group of {\em Hurwitz units} 

   $$\mathfrak{H}_u:=
\{\pm 1,\ \pm \mathbf{i},\ \pm\mathbf{j},\pm\mathbf{k},\ \frac{1}{2} (\pm 1 \pm \mathbf{i} \pm\mathbf{j} \pm \mathbf{k}) :\ \mathbf{i}^2=\mathbf{j}^2=\mathbf{k}^2=-1, \mathbf{ij=k} \}$$
\noi where in $\frac{1}{2} (\pm 1 \pm \mathbf{i} \pm\mathbf{j} \pm \mathbf{k})$ all 16 possible combinations of signs are allowed.
\end{definition} 

\noi This group is of order 24 and it is known as the {\it binary tetrahedral} group. Its elements can be seen as the vertices of the 24-cell. We recall that the 24-cell is a convex regular 4-polytope, whose boundary is composed of 24 octahedral cells with six meeting at each vertex, and three at each edge. Together they have 96 triangular faces, 96 edges, and 24 vertices. It is possible to give an (ideal) model of the 24-cell by considering the convex hull (of the images) of the 24 unitary Hurwitz numbers via the Cayley transformation $\Psi(\mathbf{q})=(1+\mathbf{q})(1-\mathbf{q})^{-1}$. 

\begin{definition}
The subgroup $\mathcal{U}(\mathfrak{H})$ of $PSL(2,\mathbb{H})$ given by the 12 diagonal matrices 
$$D_\mathbf{u}:=\begin{pmatrix}
\mathbf{u}&0\\
0&\mathbf{u}
\end{pmatrix}$$
with $\mathbf{u}$ a Hurwitz unit is called {\em Hurwitz unitary group}. 
\end{definition}

\noi The epimorphism $\mathfrak{H}_u\to \mathcal{U}(\mathfrak{H})$ given by $\mathbf{u}\mapsto{D_\mathbf{u}}$ has kernel $\{1,-1\}$ so it is of order two. Any matrix in $\mathcal{U}(\mathfrak{H})$ satisfies the (BG) conditions and is an isometry which represents a rotation in $\mathbf{H}^1_{\mathbb{H}}$. Moreover, we observe that the action on  $\mathbf{H}^{1}_{\mathbb{H}}$ of the transformation $D_{\mathbf{u}}$ where $\mathbf{u}=\mathbf{i},\mathbf{j},\mathbf{k}$ or $\mathbf{u}=\frac{1}{2}(\pm \mathbf{i} \pm \mathbf{j} \pm \mathbf{k})$ is for conjugation and sends a quaternion $\mathbf{q}\in \mathbf{H}^{1}_{\mathbb{H}}$ to $\mathbf{uqu}^{-1}$. If $\mathbf{u}=\mathbf{i},\mathbf{j},\mathbf{k}$ it acts as a rotation of angle $\pi$ and if $\mathbf{u}=\frac{1}{2}(\pm \mathbf{i} \pm \mathbf{j} \pm \mathbf{k})$ it acts as a rotation of angle $\frac{2\pi}{3}$. The axis of rotation of the transformation $D_{\mathbf{u}}$ is the vertical hyperbolic 2-plane
$$
S_{\mathbf{u}}=\{ x+y\mathbf{u}\,\, :\,\, x,y \in \R, x>0 \}
.$$

\noi The  group $\mathcal{U}(\mathfrak{H})$ is of order 12 and in fact it is isomorphic to the group of orientation preserving isometries  
 of the regular tetrahedron. It clearly contains $\mathcal{U}(\mathfrak{L})$ as a subgroup but is not contained in 
 the Lipschitz modular group $PSL(2, \mathfrak{L})$.

 \begin{definition}
The {Hurwitz modular group}  is the group generated
by the inversion $T$, by the translations $\mathcal{T}_{\Im \mathbb{H}(\mathbb{Z})} $ and by $\mathcal{U}(\mathfrak{H})$. It will be denoted by $PSL(2,\mathfrak{H})$.
\end{definition}
\begin{proposition} The group $PSL(2,\mathfrak{L})
$ is a subgroup of index three of the group $PSL(2,\mathfrak{H})$.
\end{proposition}
\begin{proof} This is so since the order of the group of transformations induced by the diagonal matrices with entries in the Lipschitz units is of index three in the group of transformations induced by diagonal matrices with entries in the 
Hurwitz units.\end{proof}

\begin{definition} The Hurwitz affine subgroup (or the Hurwitz parabolic subgroup) $\mathcal{A}(\mathfrak{H})$ is the group generated by the unitary Hurwitz group $\mathcal{U}(\mathfrak{H})$ and the group of translations $\mathcal{T}_{\Im \mathbb{H}(\mathbb{Z})} $.
Thus,
\begin{subequations}
\begin{align}
\mathcal{A}(\mathfrak{H}) &=\left\{\begin{pmatrix}
\mathbf{u}&\mathbf{u}b\\
0&\mathbf{u}
\end{pmatrix}
:\ \mathbf{u}\in{\mathfrak{H}_u}   ,\ \Re(b)=0  \right\}\\
&=\left\{\begin{pmatrix}
\mathbf{u}&b\mathbf{u}\\
0&\mathbf{u}
\end{pmatrix}
:\ \mathbf{u}\in{\mathfrak{H}_u}   ,\ \Re(b)=0  \right\}.
\end{align}
\end{subequations}

\end{definition}

\noi It follows from the definition that $PSL(2,\mathfrak{L})\subset
PSL(2,\mathfrak{H})$. It is worth observing here that 
using the Cayley transformations $\Psi(\mathbf{q})=(1+\mathbf{q})(1-\mathbf{q})^{-1}$
one can represent the actions (in terms of mutiplication/rotations) of the Hurwitz units
on the unitary sphere $\mathbb{S}^3$ as
transformations of $\mathbf{H}^1_{\mathbb{H}}$.
Indeed, to any such a transformation it is possible to associate one of the following 24 matrices 
$$P_{\mathbf{u}}:=\dfrac{1}{2}\begin{pmatrix}
\mathbf{u}+1&\mathbf{u}-1\\
\mathbf{u}-1&\mathbf{u}+1\\
\end{pmatrix}\in{PSL(2,\mathfrak{H})}
$$
(with $\mathbf{u}$ a Hurwitz unit); each of these matrices represents a rotation around 1 given by the formula 
\begin{equation}\label{rHur}
\mathbf{q}\mapsto ((\mathbf{u}+1)\mathbf{q+u}-1)((\mathbf{u}-1)\mathbf{q+u}+1)^{-1}.
\end{equation}
\noi This way of representing (the group of) Hurwitz units in terms of
matrices can be considered as a way to generalize Pauli matrices. Let
$P(\mathfrak{H})\subset{PSL(2,\mathbb{H})}$ be the group of order 24 of rotations
as in (\ref{rHur}). This group is obviously isomorphic to
$\mathfrak{H}_u$. 
The orbit of 0 under the action of $\mathfrak{H}_u$ on the boundary  $\partial \mathbf{H}^1_{\mathbb{H}} \cup \{\infty\}$ are the vertices of the 24-cell given by
(\ref{v24cell}) in section 10 and are the images under the Cayley transformation of the Hurwitz units.

\noi If we consider the group $\mathfrak{H}_u$ of order 24 of Hurwitz units, then the map 
$$
\mathcal{A}(\mathfrak{H}) \to
\mathfrak{H}_u\, , \,\, \,\,\,
 \left(\begin{array}{cc}\mathbf{u} & \mathbf{u}b \\ 0 & \mathbf{u}\end{array}\right) \mapsto \mathbf{u}
 $$
\noi  is an epimorphism whose kernel is:
$$
\mathcal{T}_{\Im \mathbb{H}(\mathbb{Z})}=\left\{ \left(\begin{array}{cc}1 & \omega \\ 0 & 1\end{array}\right)\, :\,\, \omega \in \Im\mathbb{H}(\mathbb{Z})  \right\}
$$ 

\noi Thus we have the exact sequence 
$$
0\longrightarrow{ \mathcal{T}_{\Im \mathbb{H}(\mathbb{Z})} }\longrightarrow{ \mathcal{A} }(\mathfrak{H})\longrightarrow\mathcal{U}(\mathfrak{H})\longrightarrow0
$$
\noi This sequence splits and the group $\mathcal{A}(\mathfrak{H})$ is the semi-direct product of $\mathcal{T}_{\Im \mathbb{H}(\mathbb{Z})}$ with $\mathcal{U}(\mathfrak{H}).$

\noi The group $\mathcal{U}(\mathfrak{L}) \subset \mathcal{U}(\mathfrak{H})$ is a normal subgroup and we have the exact sequence
 
$$
0\longrightarrow{ \mathcal{U}(\mathfrak{L})}\longrightarrow{\mathcal{U}(\mathfrak{H})}\longrightarrow {\mathbb{Z}/3\mathbb{Z}}\longrightarrow0
$$ 
\begin{definition}
Let $\hat{\mathcal{U}}(\mathfrak{L})$ and  $\hat{\mathcal{U}}(\mathfrak{H})$ be the subgroups of $PSL(2,\mathfrak{L})$ and $PSL(2,\mathfrak{H})$ which fix 1
(in fact these subgroups are the maximal subgroups  which also preserve the cube $\mathcal{C}$ and the hyperplane $\Pi$). 
\end{definition}

\noi We have the following proposition as a conseguence of all previous results: 
\begin{proposition} 
The groups $\hat{\mathcal{U}}(\mathfrak{L})$ and $\hat{\mathcal{U}}(\mathfrak{H})$ are the subgroups generated by $T$ and $\mathcal{U}(\mathfrak{L})$ and $T$  and $\mathcal{U}(\mathfrak{H})$, respectively. We write
$$
\hat{\mathcal{U}}(\mathfrak{L})=\langle T, \mathcal{U}(\mathfrak{L}) \rangle \quad\quad \rm{and}\quad\quad \hat{\mathcal{U}}(\mathfrak{H})=\langle T, \mathcal{U}(\mathfrak{H})\rangle.
$$

\noi Since $T^2=\mathcal{I}$ and $T$ commutes with all of the elements of $\mathcal{U}(\mathfrak{L})$ and  $\mathcal{U}(\mathfrak{L})$  we have:
$$
\hat{\mathcal{U}}(\mathfrak{L})=\Z/2\Z\oplus\mathcal{U}(\mathfrak{L})\quad\quad \rm{and}\quad\quad \hat{\mathcal{U}}(\mathfrak{H})=\Z/2\Z\oplus\mathcal{U}(\mathfrak{H}).
$$ 
\end{proposition}

\begin{remark} The groups  $PSL(2,\mathfrak{L})$ and $PSL(2,\mathfrak{H})$ are discrete and preserve the half-space  $\mathbf{H}^1_{\mathbb{H}}$ and the hyperbolic metric $ds$ (as introduced in Remark 2.7) so they are \emph{4-dimensional hyperbolic Kleinian groups} in the sense of Henri Poincar\'e (see the book of M. Kapovich \cite{Ka}).
\end{remark} 

\section{Examples of polyhedra and orbifolds modeled on
  $\mathbf{H}^{1}_{\mathbb{H}}$.}
 \noi Let us start by recalling some basic facts about group
 actions.
\begin{definition}
Given a group $\Gamma_{\Omega}$ acting continuously on a metric space $\Omega$, we say that a subset
$\mathcal{D}$ of $\Omega$ is a {\it fundamental domain} 
for $\Gamma_{\Omega}$ if it contains exactly one point from each of the 
 images of a single point under the  action of $\Gamma_{\Omega}$ (the so called
 {\em orbits} of $\Gamma_{\Omega}$).
\end{definition} 

\noi There are many ways to choose a fundamental domain for a group of
transformations of $\Omega$  although  it generally
serves as a geometric realization for the abstract set of
representatives of the orbits. Typically, a
fundamental domain is required to be a convex subset with some
restrictions on its boundary, for example, smooth or polyhedral. The
images of a chosen fundamental domain under the group action then tessellate
the space $\Omega$. 

\noi In this paper we'll mainly deal with groups of
matrices whose entries are quaternions and therefore acting on quaternionic hyperbolic spaces;  
we then investigate  their fundamental domains and the quotient
spaces. In this section we  investigate  
the geometric properties  of some
hyperbolic orbifolds obtained as quotients 
of $\mathbf{H}_{\mathbb{H}}^{1}$ under the action of some discrete subgroups of isometries generated by reflections
on the faces of two convex polyhedra of finite volume.

\noi For an introduction to the theory of orbifolds we refer to the
appendix of this paper.
One way to construct examples of good orbifolds of finite volume
which are modeled on  $ PSL(2, \mathfrak{L})$ 
is to consider ideal convex polytopes in $\mathbf{H}^{1}_{\mathbb{H}}$  
with a finite number of points at infinity and
with some faces identified in pairs 
by elements of $ PSL(2, \mathfrak{L})$
and others which admit a subdivision in subfaces 
identified in pairs 
by elements of $ PSL(2, \mathfrak{L})$. 

\subsection{A quaternionic kaleidoscope.}

\noi We begin with the ideal convex hyperbolic polytope $\mathcal
P$ with one vertex at infinity which is the intersection of the
half-spaces which contain 2 and which are determined by the following
set of hyperbolic hyperplanes

\begin{equation*}\label{P}
\mathrm{Hyperplanes}\ \mathrm{of}\ \mathrm{the}\ \mathrm{faces}\ \mathrm{of}\ \mathcal{P}=\quad\left\{\begin{array}{l}
  \Pi:=\{\mathbf{q}\in\mathbf{H}^{1}_{\mathbb{H}}\,\,\,:\,|\mathbf{q}|=1
  \}\\ \\ \Pi_{-\frac{\mathbf{i}}{2}}:=\{\mathbf{q}\in\mathbf{H}^{1}_{\mathbb{H}}\,\,\,:\,\mathbf{q}=x_0-\frac12{\mathbf{i}}+x_2\mathbf{j}+x_3\mathbf{k},\,\,x_0>0,\,\,x_2,x_3\in\mathbb
  R \}\\\\
   \Pi_{\frac{\mathbf{i}}{2}}:=\{\mathbf{q}\in\mathbf{H}^{1}_{\mathbb{H}}\,\,\,:\,\mathbf{q}=x_0+\frac12{\mathbf{i}}+x_2\mathbf{j}+x_3\mathbf{k},\,\,x_0>0,\,\,x_2,x_3\in\mathbb
  R \}\\\\
   \Pi_{-\frac{\mathbf{j}}{2}}:=\{\mathbf{q}\in\mathbf{H}^{1}_{\mathbb{H}}\,\,\,:\,\mathbf{q}=x_0+x_1\mathbf{i}-\frac12{\mathbf{j}}+x_3\mathbf{k},\,\,x_0>0,\,\,x_1,x_3\in\mathbb
  R \}\\\\
   \Pi_{\frac{\mathbf{j}}{2}}:=\{\mathbf{q}\in\mathbf{H}^{1}_{\mathbb{H}}\,\,\,:\,\mathbf{q}=x_0+x_1\mathbf{i}+\frac12{\mathbf{j}}+x_3\mathbf{k},\,\,x_0>0,\,\,x_1,x_3\in\mathbb
  R \}\\\\
   \Pi_{-\frac{\mathbf{k}}{2}}:=\{\mathbf{q}\in\mathbf{H}^{1}_{\mathbb{H}}\,\,\,:\,\mathbf{q}=x_0+x_1\mathbf{i}+x_2\mathbf{j}-\frac12{\mathbf{k}},\,\,x_0>0,\,\,x_1,x_2\in\mathbb
  R\}\\\\ 
  \Pi_{\frac{\mathbf{k}}{2}}:=\{\mathbf{q}\in\mathbf{H}^{1}_{\mathbb{H}}\,\,\,:\,\mathbf{q}=x_0+x_1\mathbf{i}+x_2\mathbf{j}+\frac12{\mathbf{k}},\,\,x_0>0,\,\,x_1,x_2\in\mathbb
  R \}\\\\
 \end{array}
\right.
\end{equation*}

\begin{figure}
\centering
\includegraphics[scale=0.3]{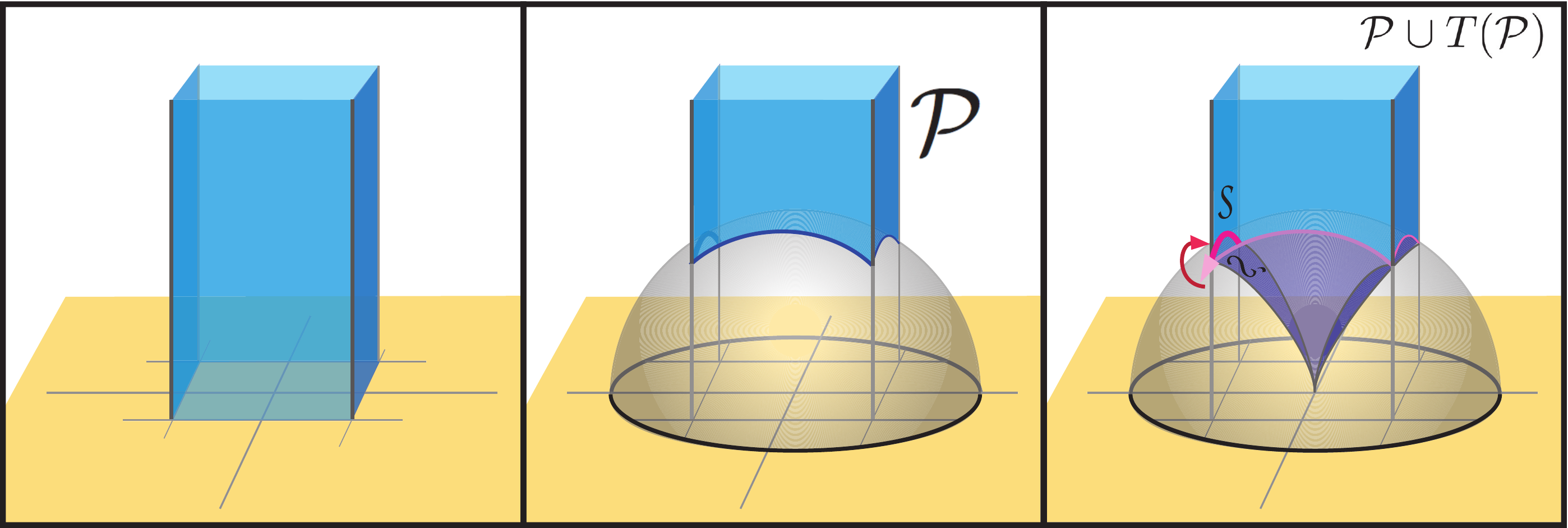}
\begin{center}
{{\bf Figure 1.} Schematic picture of the chimney which is the fundamental domain of the parabolic group  $\mathcal{T}_{\Im \mathbb{H(Z)}}$ (generated by the translations $\tau_{\mathbf{i}},$ $\tau_{\mathbf{j}}$ and $\tau_{\mathbf{k}}$), the polytope $\mathcal P$ and the polytope $\mathcal P$ and its inversion $T(\mathcal P)$. The horizontal plane represents the purely imaginary quaternions that forms the ideal boundary $\partial \mathbf{H}^1_{\mathbb{H}}$ and above it the open half-space of quaternions with positive real part $ \mathbf{H}^1_{\mathbb{H}}$}. 
\end{center}
\end{figure}

\noi The polytope $\mathcal P$ is  bounded by the hemisphere $\Pi$ and the six hyperplanes $\Pi_n$ ($n=\frac{\mathbf{i}}{2},
-\frac{\mathbf{i}}{2}, \frac{\mathbf{j}}{2},-\frac{\mathbf{j}}{2},\frac{\mathbf{k}}{2},-\frac{\mathbf{k}}{2}$) that are orthogonal to the ideal boundary and 
pass through the point at infinity that is denoted by $\infty$.

\noi The only ideal vertex of $\mathcal P$ is the point at infinity. The (non ideal) vertices of $\mathcal P$ 
are the eight points $\frac{1}{2}(\pm \ii\pm\jj \pm \kk)$ which are the vertices of the cube $\mathcal{C}\subset\Pi$ which was defined in subsection 3.2.

\noi The politope $\mathcal P$ has seven 3-dimensional faces:
one compact cube $\mathcal{C}$ and six pyramids with one ideal vertex at
$\infty$ as their common apex and the six squares of the cube $\mathcal{C}$ as their bases. Moreover $\mathcal P$ has 20 2-dimensional faces (6 compact squares and
12 triangles with one ideal vertex) and 20 edges (12 compact and 8 with one
ideal vertex).

\noi The Euler characteristic of $\mathcal{P}$ is
$$\chi(\mathcal{P})=c_0-c_1+c_2-c_3+c_4=8-20+20-7+1=2.$$

\noi The convex polytope $\mathcal P$ satisfies the conditions of the
Poincar\'e's polyhedron theorem, therefore the group generated by
reflections on the faces of $\mathcal P$ is a discrete subgroup of
hyperbolic isometries of $\mathbf{H}_\mathbb{H}^1$. We denote this subgroup by $G(3)$. The index-two
subgroup generated by composition of an even number of reflections has
as fundamental domain the convex polytope $\mathcal{P}\cup{T(\mathcal
  {P)}}$. This subgroup of  $PSL(2,\mathbb{H}(\mathbb{Z}))$ which consists of orientation-preserving isometries will denoted by $G(3)_+$. We will see below in section 6  that $\mathcal P$ can be tessellated by four copies of the
fundamental domain of the action of $PSL(2,\mathfrak{L})$ and by twelve copies of the
fundamental domain of the action of $PSL(2,\mathfrak{H})$ on $\mathbf{H}^{1}_{\mathbb{H}}$. The
quotient space $\mathbf{H}^{1}_{\mathbb{H}}/ G(3)$ is a quaternionic kaleidoscope which is a good non-orientable orbifold. Since the polytope
$\mathcal P$ is of finite volume the non-orientable orbifold obtained is
finite and has the same volume. If we imagine we are inside $\mathbf{H}^{1}_{\mathbb{H}}/ G(3)$ for a moment and open our eyes we see 4-dimensional images very similar to the 3-dimensional honeycombs of Roice Nelson of the figures 9 and 10.

\noi The orientable orbifold $\mathbf{H}^{1}_{\mathbb{H}}/ G(3)_+$ is obtained from the double pyramid $\mathcal{P}\cup {T(\mathcal {P})}$ by identifying in pairs the faces with an ideal vertex at infinity with corresponding faces with an ideal vertex at zero.  These 3-dimensional faces meet at the square faces of the cube $\mathcal{C}$ in $\Pi$ and they are identifying in pairs by a rotation of angle $2\pi /3$ around the hyperbolic plane that contains the square faces. The underlying space is $\mathbb{R}^4$ and the singular locus of $\mathcal{O}_{G(3)_+}$ is a cube. This group is generated by the six rotations of angle $2\pi/3$
around the hyperbolic planes that contain the square faces of the cube $\mathcal{C}$.

\subsection{An orientable bad orbifold $\mathcal{O}_{
K^3}$.}
\noi Now we will consider an orientable bad orbifold $\mathcal{O}_{
K^3}$ obtained by identifying pairs of parallel faces of $\mathcal{P}$ with a vertex at infinity by the translations $\tau_{\mathbf{i}},\tau_{\mathbf{j}}$, and $\tau_{\mathbf{k}}$ and identifying pairs of points of the cube $\mathcal{C}$ by $T$.
Although this orbifold is not good it is also very interesting since it has as a deformation retract the real Kummer 3-fold $K^3$, which is defined in the appendix. This Kummer 3-fold is the set of singularities of the orbifold.

\begin{proposition}
The orbifold  $\mathcal{O}_{K^3}$ 
has the following properties: 

\begin{enumerate}

\item $\mathcal{O}_{K^3}$  has one cusp at infinity; 

\item $\mathcal{O}_{K^3}$  has finite volume;

\item the map $F: \mathcal{O}_{K^3} \times [0,1]\to \mathcal{O}_{K^3} \times[0,1]$ given by
$$F(([z_1,z_2,z_3], t), s)=(([z_1,z_2,z_3], t),(1-s)t)$$ 
\noi is well-defined and gives a strong deformation retract of
 $\mathcal{O}_{K^3}$ to the real Kummer variety $K_3$;
 
\item 
$\mathcal{O}_{K^3}-\{[z_1,z_2,z_3,0]\in\mathcal{O}_{K^3} \}=
\mathbf{T}^3\times(0,1)$;

\begin{figure}
\centering
\includegraphics[scale=0.3]{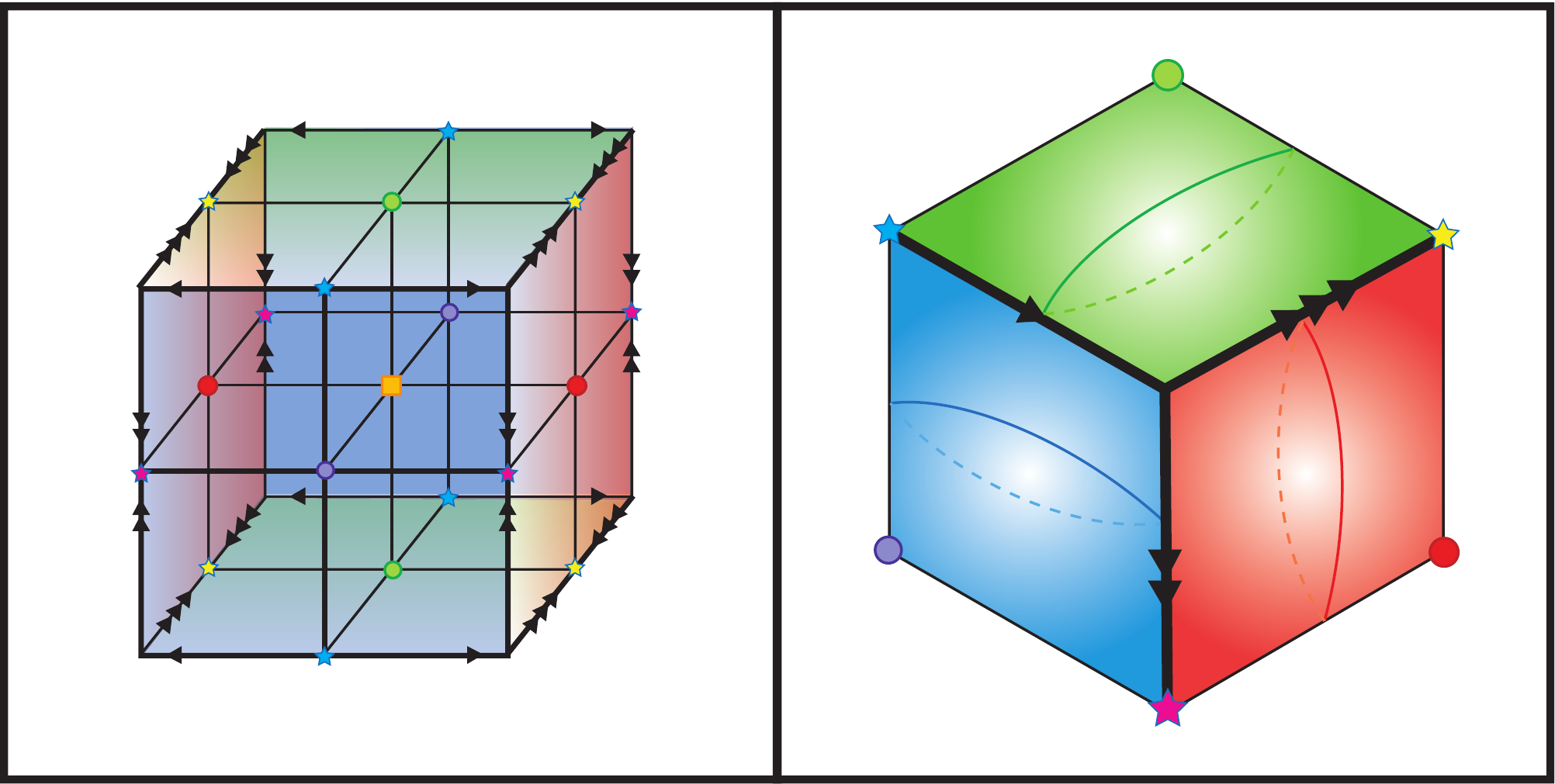}
\begin{center}
{{\bf Figure 2.} The boundary of the hyperbolic cube $\mathcal{C}$ with the identifications by $T$ and the translations $\tau_{\mathbf{i}}$, $\tau_{\mathbf{j}}$ and $\tau_{\mathbf{k}}$ is a hyperbolic 3-Kummer space. The picture at the right is its singular locus: a bouquet of three copies of $\mathbb{S}^2$} (``pillows'') attached to a tripod.
\end{center}
\end{figure}

\item The orbifold $\mathcal{O}_{K^3}$ is simply connected and 

\item its homology groups are: 
$$H_{0}(\mathcal{O}_{K^3})=\Z \quad  H_{1}(\mathcal{O}_{K^3})=0\quad
H_{2}(\mathcal{O}_{K^3})=\Z^{3}\oplus\Z/2\Z
\quad H_{3}(\mathcal{O}_{K^3})=0.$$ 
 \end{enumerate}
\end{proposition}
\begin{proof}
The volume of $\mathcal{O}_{K^3}$ is a rational function of $\pi$ by a theorem of Coxeter 
and it was computed explicitely in \cite{Vin}. Furthermore, as a topological space, 
 $\mathcal{O}_{K^3}$  is homeomorphic to
 $\mathbf{T}^3\times[0,1)/\mathfrak t$, 
where 
$$\mathbf{T}^3=\{(z_1,z_2,z_3) \in \mathbb{C}^{3} \ : \ |z_1|=|z_2|=|z_3|=1\}$$ and $\mathfrak t(z_1,z_2,z_3, 0)=(\overline{z_1},\overline{z_2},\overline{z_3},0)$ so that the homology groups of $\mathcal{O}_{K^3}$ and $K^{3}$
are isomorphic, i.e.
$H_{p}(K^{3})=H_{p}(\mathcal{O}_{K^3})$ $p=0,1,2,3,4$.
In particular, since the Kummer variety $K_3$ has real dimension 3 it 
can be triangulated as a compact 
 non orientable pseudomanifold; therefore, it  is clear that 
$H_{0}(\mathcal{O}_{K^3})=\Z$, $H_{3}(\mathcal{O}_{K^3})=0$ and $ H_{4}(\mathcal{O}_{K^3})=0$.

\noi On the other hand,  $K_3$  can be obtained as a quotient of the cube $[-1/2,1/2]\times [-1/2,1/2]\times [-1/2,1/2]$ 
by identifying opposite faces by translations and 
symmetric points with respect to the origin which is the center of the cube.
After this identification, the boundary of the cube has the homotopy type of a bouquet of three copies of $\mathbb{S}^2$ and the round ball in $\mathbb{R}^3$ 
centered at the origin and radius $1/4$ becomes the cone $A$ over the projective real plane $\mathbf{P}^2_{\mathbb{R}}$. Let $B$ 
be the complement in $K_3$ of the interior of  $A$. 
We can retract radially $B$ to the boundary of the cube with identification and observe that 
$A\cap B=\mathbf{P}^2_{\mathbb{R}}$. 
Then $K_3=A\cup B$ with $A$ contractible and $B$ homotopically equivalent to a bouquet of three 2-spheres. 
Hence it follows from the celebrated Seifert--van Kampen Theorem that  $K_3$ is simply connected. 
Finally from Mayers--Vietoris sequence we have

$$\longrightarrow H_3(K_3)\longrightarrow H_2(A\cap B)\longrightarrow H_2(A)\oplus H_2(B)
\longrightarrow H_2(K_3)\longrightarrow H_1(A\cap B)\longrightarrow 0;$$
since 
$$H_2(A\cap B)=H_2(\mathbf{P}^2_{\mathbb{R}})=0,$$
$$H_1(A\cap B)=H_1(\mathbf{P}^2_{\mathbb{R}})=\Z/2\Z,$$ 
$$H_2(A)\oplus H_2(B)=H_2(B)=\Z\times\Z\times\Z,$$ 
we obtain  the exact sequence 

$$0\longrightarrow  \Z\times\Z\times\Z\longrightarrow H_2(K_3)\longrightarrow \Z/2\Z\longrightarrow 0$$
so we conclude that $H_{2}(\mathcal{O}_{K^3})=\Z^{3}\oplus\Z/2\Z$.
\noi Finally, if we remove the Kummer variety $K_3$ we obtain the structure of the cusp at infinity.

\end{proof}

\section{Fundamental domains of the quaternionic modular groups $ PSL(2, \mathfrak{L})$ and $ PSL(2, \mathfrak{H})$.}

\noi We start from  the following important lemma:

 \begin{lemma}\label{re} Let $\gamma\in{PSL(2,\mathbb{H}})$ 
 satisfy (BG) conditions. If
 $\mathbf{q}\in{\mathbf{H}^{1}_{\mathbb{H}}}$, then   
\begin{equation}
 \Re (\gamma(\mathbf{q}))=\frac{\Re(\mathbf{q})}{|\mathbf{q}c+d|^2}
\end{equation}
 
 \end{lemma}
 \begin{proof}
\noi We recall that if $\mathbf{q} \in \mathbf{H}^1_{\mathbb{H}}$ 
the action of $\gamma$ in $\mathbf{H}^1_{\mathbb{H}}$ is given by the rule
\begin{eqnarray*}
 \nonumber \gamma(\mathbf{q}) & = &
 (a\mathbf{q}+b)(c\mathbf{q}+d)^{-1}\\ & = &
 (a\mathbf{q}+b)(\overline{\mathbf{q}}
 \overline{c}+\overline{d})\Big(\frac{1}{|\mathbf{q}c+d|^2}\Big).
\end{eqnarray*} 
 \noi Then:
 \begin{eqnarray*}
  \nonumber \Re (\gamma(\mathbf{q}))& = &\frac{\Re
    (a\mathbf{q}+b)(\overline{\mathbf{q}}\,\overline{c}+\overline{d})+(c\mathbf{q}+d)(\overline{\mathbf{q}}\,\overline{a}+\overline{b})}{2|\mathbf{q}c+d|^2}\\ &
  =
  &\frac{|\mathbf{q}|^2a\overline{c}+a\mathbf{q}\overline{d}+b\overline{\mathbf{q}}\overline{c}+b\overline{d}
    +|\mathbf{q}|^2c\overline{a}+c\mathbf{q}\overline{b}+d\overline{\mathbf{q}}\overline{a}+d\overline{b}}{2|\mathbf{q}c+d|^2}\\ &
  =
  &\frac{\Re(b\overline{\mathbf{q}}\,\overline{c}+a\mathbf{q}\,\overline{d})}{|\mathbf{q}c+d|^2}.
\end{eqnarray*} 
  
\noi Let $\mathbf{q}=x+yI$, where $x>0$, $y\in \mathbb{R}$ and
$I^2=-1$. 
Then $\overline{\mathbf{q}}=x-yI$ and 
\begin{eqnarray*}
 \nonumber \Re (\gamma(\mathbf{q}))& = &\frac{\Re
   (b(x-yI)\overline{c}+a(x+yI) \overline{d})}{|\mathbf{q}c+d|^2}\\ &
 = &\frac{\Re
   (xb\overline{c}-ybI\overline{c}+xa\overline{d}+yaI\overline{d})}{|\mathbf{q}c+d|^2}\\ &
 = &\frac{x+\Re
   (-ybI\overline{c}+yaI\overline{d})}{|\mathbf{q}c+d|^2}\\ & =
 &\frac{x-ybI\overline{c}+ycI\overline{b}+yaI\overline{d}
   -ydI\overline{a})}{|\mathbf{q}c+d|^2}\\ & =
 &\frac{x+y(-bI\overline{c}+cI\overline{b}+aI\overline{d}
   -dI\overline{a})}{|\mathbf{q}c+d|^2}\\
\end{eqnarray*}

\noi On the other hand, 
since $\left(\begin{array}{cc} a & b \\ c & d\end{array}\right)\in
  PSL(2,\mathbb{H})$ and $\left(\begin{array}{cc} 1 & I \\ 0 & 1\end{array}\right)\in
  PSL(2,\mathbb{H})$ and both satisfy (BG) conditions, then 
$$ \left(\begin{array}{cc} a & b \\ c & d\end{array}\right) \cdot
    \left(\begin{array}{cc} 1 & I \\ 0 &
      1\end{array}\right)=\left(\begin{array}{cc} a & aI+b \\ c &
        cI+d\end{array}\right) \in PSL(2,\mathbb{H})
$$
and satisfies (BG) conditions.
Therefore (BG) conditions imply:
$$(aI+b)(-I\overline{c}+\overline{d})+(cI+d)(-I\overline{a}+\overline{b})=0,$$
Then,
$$-aI^2\overline{c}+aI\overline{d}-bI\overline{c}+b\overline{d}-cI^2\overline{a}+cI\overline{b}-dI\overline{a}+d\overline{b}=0,$$
 
$$aI\overline{d}-dI\overline{a}+cI\overline{b}-bI\overline{c}=0.$$

\noi Finally, we have

$$\Re (\gamma(\mathbf{q}))=\frac{x}{|\mathbf{q}c+d|^2}=\frac{\Re(\mathbf{q})}{|\mathbf{q}c+d|^2}.$$ 
\end{proof}

\noi We notice that if one restricts the entries of the matrices to the set
$\mathbb{H}(\mathbb{Z})$ or $\mathbb{H}ur$, then there are only a finite number of
possibilities for $c$ and $d$ in such a way that $|\mathbf{q}c+d|$ is
less than a given number; therefore we obtain the following important

\begin{corollary} For every $\mathbf{q}\in{\mathbb H}$ one has 
$$\underset{\gamma\in PSL(2,\mathfrak{L})}\sup{\Re(\gamma(\mathbf{q}))}<\infty \qquad  \textrm{and} \qquad \underset{\gamma\in PSL(2,\mathfrak{H})}\sup{\Re(\gamma(\mathbf{q}))}<\infty .$$
\end{corollary}

\noi This corollary is the key reason why the orbifolds $\Hy/PSL(2,\mathfrak{L})$ and $\Hy/PSL(2,\mathfrak{H})$ are all compact. See section 9. The orbifolds $\Hy/PSL(2,\mathbb{H}(\mathbb{Z}))$ and $\Hy/PSL(2,\mathbb{H}ur(\mathbb{Z}))$ of the actions of $PSL(2,\mathbb{H}(\mathbb{Z}))$ and $PSL(2,\mathbb{H}ur(\mathbb{Z}))$ on $\mathbf{H}^5_{\R}$ are compact by a similar inequality. See section 13.  

\subsection{Orbifolds and fundamental domains for translations and for  the inversion.}

\noi From the above considerations and taking into account the actions of
the generators and the affine group one can describe the fundamental domains 
of $PSL(2, \mathfrak{L})$ and $PSL(2, \mathfrak{H})$.\\

\noi We recall that a fundamental domain of the parabolic group $\mathcal{T}_{\Im \mathbb{H(Z)}}$ (generated by the translations $\tau_{\mathbf{i}},$ $\tau_{\mathbf{j}}$ and
$\tau_{\mathbf{k}}$) is the infinite-volume convex hyperbolic chimney (defined in 3.1) with one
vertex at infinity which is the intersection of the half-spaces which
contain 2 and which are determined by the set of six hyperbolic
hyperplanes $\Pi_n$, where $n=\frac{\mathbf{i}}{2},
-\frac{\mathbf{i}}{2}, \frac{\mathbf{j}}{2},-\frac{\mathbf{j}}{2},\frac{\mathbf{k}}{2},-\frac{\mathbf{k}}{2}$.

\noi The hyperbolic 4-dimensional orbifold $\mathcal{M}_{\mathcal{T}_{\Im \mathbb{H}(\mathbb{Z})}}$ is a 2-cusped manifold which is an infinite volume cylinder on the 3-torus $\mathbf{T}^3$ with one cusp (an end of finite volume) and one tube (an end of infinite volume). We can write $\mathcal{M}_{\mathcal{T}_{\Im \mathbb{H}(\mathbb{Z})}}=\mathbf{T}^3 \times \R$.\\

\noi The fundamental domain of the inversion $T$ is closed half--space whose boundary is the hyperbolic hyperplane 
$$
\Pi=\{\mathbf{q}\in\mathbf{H}^{1}_{\mathbb{H}}\,\,\,:\,|\mathbf{q}|=1 \}.\\$$

\noi The hyperbolic 4-dimensional orbifold $\mathcal{M}_{T}$ has a unique singular point and it is homeomorphic to the cone over the real projective space ${\mathbf{P}}^3_{\R}$.
 \subsection{Fundamental domain of $PSL(2,\mathfrak{L})$}

\noi Since the quaternionic modular group $PSL(2, \mathfrak{L})$ is
generated by $\mathcal{T}_{\Im \mathbb{H}(\mathbb{Z})}$ and the inversion $T$, we can choose a fundamental domain which is totally contained in $\mathcal{P}$. 

\begin{figure}
\centering
\includegraphics[scale=0.33]{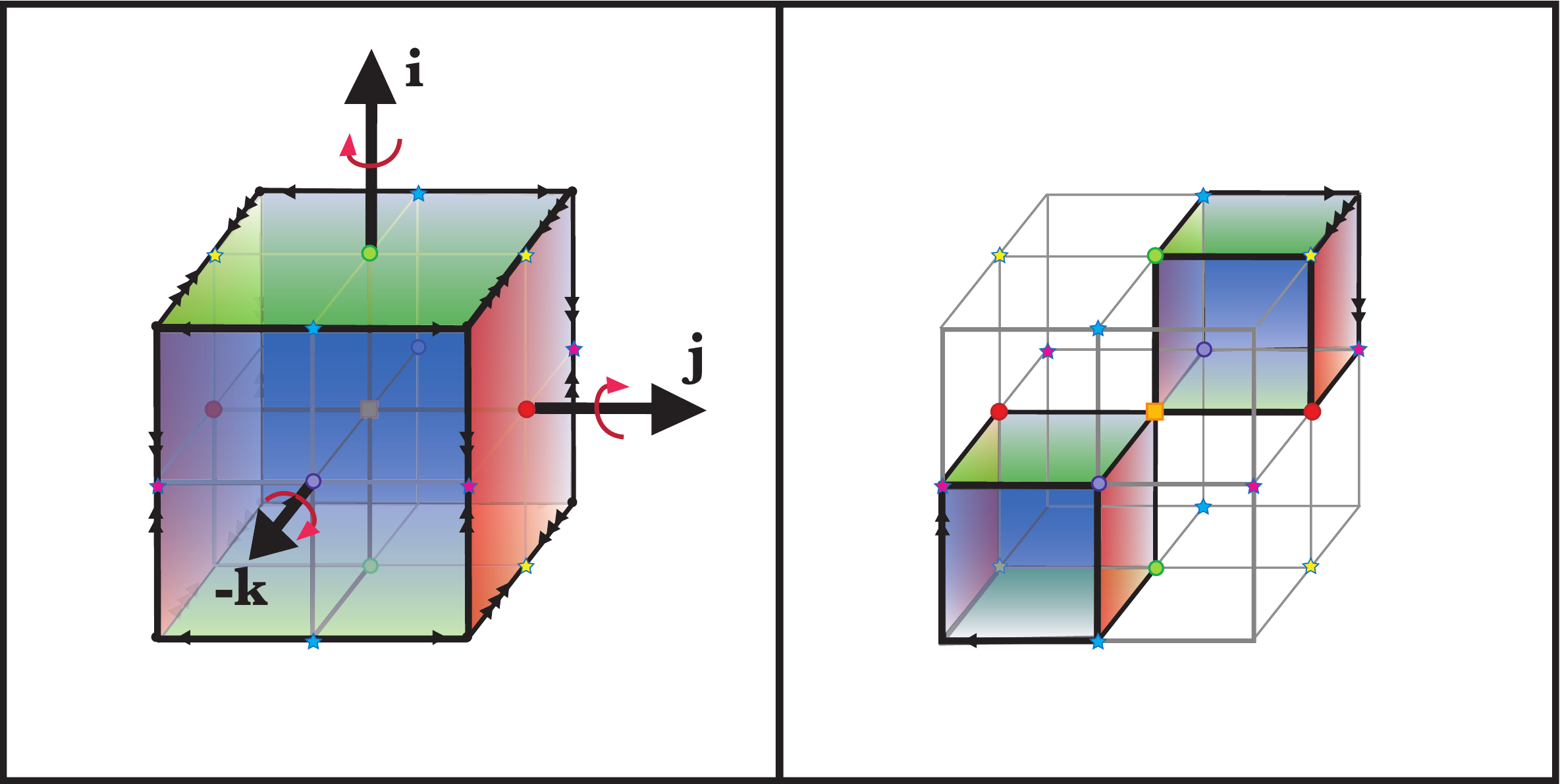}
\begin{center}{{\bf Figure 3.} Left: The action of $U(\mathfrak{L})$ on the cube $\mathcal{C}$. \\ÊRight: The two hyperbolic cubes $\mathcal{C}_1$ and $\mathcal{C}_2$ in $\mathcal{C}$ which are the bases of a fundamental domain $\mathcal{P}_{\mathfrak{L}}$ of $PSL(2,\mathfrak{L})$.}
\end{center}
\end{figure}

\noi The finite Lipschitz unitary group $\mathcal{U}(\mathfrak{L})$ acts by rotations of angle $\pi$ around the three hyperbolic 2-planes generated by 1 and $\mathbf{u}$ where $\mathbf{u}=\mathbf{i, j}$ or $\mathbf{k}$. We divide the cube $\mathcal{C}$ in eight congruent cubes by cutting it along the coordinate planes. Then $\mathcal{P}$ is divided in eight congruent cubic pyramids. We label the cubes with two colors as a chessboard. 

\noi An element of the finite unitary Lipschitz group identifies four cubes (two white cubes and
two black ones) with other four cubes (two white and two black) preserving the colors.
We can give a geometric description of 
 the action of the unitary group by means of its representation on 
the 3-torus $\mathbf{T}^3=\{(z_1,z_2,z_3) \in \mathbb{C}^{3} \ :
\ |z_1|=|z_2|=|z_3|=1\}.$ 
Indeed, we associate to  
$$
\begin{pmatrix}
\mathbf{i}&0\\
0&\mathbf{i} 
\end{pmatrix} \quad \rm{the\ self-map\ of\ }\mathbf{T}^3 \quad \rm{F}_\mathbf{i}(z_1,z_2,z_3)=(z_1,\overline{z_2},\overline{z_3}) $$ 
\noindent and  to

$$\begin{pmatrix} 
\mathbf{j}&0\\
0&\mathbf{j} 
\end{pmatrix}\quad \rm {the\ self-map\ of\ }\mathbf{T}^3  \quad \rm{F}_\mathbf{j}(z_1,z_2,z_3)=(\overline{z_1},z_2,\overline{z_3}).$$

\noi The composition of $\begin{pmatrix}
\mathbf{i}&0\\
0&\mathbf{i} 
\end{pmatrix}$ and $\begin{pmatrix}
\mathbf{j}&0\\
0&\mathbf{j} 
\end{pmatrix}$ is $\begin{pmatrix}
\mathbf{k}&0\\
0&\mathbf{k} 
\end{pmatrix}$ that is represented as 
$$\rm{F}_\mathbf{k}(z_1,z_2,z_3)=(\overline{z_1},\overline{z_2},z_3).$$

\noi This is a sort of ``complex
multiplication'' (like the one defined on certain elliptic curves) on the real torus $\mathbf{T}^3$. The functions $\rm{F}_\ii$, $\rm{F}_\jj$ and $\rm{F}_\kk$ are 
conformal automorphisms of the 3-torus acting as symmetries of the cube $\mathcal C$
and form a group of order 4 isomorphic to $\mathcal{U}(\mathfrak{L})= \Z/2\Z \times
\mathbb Z/2\Z$.

\noi A fundamental domain for $PSL(2,\mathfrak{L})$ can be taken to be the union of
two cubic pyramids with bases two of the cubes described in the previous paragraph, one white and one black and with a common vertex at the point at infinity. We
can choose adjacent cubes to obtain a convex fundamental domain but
this is not necessary to have a fundamental domain. 

\noi The inversion $T$ acts by identifying each white cube with a diametrally opposite black one in $\Pi$. Then a fundamental domain for $PSL(2,\mathfrak{L})$ is the union of two cubic pyramids in $\mathcal{P}$. See figure 3. Below we describe other fundamental domains which are more suitable to study the isotropy groups and the tessellation in $\Hy$ around singular points. 

\begin{definition} Let $\mathcal{C}_1$ and $\mathcal{C}_2$ be the two hyperbolic cubes in $\mathcal{C}$ which contain the vertices $\frac{1}{2}(1+\ii+\jj+\kk)$ and
$\frac{1}{2}(1-\ii-\jj-\kk)$, respectively. 

\noi Let $\mathcal{P}_{\mathfrak{L}}$ be the union of the two hyperbolic cubic pyramids with vertex at infinity and bases the two cubes $\mathcal{C}_1$ and $\mathcal{C}_2$. \end{definition}

\begin{remark} The elements in 
$\mathcal{P}_{\mathfrak{L}}$ are the points $\mathbf{q}=x_{_0}+x_{1}\mathbf{i}+x_{2}\mathbf{j}+x_{3}\mathbf{k}$ in 
$\mathcal{P}$ such that the real numbers $x_{n}$ have the same sign for all $n=1,2,3$.
\end{remark}

 \subsection{Fundamental domain of $PSL(2,\mathfrak{H})$} 
 
\begin{figure}
\centering
\includegraphics[scale=0.33]{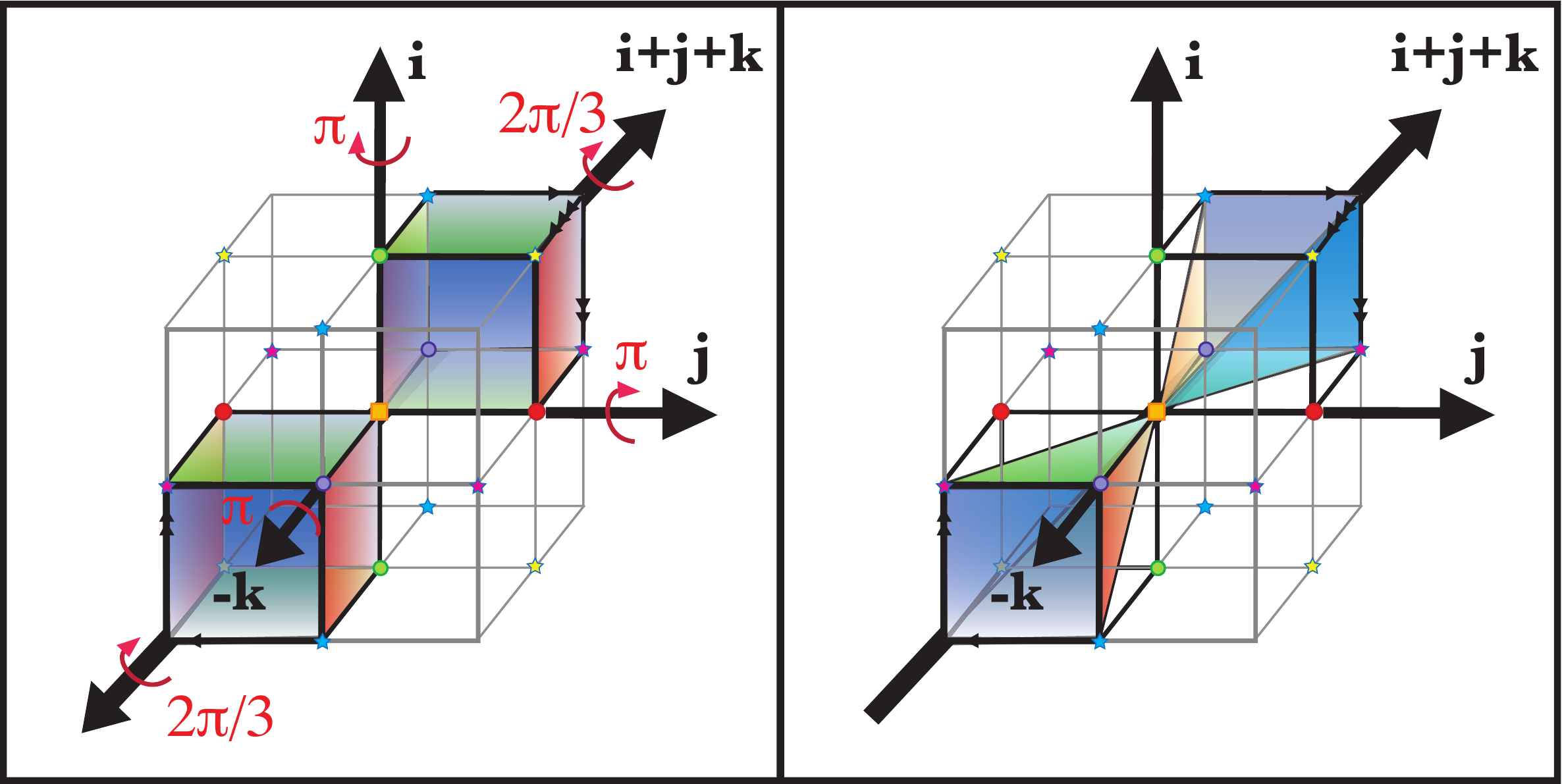}
\begin{center}
{{\bf Figure 4.} Left: The action of $U(\mathfrak{H})$ in the cube $\mathcal{C}$. \\ Right: the bases of the fundamental domain of $PSL(2,\mathfrak{H})$. The two hyperbolic pyramids $\mathcal{P}_1$ and $\mathcal{P}_2$ in $\mathcal{C}$ which are the bases of a fundamental domain $\mathcal{P}_{\mathfrak{H}}$ of $PSL(2,\mathfrak{H})$.}
\end{center}
\end{figure}
 
The analysis of the fundamental domain $\mathcal{P}_\mathfrak{H}$ of the Hurwitz modular group $PSL(2,\mathfrak{H})$
is analogous to the one of the fundamental domain of $PSL(2,\mathfrak{L})$. We recall that $PSL(2,\mathfrak{H})$ is generated by the parabolic group of translations $\mathcal{T}_{\Im \mathbb{H}(\mathbb{Z})}$, the inversion $T$ and the unitary Hurwitz group $\mathcal{U}(\mathfrak{H})$. From 
our previous descriptions of the fundamental domains of the group of translations and the group of order 2 generated by $T$ we know that the 
fundamental domain of $PSL(2,\mathfrak{H})$ is commensurable with  $\mathcal{P}$, the pyramid over the cube $\mathcal{C}$. More precisely,
$\mathcal{P}$ is invariant under $\mathcal{U}(\mathfrak{H})$ and therefore the fundamental domain of $PSL(2,\mathfrak{H})$ is the fundamental domain in 
$\mathcal{P}$ of the action of  $\mathcal{U}(\mathfrak{H})$  on $\mathcal{P}$. Moreover, as $\mathcal{U}(\mathfrak{L}) \subset \mathcal{U}(\mathfrak{H})$ we have that the fundamental domain of $PSL(2,\mathfrak{H})$ is a subset of the fundamental domain $\mathcal{P}_{\mathfrak{L}}$ of the action of $PSL(2,\mathfrak{L})$. Moreover, since $\mathcal{U}(\mathfrak{L})$ is a subgroup of $\mathcal{U}(\mathfrak{H})$ of index three then we have that  $\mathcal{P}_{\mathfrak{H}}$ is a third part of $\mathcal{P}_{\mathfrak{L}}$.

\noi Let $\mathbf{u}\in \mathfrak{H}_\mathbf{u}$, then $D_\mathbf{u}\in \mathcal{U}(\mathfrak{H})$ is induced by the diagonal matrix
$$D_\mathbf{u}=\left(\begin{array}{cc} \mathbf{u} & 0 \\ 0 & \mathbf{u}\end{array}\right)$$
and acts as follows: $\mathbf{q}\mapsto \mathbf{u}\mathbf{q}\mathbf{u}^{-1}=\mathbf{u}\mathbf{q}\overline{\mathbf{u}}$. Then if $\mathbf{u}$ is a Hurwitz unit which is not a Lipschitz unit (\textit{i,e.} $\mathbf{u}=1/2 (1\pm\mathbf{i} \pm\mathbf{j}\pm\mathbf{k}))$ then the matrix $D_\mathbf{u}$ is of order three and geometrically is a rotation of angle $2\pi/3$ around the diagonal
of $\mathcal{C}$ which contains $\mathbf{u}$ or $-\mathbf{u}$, but only one has a positive real part and then is in $\mathbf{H}^{1}_{\mathbb{H}}$. As $D_\mathbf{u}=D_\mathbf{-u}$ we can suppose that $\Re(\mathbf{u})=1/2$.
One has $D_\mathbf{u}^2=D_\mathbf{u^2}=D_{1-\mathbf{u}}$ (since $\Re(\mathbf{u})=1/2>0$, $1-\mathbf{u}$ is Hurwitz unit and $\Re(\mathbf{u})=\Re(1-\mathbf{u})>0$).  

\noi  The group of Hurwitz units $\mathcal{U}(\mathfrak{H})$ acts transitively on the edges of $\mathcal{C}$. Therefore
a fundamental domain can be determined by the choice of one edge of $\mathcal{C}$. Hence a convex fundamental domain is the pyramid with vertex at infinity with base the hyperbolic convex polyhedron with vertices the two end points of the edge, the two barycenters of the square faces that have the edge in common and 1.  However we choose as a fundamental domain a non-convex polyhedron.  
 
\begin{definition} Let $\mathcal{P}_1$ and $\mathcal{P}_2$ be the two hyperbolic 3-dimensional square pyramids in $\mathcal{C}_1\subset  \mathcal{C}$ and $\mathcal{C}_2 \subset  \mathcal{C}$, respectively, with apex  1 and which have as bases the squares in the boundary of $\mathcal{C}$ with sets of vertices 
$$
S:=\{ v_1= \frac{1}{2}(1+\ii+\jj+\kk), v_2=\frac{1}{2}(\sqrt{2}+\ii+\kk), v_3=\frac{1}{2}(\sqrt{2}+\jj+\kk), v_4=\frac{1}{2}(\sqrt{3}+\kk)\} 
$$
\noi  and $T(\{S\}):=(T(v_1), T(v_2),T(v_3).T(v_4))$, respectively.

\noi Let $\mathcal{P}_{\mathfrak{H}}$ be the union of the two hyperbolic 4-dimensional pyramids with vertex at infinity and bases the two hyperbolic 3-dimensional pyramids $\mathcal{P}_1$ and $\mathcal{P}_2$. 

\end{definition}

\subsection{Proof that $\mathcal{P}_{\mathfrak{L}}$ and $\mathcal{P}_{\mathfrak{H}}$ are fundamental domains.}

\noi With the geometric tools so far introduced we can establish the following:

\begin{theorem}\label{Domain} The fundamental domains  $\mathcal{P}_{\mathfrak{L}}$ and $\mathcal{P}_{\mathfrak{H}}$ for the actions of the groups  $PSL(2,\mathfrak{L})$ and $PSL(2,\mathfrak{H})$, respectively, have the following properties:
\begin{enumerate}
\item for every $\mathbf{q}\in\mathbf{H}^{1}_{\mathbb{H}}$ there
  exists $\gamma \in PSL(2,\mathfrak{L})$ (resp. $PSL(2,\mathfrak{H})$) such that $\gamma(\mathbf{q})
  \in \mathcal{P}_{\mathfrak{L}}$ (resp. $\mathcal{P}_{\mathfrak{H}}$). \\

\item If two distinct points $\mathbf{q, q}^{\prime}$ of $ \mathcal{P}_{\mathfrak{L}}$ (resp. $\mathcal{P}_{\mathfrak{H}}$) are congruent modulo $PSL(2,\mathfrak{L})$ (resp. $PSL(2,\mathfrak{H})$); i.e. if there exists $\gamma \in PSL(2,\mathfrak{L})$ (resp. $PSL(2,\mathfrak{H})$) such that $\gamma(\mathbf{q})=\mathbf{q}^{\prime}$, then $\mathbf{q, q}^{\prime} \in \partial  \mathcal{P}_{\mathfrak{L}}$ (resp. $\mathcal{P}_{\mathfrak{H}}$). If $|\mathbf{q}| > 1$ then $\gamma\in\mathcal{A}(\mathfrak{L})$ (resp. $ \mathcal{A}(\mathfrak{H})$). If $|\mathbf{q}| = 1$ then $\gamma \in{\mathcal{A}}(\mathfrak{L})$ (resp. $ \mathcal{A}(\mathfrak{H})$) or $\gamma=A {T}$ where $T$ is the usual inversion and $A\in{\mathcal{A}}(\mathfrak{L})$ (resp. $ \mathcal{A}(\mathfrak{H})$).\\

\item Let $\mathbf{q} \in  \mathcal{P}_{\mathfrak{L}}$ (resp. $\mathcal{P}_{\mathfrak{H}}$) and 
let $G_\mathbf{q}=\{g\in PSL(2,\mathfrak{L})\}$ (resp. $PSL(2,\mathfrak{H})$) be the stabilizer of $\mathbf{q}$ in $PSL(2,\mathfrak{L})\}$ (resp. $PSL(2,\mathfrak{H})$) then $G_\mathbf{q}=\{1\}$ if $\mathbf{q} \neq \partial  \mathcal{P}_{\mathfrak{L}}$ (resp. $\partial \mathcal{P}_{\mathfrak{H}}$).
\end{enumerate}
\end{theorem}

\begin{proof}
\noi Let $\mathbf{q}\in \mathbf{H}^1_{\mathbb{H}}$. By corollary 6.2 there exists
$\gamma \in PSL(2,\mathfrak{L})$ (resp. $PSL(2,\mathfrak{H})$) such that $\Re(\gamma (\mathbf{q}))$ is
maximum. There exists $(n_1,n_2,n_3)\in\Z^3$ such that the element
$\mathbf{q}^{\prime}=\tau_{\mathbf{i}}^{n_1}\tau_{\mathbf{j}}^{n_2}\tau_{\mathbf{k}}^{n_3}\gamma
(\mathbf{q})$ is of the form $\mathbf{q}^{\prime}=x_0+x_1\mathbf{i}+x_2\mathbf{j}+x_3\mathbf{k}$ where $|x_n| \leq\frac{1}{2}, n=1,2,3$. Then $\mathbf{q}^{\prime}$ is an element  of the fundamental domain of the parabolic group $\mathcal{T}_{\Im \mathbb{H}(\mathbb{Z})}$.

\noi  If
$|\mathbf{q}^{\prime}|<1$, then the element $T\mathbf{q}^{\prime}=
(\mathbf{q}^{\prime})^{-1}$ has real part strictly larger
than $\Re(\mathbf{q}^{\prime})=\Re(\gamma (\mathbf{q}))$, which is
impossible. Then we must have $|\mathbf{q}^{\prime}|\geq1$, and
$\mathbf{q}^{\prime} \in \mathcal{P}$.  This shows that given any $\mathbf{q}\in \mathbf{H}^1_{\mathbb{H}}$ there exists 
$\gamma\in{PSL(2,\mathfrak{L})}$ (resp. $PSL(2,\mathfrak{H})$) such that $\gamma(\mathbf{q})\in\mathcal{P}$. We remember that the elements in 
$\mathcal{P}_{\mathfrak{L}}$ are the points $\mathbf{q}=x_{_0}+x_{1}\mathbf{i}+x_{2}\mathbf{j}+x_{3}\mathbf{k}$ in 
$\mathcal{P}$ such that the real numbers $x_{n}$ have the same sign for all $n=1,2,3$. The action of an element $D_\mathbf{u}$, with $\mathbf{u}=\mathbf{i}, \mathbf{j},\mathbf{k}$, in the unitary Lipschitz group $\mathcal{U}({\mathfrak{L}})$ has the property of leaving invariant  $x_{_1}$ and $x_n$ and changing the signs of the other two coefficients. The action of an element $D_\mathbf{u}$, with $\mathbf{u}=\frac{1}{2}(\pm 1\pm \mathbf{i} \pm \mathbf{j} \pm \mathbf{k})$, in the unitary Hurwitz group $\mathcal{U}({\mathfrak{H}})$ has the property that it rotates multiples of $2\pi/3$ the cells of $\mathbf{Y}_{\mathfrak{H}}$ around the diagonal passing through $\mathbf{u}$ and $-\mathbf{u}$ of the cube $\mathcal{C}$. Then we can use one element in $\mathcal{U}({\mathfrak{L}})$ to have a point $\mathbf{q}^{\prime \prime}$ of the orbit of $\mathbf{q}\in \mathcal{P}_{\mathfrak{L}}$. In the Hurwitz case we can use one element in $\mathcal{U}({\mathfrak{H}})$ to have a point $\mathbf{q}^{\prime \prime}$ of the orbit of $\mathbf{q}\in \mathcal{P}_{\mathfrak{H}}$.    This proves (1). In others words, the orbit of any point $\mathbf{q}\in \mathbf{H}^1_{\mathbb{H}}$ under the action of the group
${PSL(2,\mathfrak{L})}$ (resp. $PSL(2,\mathfrak{H})$) has a representative in $\mathcal{P}_{\mathfrak{L}}$ (resp. $\mathcal{P}_{\mathfrak{H}}$). 
 
\noi Let $\mathbf{q} \in  \mathcal{P}_{\mathfrak{L}}$ (resp. $\mathcal{P}_{\mathfrak{H}}$) and let $\gamma=\left(\begin{array}{cc} a & b \\ c & d\end{array}\right) \in PSL(2,\mathfrak{L})$ (resp. $PSL(2,\mathfrak{H})$) such that $\gamma \neq \mathcal{I}$, where $\mathcal{I}$ is the identity matrix in $PSL(2,\mathbb{H})$ and $\gamma (\mathbf{q}) \in  \mathcal{P}_{\mathfrak{L}}$ (resp. $\mathcal{P}_{\mathfrak{H}}$). We can suppose that $\Re(\gamma (\mathbf{q})) \geq \Re (\mathbf{q})$, i.e. $|c\mathbf{q}+d|\leq 1$. This is clearly impossible if $|c|\geq 1$, leaving then the cases $c=0$ or $|c|=1$. 

\begin{enumerate}
\item[I)] If $c=0$, we have $|d|=1$ and (BG) conditions imply that $a\overline{d}$=1 and $b\overline{d}+d\overline{b}=0$. There are two cases: 

\begin{enumerate} 
\item[I.1)] If $d=1$, then $a=1$ and $\Re (b)=0$. Then
$$\gamma=\left(\begin{array}{cc} 1 & b \\ 0 & 1\end{array}\right)$$
where $b=b_{_\mathbf{i}}\mathbf{i}+b_{_\mathbf{j}}\mathbf{j}+b_{_\mathbf{k}}\mathbf{k}$. If $\mathbf{q}=x_{_1}+x_{_\mathbf{i}}\mathbf{i}+x_{_\mathbf{j}}\mathbf{j}+x_{_\mathbf{k}}\mathbf{k} \in  \mathcal{P}_{\mathfrak{L}}$ then
$$\gamma(\mathbf{q})=\mathbf{q}^{\prime}=x_{_1}+(x_{_\mathbf{i}}+b_{_\mathbf{i}})\mathbf{i}+(x_{_\mathbf{j}}+b_{_\mathbf{j}})\mathbf{j}+(x_{_\mathbf{k}}+b_{_\mathbf{k}})\mathbf{k}  \in  \mathcal{P}_{\mathfrak{L}}, \rm{and}$$ 

\begin{enumerate}
\item[I.1.1)] If $|b|=1$ then $b=\pm \mathbf{i},\pm \mathbf{j},\pm \mathbf{k};$ and $\mathbf{q}=r - \frac{b}{2}$, where $r\geq \frac{\sqrt{3}}{2}$. Then $\mathbf{q}$ is on the vertical geodesic that joins a barycenter of a square face of the cube $\mathcal{C}$ in the base of $\mathcal{P}$ with the point at infinity $\infty$ and so $\mathbf{q, q}^{\prime} \in \partial  \mathcal{P}_{\mathfrak{L}}$.

\item[I.1.2)] If $|b|=2$ then $b=\pm \mathbf{i} \pm \mathbf{j},\pm \mathbf{i} \pm \mathbf{k}, \pm \mathbf{j}\pm \mathbf{k};$ and $\mathbf{q}=r - \frac{b}{2}$, where $r\geq\frac{1}{\sqrt{2}}$. Then $\mathbf{q}$ is on the vertical geodesic that joins the middle point of an edge of the cube $\mathcal{C}$ with the point at infinity $\infty$ and so $\mathbf{q, q}^{\prime} \in \partial  \mathcal{P}_{\mathfrak{L}}$.

\item[I.1.2)] If $|b|=3$ then $b=\pm \mathbf{i} \pm \mathbf{j}\pm \mathbf{k};$ and $\mathbf{q}= r - \frac{b}{2}$, where $r\geq\frac{1}{2}$. Then $\mathbf{q}$ is on the vertical geodesic that joins a vertex of the cube $\mathcal{C}$ with the point at infinity and so $\mathbf{q, q}^{\prime} \in \partial  \mathcal{P}_{\mathfrak{L}}$.

\end{enumerate}

\begin{figure}
\centering
\includegraphics[scale=0.33]{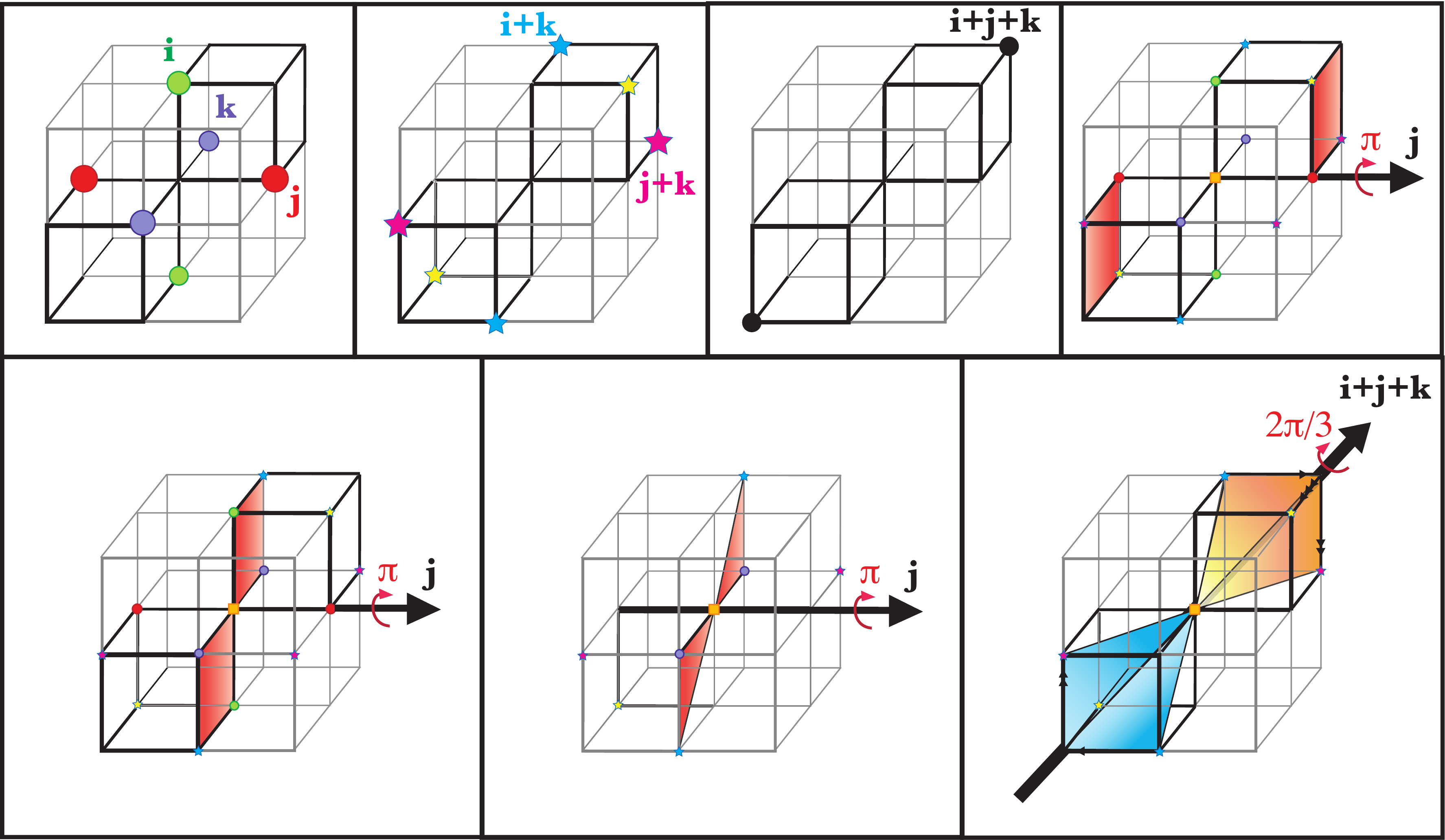}
\begin{center}
{{\bf Figure 5.} Points in $\mathcal{C}$ which are congruent modulo $PSL(2,\mathfrak{L})$ and $PSL(2,\mathfrak{H})$. Top line: Points in $\mathcal{C}$ which are congruent by translations and the top right is the case of the composition of a translation and the inversion $T$. Bottom line: The case of points  in $\mathcal{C}$ congruent by the action of the unitary groups $\mathcal{U}(\mathfrak{L})$ and $\mathcal{U}(\mathfrak{H})$.}
\end{center}
\end{figure}

\item[I.2)] If $d\neq 1$ then $d=a$, $|a|= 1$ and there are two subcases: 
\begin{enumerate} 
\item[I.2.1)] If $b=0$ then 
$$\gamma=\left(\begin{array}{cc} a & 0 \\ 0 & a\end{array}\right) .$$
Then $\mathbf{q}$ is on the hyperbolic plane generated by 1 and $a$  and so $\mathbf{q} \in \partial  \mathcal{P}_{\mathfrak{L}}$.

\item[I.2.2)] If $b\neq 0$, then $\Re (b)=0$, $|b|=1$ hence $b=\pm \mathbf{i}, \pm \mathbf{j},\pm \mathbf{k}$, but $b\neq a$. Then 
\end{enumerate}
$$\gamma=\left(\begin{array}{cc} a & b \\ 0 & a\end{array}\right).$$
Then $\mathbf{q}$ is on the hyperbolic plane generated by 1 and $a$  and so $\mathbf{q} \in \partial  \mathcal{P}_{\mathfrak{L}}$.

\end{enumerate} 

\item[II)] If $c\neq0$, as $|\mathbf{q}|\geq1$ then $d=0$. As $|c\mathbf{q}|\leq 1$ then $|c|=|\mathbf{q}|=1$. Then $\mathbf{q}\in \partial  \mathcal{P}_{\mathfrak{L}}$. (BG) conditions imply that $\overline{b}c=1$ and $a\overline{c}+\overline{a}c=0$.

\begin{enumerate} 
\item[II.1)] If $c=1$, then $b=1$ and $\Re (a)=0$. Then 
$$\gamma=\left(\begin{array}{cc} a & 1 \\ 1 & 0\end{array}\right).$$

\item[II.2)] If $c\neq 1$, then $c=\pm \mathbf{i}, \pm \mathbf{j},\pm \mathbf{k}$ and $b=c$. Moreover $a=0$ or $|a|=1$ and $\Re (a\overline{c})=0$. Then 

$$\gamma=\left(\begin{array}{cc} a & c \\ c & 0\end{array}\right).$$
\end{enumerate}
\end{enumerate}

\noi To prove (3) suppose that $\mathbf{q}\in\rm{Interior}(\mathcal{P}_{\mathfrak{L}})$. Let $\gamma$ be such that 
 $\gamma(\mathbf{q})=\mathbf{q}$. Then there exist $\epsilon>0$ such that $\mathbf{q}+\epsilon$ and 
 $\gamma(\mathbf{q}+\epsilon)\in \rm{Interior}(\mathcal{P}_{\mathfrak{L}})$. 
But then by (2) $\mathbf{q}+\epsilon$ and  $\gamma(\mathbf{q}+\epsilon)$ are in $ \partial  \mathcal{P}_{\mathfrak{L}}$ that is a contradiction. The same proof applies to show that points in $\rm{Interior}(\mathcal{P}_{\mathfrak{H}})$ have trivial isotropy group. 
\end{proof}

\noi If we use the group $PSL(2,\mathfrak{L})$ to propagate
$\mathcal{P}_{\mathfrak{L}}$ we obtain a tessellation of
$\mathbf{H}^{1}_{\mathbb{H}}$ that we denote by $\mathbf{Y}_{\mathfrak{L}}$. The
intersection of $\mathcal{P}_{\mathfrak{L}}$ and $\mathbf{Y}_{\mathfrak{L}}$ with
each of the totally geodesic planes $S_\mathbf{i}$, $S_\mathbf{j}$,
$S_\mathbf{k}$, where
$S_\mathbf{u}:=\{q=x_{_1}+x_{_\mathbf{i}}\mathbf{i}+x_{_\mathbf{j}}\mathbf{j}+x_{_\mathbf{k}}\mathbf{k}\in\mathbf{H}^{1}_{\mathbb{H}}\, : \, x_{_\mathbf{s}}=
0\ \mathrm{if} \ \mathbf{s}\neq \mathbf{u}, 0\}$, with $\mathbf{u}=\mathbf{i,j,k}$, gives a copy
of the closure of a (non-convex) fundamental domain of $PSL(2,\Z)$ and the associated
tessellation in the half--space
model of $\mathbf{H}^{2}_{\mathbb{R}}$. See figure 14 in the appendix. Indeed, it is worth noticing
here that $S_\mathbf{u}$ is an {\em invariant} set for $R_\mathbf{u}, L_\mathbf{u}$ and $\tau_\mathbf{u}$
($\mathbf{u}=\mathbf{i,j,k}$); therefore the intersection of $\mathcal{P}_{\mathfrak{L}}$
and $\mathbf{Y}_{\mathfrak{L}}$ with each of the 3-dimensional totally geodesic
hyperbolic 3-spaces $S_{\mathbf{ij}}$, $S_{\mathbf{jk}}$,
$S_{\mathbf{ik}}$, where
$S_{\mathbf{lm}}:=\{q=x_{_1}+x_{_\mathbf{i}}\mathbf{i}+x_{_\mathbf{j}}\mathbf{j}+x_{_\mathbf{k}}\mathbf{k}\in\mathbf{H}^{1}_{\mathbb{H}}\, : \,\ x_{_\mathbf{s}}=
0\ \mathrm{if} \ \mathbf{s}\neq \mathbf{l}, \mathbf{m},0\}$, with $\mathbf{l},\mathbf{m}=\mathbf{i,j,k}$, gives a
copy of the closure of the classical fundamental domain (and the
tessellation generated by it) of the Picard group $PSL(2,{\mathbb
  Z}[\mathbf{i}])$ for the Gaussian integers acting on the half-space
model of $\mathbf{H}^{3}_{\mathbb{R}}$. See figure 14 in the appendix.

\subsection{A fundamental domain of $PSL(2,\mathfrak{L})$ as an ideal cone over a rhombic hyperbolic dodecahedron} 
\noi We can describe another fundamental domain of the modular group 
which is convex (and in some sense more symmetrical) using ``cut and paste'' techniques as follows: the fundamental domain 
$\mathcal{P}_{\mathfrak{L}}$ is the union of two ideal  pyramids with vertex at infinity with bases
two antipodal cubes $\mathcal{C}_1\subset\Pi$ and $\mathcal{C}_2\subset\Pi$ of $\mathcal{C}$. We recall that  $\mathcal{C}_1$ contains the vertices 1 and $v=\frac{1}{2}(1+\ii+\jj+\kk)$. We can divide one of the cubes for example $\mathcal{C}_2$ into the six hyperbolic square
pyramids contained in $\Pi$ which have as common apex the barycenter of $\mathcal{C}_2$ and bases the square faces of $\mathcal{C}_2$. Let $P_\ii, P_\jj,P_\kk$ be the three pyramids in $\mathcal{C}_2$ that contain the vertices 1 and $\frac{\sqrt{2} -\jj-\kk}{2}, \frac{\sqrt{2} -\ii-\kk}{2},\frac{\sqrt{2} -\ii-\jj}{2}$, respectively. Let $Q_\ii, Q_\jj,Q_\kk$ be the three pyramids in $\mathcal{C}_2$ that do not contain the vertex 1 but contain $\frac{\sqrt{2} -\jj-\kk}{2}, \frac{\sqrt{2} -\ii-\kk}{2},\frac{\sqrt{2} -\ii-\jj}{2}$, respectively. Then $P_\ii, P_\jj,P_\kk$ are permuted by the element $D_v$ in the unitary Hurwitz group $\mathcal{U}({\mathfrak{H}})$ which is a rotation of angle $2\pi/3$ around the 2-dimensional hyperbolic plane $\{s+vt \in \Hy: s, t>0 \}$. Then they are isometric to each other. The three square pyramids $Q_\ii, Q_\jj,Q_\kk$ are also permuted by $D_v \in \mathcal{U}({\mathfrak{H}})$ therefore they are isometric to each other. However $P_\mathbf{u}$ is not isometric to $Q_\mathbf{u}$ for $\mathbf{u}=\ii,\jj,\kk$.

\noi Let $\mathcal{Q}=\mathcal{C}_1 \cup D_\ii(P_\ii)\cup D_\jj(P_\jj)\cup D_\kk(P_\kk)$ be the convex hyperbolic 3-dimensional polyhedron contained in $ \Pi$ with set of vertices the eleven points  which are the eight vertices of $\mathcal{C}_1$ and $\{D_\ii(v_2), D_\jj(v_2), D_\kk(v_2)\}$ where $v_2$ is the barycenter of $\mathcal{C}_2$. $\mathcal{Q}$ has 12 triangular faces and 3 square faces. 

\noi Let $\hat{\mathcal{R}} =\mathcal{Q} \cup \tau_\ii D_\ii(Q_\ii)\cup \tau_\jj  D_\jj(Q_\jj)\cup \tau_\kk  D_\kk(Q_\kk)$ be the non convex hyperbolic 3-dimensional polyhedron

\begin{figure}
\centering
\includegraphics[scale=0.33]{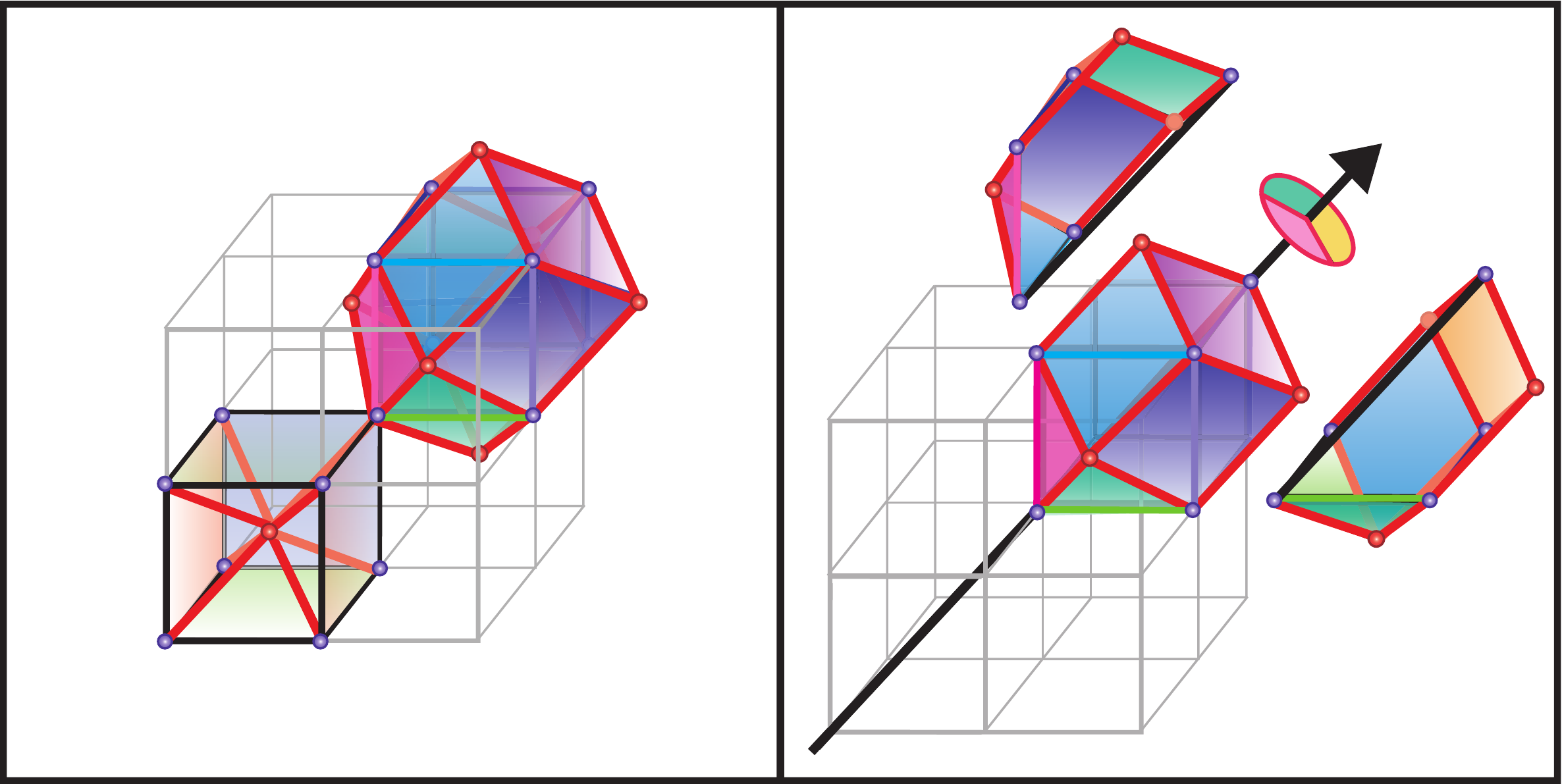}
\begin{center}
{{\bf Figure 6.} A fundamental domain $\mathcal{R}_{\mathfrak{L}}$ for $PSL(2,\mathfrak{L})$ can be taken to be the pyramid over the rhombic dodecahedron $\mathcal{R}$ with apex the point at infinity.  A fundamental domain $\mathcal{R}_{\mathfrak{H}}$ for $PSL(2,\mathfrak{H})$ can be taken to be the pyramid over a third part of the rhombic dodecahedron $\mathcal{R}$.}
\end{center}
\end{figure}

\noi The pyramids $D_\ii(P_\ii), D_\jj(P_\jj), D_\kk(P_\kk)$ are in $\Pi$ and the other three pyramids $\tau_\ii D_\ii(Q_\ii), \tau_\jj  D_\jj(Q_\jj), \tau_\kk  D_\kk(Q_\kk)$ are in $\{ \q\in \mathbf{H}_{\mathbb{H}}^1 : |\q|>1\}$. They form a dihedral  angle of $2\pi/3$ with the ridges of the three square faces in $\mathcal{Q}$.
 
\noi The hyperbolic polyhedron $\hat{\mathcal{R}}$ has 24 triangular faces, three square ridges faces and 14 vertices (the 8 vertices of the cube $\mathcal{C}_1$ and six vertices for each apex of the pyramids of $\mathcal{C}_2$). The convex hull of these six vertices is a non regular octahedron which is combinatorially equal to the dual polyhedron of the cube $\mathcal{C}_1$. The polyhedron $\hat{\mathcal{R}}$ is decomposed in twelve square pyramids and this decomposition is combinatorially equivalent to the decomposition of the rhombic dodecahedron. See figures 6 and 7. 

\noi The set of imaginary parts of the set of vertices of $\mathcal{C}_1$ is the set 
$$
V_{\mathcal{C}_1}:= \left\{ 0, \frac{\ii}{2},\frac{\jj}{2},\frac{\kk}{2}, \frac{\ii+\jj}{2},\frac{\jj+\kk}{2},\frac{\ii+\kk}{2},\frac{\ii+\jj+\kk}{2} \right\}.$$

\noi The set of imaginary parts of the set of vertices of $\hat{\mathcal{R}}$ is the set 
$$
V_{\hat{\mathcal{R}}}:=V_{\mathcal{C}_1}\cup V_{O}
$$
\noi where
$$
V_O:= \left\{  \frac{-\ii+\jj+\kk}{4}, \frac{\ii-\jj+\kk}{4}, \frac{\ii+\jj-\kk}{4}, \frac{3\ii+\jj+\kk}{4},\frac{\ii+3\jj+\kk}{4},\frac{\ii+\jj+3\kk}{4} \right\}.
$$

\noi We see that the convex hulls of $V_{\hat{\mathcal{R}}}, V_{\mathcal{C}_1}$ and $V_{O}$ are Euclidean polyhedra which are a rhombic dodecahedron, a regular cube and a regular octahedron, respectively.  

\begin{definition} Let $\mathcal{R}_{\mathfrak{L}}$ be the  4-dimensional non-compact polytope in $\Hy$ which is obtained by deleting the point at infinity from the convex hull of the point at infinity and the fourteen points which are the vertices of $\hat{\mathcal{R}}$.  
 \end{definition}

\noi The hyperbolic polytope $\mathcal{R}_{\mathfrak{L}}$ is an ideal cone over the non convex polyhedron $\hat{\mathcal{R}}$ with apex the point at infinity. From all the previous results and remarks it follows the following:

\begin{proposition}  $\mathcal{R}_{\mathfrak{L}}$ is a convex fundamental domain for $PSL(2,\mathfrak{L})$.
\end{proposition}

\noi Let $\mathcal{R}^{1/3}\subset \hat{\mathcal{R}}$ be the hyperbolic polyhedron which is the intersection of $\hat{\mathcal{R}}$ with two hyperbolic half-spaces determined by two 3-dimensional hyperbolic hyperplanes which intersect at the plane generated by $v$ and the real positive axis with dihedral angle $2\pi/3$. The non convex polyhedron $\mathcal{R}^{1/3}$ has seven faces: 3 rhombus, 2 triangles and 2 trapezoids. 

\noi The vertices of $\mathcal{R}^{1/3}$ are the following 8 points 
$$\frac{1}{2}(1+ \mathbf{i} + \mathbf{j}+ \mathbf{k}),\sqrt{ \frac{7}{16}}+ \frac{3}{4}\mathbf{i}, \sqrt{ \frac{7}{16}}+ \frac{3}{4}\mathbf{j},  \sqrt{ \frac{1}{2}}+ \frac{1}{2}(\mathbf{i}+\mathbf{j}), \sqrt{ \frac{3}{4}}+ \frac{1}{2}\mathbf{i}, \sqrt{ \frac{3}{4}}+ \frac{1}{2}\mathbf{j}, \sqrt{ \frac{7}{8}}+ \frac{1}{4}(\mathbf{i}+\mathbf{j}), 1.$$

\noi We have that $\hat{\mathcal{R}}= \mathcal{R}^{1/3}\cup  D_{v} (\mathcal{R}^{1/3})  \cup D_{v}^2 (\mathcal{R}^{1/3})$. Hence the hyperbolic volume of $\mathcal{R}^{1/3}$ is a third part of the volume of $\hat{\mathcal{R}}$.

\begin{definition} Let $\mathcal{R}_{\mathfrak{H}}$ the 4-dimensional non-compact polytope in $\Hy$ which is obtained by deleting the point at infinity from the convex hull of the point at infinity and the eight points which are the vertices of $\mathcal{R}^{1/3}$.  
 \end{definition}

\noi $\mathcal{R}_{\mathfrak{H}}$ is an ideal cone over the non convex polyhedron $\mathcal{R}^{1/3}$ with apex the point at infinity. From all the previous results and remarks it follows the following:

\begin{proposition}  $\mathcal{R}_{\mathfrak{H}}$ is a convex fundamental domain for $PSL(2,\mathfrak{H})$.
\end{proposition}

\begin{figure}
\centering
\includegraphics[scale=0.33]{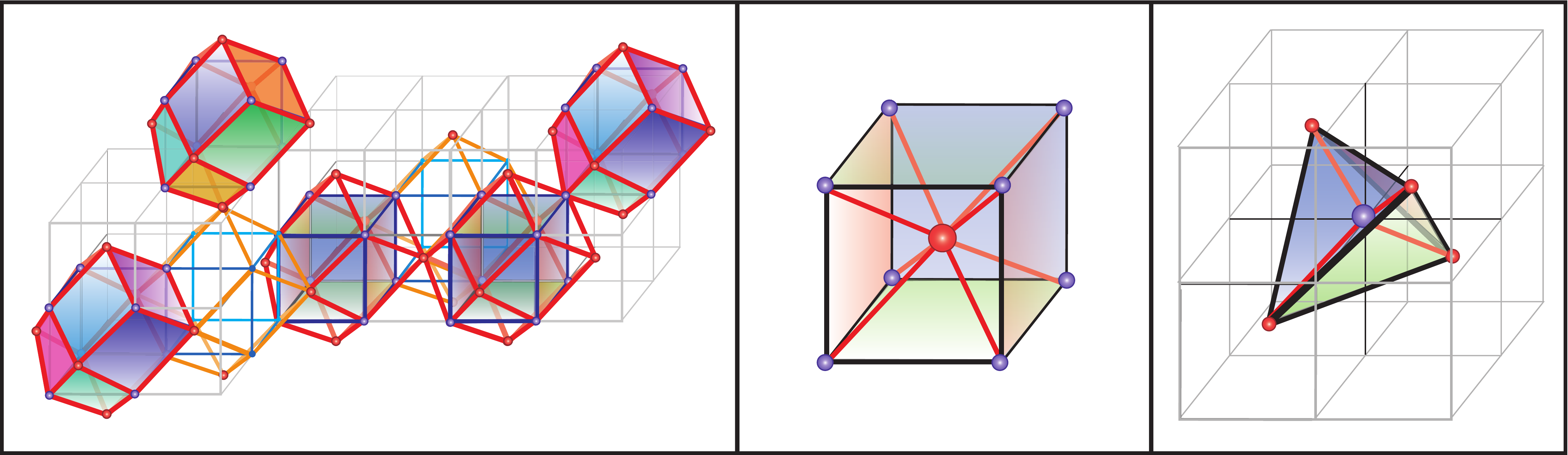}
\begin{center}
{{\bf Figure 7.} The uniform tessellation of the Euclidean 3-space with rhombic dodecahedra. The link of the vertices of the tessellation are of two types: a cube and a tetrahedron.}\end{center}
\end{figure}

\subsection{Geometric characterization of the quaternionic modular groups}

\noi The following fundamental theorem gives the description of the Lipschitz quaternionic modular group $PSL(2,\mathfrak{L})$ as the group of quaternionic M\"obius transformations whose entries are Lipschitz integers and which also satisfy (BG) conditions.

\begin{theorem}
Any element in $PSL(2,\mathbb{H}(\mathbb{Z}))$ which satisfies (BG)
conditions belongs to the quaternionic modular group $PSL(2,\mathfrak{L})$.
\end{theorem}

\begin{proof} Let $A\in{PSL(2,\mathbb{H}(\mathbb{Z}))}$ satisfy (BG) conditions. Let $q=A(1)$ and $S\in{PSL(2,\mathfrak{L})}$ be such that
$p:=S(q)\in\mathcal{P}$. Then $(S {A})(1)=p$ and by \ref{Domain} it follows that $S {A}\in{\mathcal{A}(\mathfrak{L})}$.
Hence $A\in{\mathcal{A}(\mathfrak{L})}\subset{PSL(2,\mathfrak{L})}$.

\end{proof}
 \noi This theorem completely characterizes the group of M\"obius transformations with entries in the Lipschitz integers which preserve the hyperbolic half-space $\mathbf{H}^1_{\mathbb{H}}$.
 
 \noi This proof can be adapted \emph{verbatim} to prove the following 
 Theorem which characterizes $PSL(2,\mathfrak{H})$:

\begin{theorem}
Any element in $PSL(2,\mathbb{H})$ which satisfies (BG)
conditions with entries in the Hurwitz integers $\mathfrak{H}$ belongs to the quaternionic modular group $PSL(2,\mathfrak{H})$.
\end{theorem}

\section{Coxeter decomposition of the fundamental domains.}
\label{COXDEC}

\noi The polytope $\mathcal{P}$ is a \textit{Coxeter
  polytope} i.e. the angles between its faces called \emph{dihedral
  angles} are submultiples of $\pi$. The geometry of the hyperbolic
tessellation of $\Hy$ that is generated by reflections
on the sides of $\mathcal{P}$ is codified by these angles. We denote
this tessellation of $\mathbf{H}_{\mathbb{H}}^1$ by $\mathbf{Y}$. In order to understand it we
will consider another tessellation of $\mathbf{H}_{\mathbb{H}}^1$, which is a refinement based on a
barycentric decomposition of $\mathbf{Y}$ and whose cells all are isometric to a fixed hyperbolic 4-simplex which we denote by $\Delta_{\mathfrak{L}}$. This model simplex $\Delta_{\mathfrak{L}}$ is a Coxeter simplex with one ideal vertex. We denote by $\mathbf{Y}_{\Delta_{\mathfrak{L}}}$ this refined tessellation of $\mathbf{Y}$. We denote by $\mathbf{Y}_{ \mathfrak{L}}$ and $\mathbf{Y}_{ \mathfrak{H}}$, 
the tessellations of $\mathbf{H}_{\mathbb{H}}^1$ whose cells are congruent copies of the fundamental domains $\mathcal{P}_{\mathfrak{L}}$ and  $\mathcal{P}_{\mathfrak{H}}$ of $PSL(2, \mathfrak{L})$ and $PSL(2, \mathfrak{L})$, respectively.  The tessellation $\mathbf{Y}_{ \Delta_{\mathfrak{L}}}$ is a  refinement tessellation of $\mathbf{Y}_{ \mathfrak{H}}$ which in turn is a refinement tessellation of $\mathbf{Y}_{ \mathfrak{L}}$ so that both are refinements of $\mathbf{Y}$. 

\noi It is important for us to describe the groups $PSL(2,
\mathfrak{L})$ and  $PSL(2,
\mathfrak{H})$ as Coxeter subgroups of rotations of the symmetries of the
tessellation generated by hyperbolic reflections of $\Delta_{\mathfrak{L}}$. 

\noi Let us recall that a {\it complete flag} of the 4-dimensional
polytope $\mathcal{P}$ consists of a sequence $(v,e,r, f)$ of $k-$faces of $\mathcal{P}$, where $v$ is a vertex of
the edge $e$, in turn $e$ is an edge of the 2-dimensional face $r$ (a ridge) and finally $r$ 
is a 2-dimensional face of the 3-dimensional face $f$ of the polytope $\mathcal{P}$.

\noi Let us consider a complete flag of the pyramid $\mathcal{P}$ in the half-space
model of $\mathbf{H}^{1}_{\mathbb{H}}$ with the ideal vertex at
$\infty$ and with base the cube $\mathcal{C}$. Let $A,B,C,D$ be the barycenters of the $k-$faces of a flag in $\mathcal{C}$; $\it{i. e.}$ the barycenter of the cube $\mathcal{C}$, one of its six
square faces, one of its four incident edges and one of its vertices,
respectively. We chose, for example, $A=1,B=\frac{1}{2}(\sqrt{3}+\mathbf{i}),
C=\frac{1}{\sqrt{2}}+\frac 12(\mathbf{i+j})$ y $D=\frac{1}{2}(1+\mathbf{i+j+k})$, as in the figure 8.  

\noi Let
$\Delta_{\mathfrak{L}}$ be the non-compact hyperbolic 4-simplex whose five vertices are
$A,B,C,D$ and $\infty$. Then $\Delta_{\mathfrak{L}}$ has five hyperbolic tetrahedral faces:
one compact tetrahedron with vertices A,B,C,D, and four tetrahedra with one ideal vertex at
$\infty$. Moreover $\Delta_{\mathfrak{L}}$ has 10 triangles (4 compact and
6 with one ideal vertex) and 10 edges (6 compact and four with one
ideal vertex). The dihedral angles of the five 3-dimensional faces of $\Delta_{\mathfrak{L}}$ meet in pairs in the 
10  two dimensional faces. 

\noi The orthogonal projection of $\Delta_{\mathfrak{L}}$ on the ideal boundary of $\mathbf{H}^{1}_{\mathbb{H}}$ 
is a 3-dimensional Euclidean tetrahedron that we denote by
$\Delta_{\mathfrak{L}}^{e}$. The ideal triangles of $\Delta_{\mathfrak{L}}$
which are asymptotic at $\infty$ have the same angles that the corresponding
dihedral angles of the edges of $\Delta_{\mathfrak{L}}^{e}$. The 3-dimensional faces of $\Delta_{\mathfrak{L}}$ which meet in compact regular 2-faces
which have the origin as a vertex have dihedral angle $\pi/2$ because the hyperplanes in
$\Delta_{\mathfrak{L}}$ which are asymptotic at $\infty$ and contain 0 are orthogonal to the
ideal boundary $\partial \mathbf{H}^{1}_{\mathbb{H}}$ and to the
unitary sphere $\Pi$.

\begin{figure}
\centering
\includegraphics[scale=0.33]{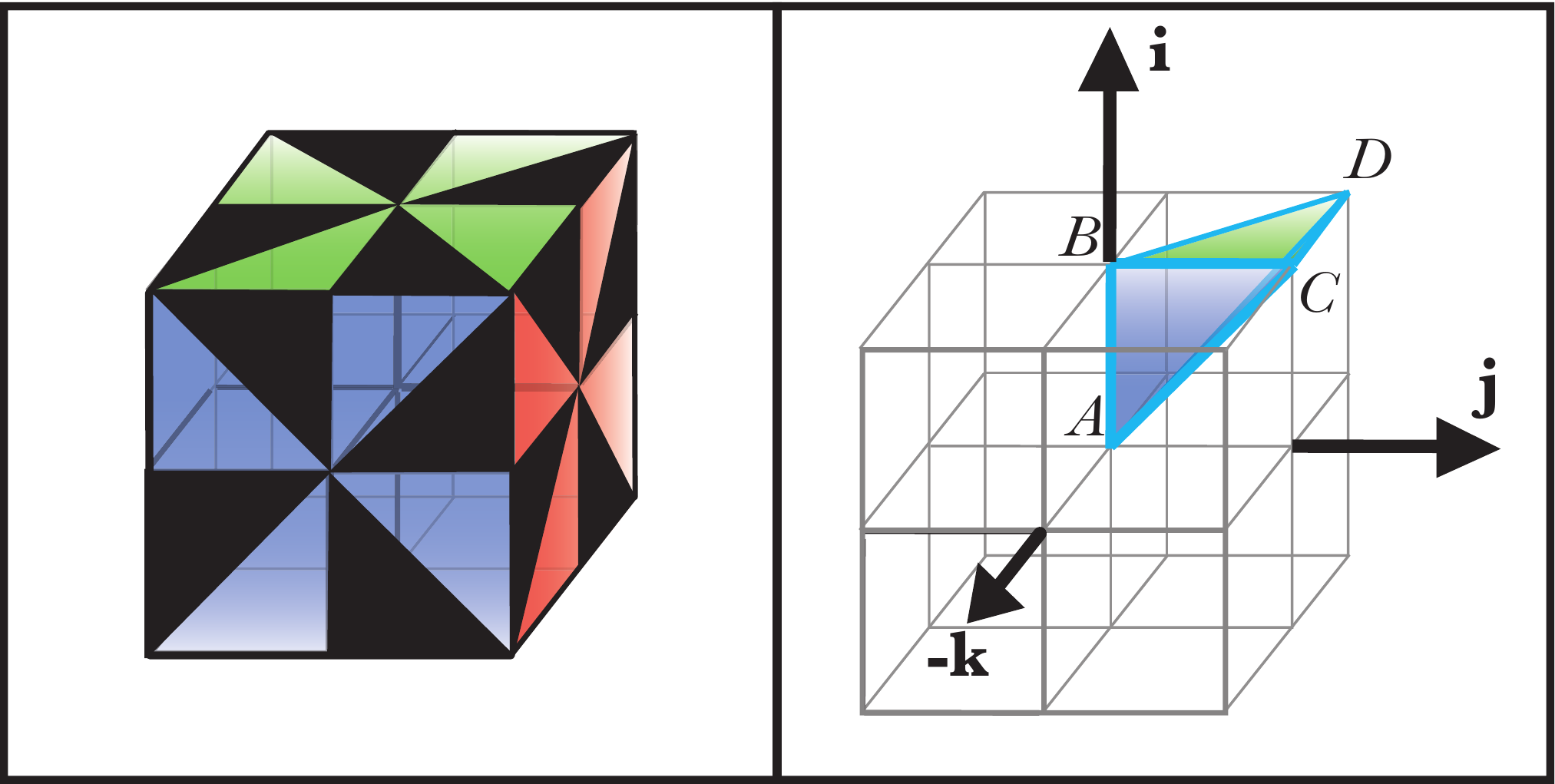}
\begin{center}
{{\bf Figure 8.} The Coxeter decomposition into 48 tetrahedra of a cube in the Euclidean 3-space.}\end{center}
\end{figure}

\noi The Euclidean tetrahedron $\Delta_{\mathfrak{L}}^{e}$ is the standard Coxeter's
3-simplex $\triangle(4,3,4)$, see \cite{C}. The Euclidean tessellation
whose cells are isometric copies of the tetrahedron with Coxeter symbol $\triangle(4,3,4)$ is the refinament obtained by
means of the barycentric subdivision of the classic tessellation by cubes
of the Euclidean 3-space whose Schl\"afli symbol is $\{4,3,4\}$. Each cube is divided into 48 tetrahedra of type $\triangle(4,3,4)$. The Schl\"afli symbol of a cube is $\{4,3\}$.This symbol means that the faces of a cube are squares with Schl\"afli symbol $\{4\}$ and that the link of each vertex is an equilateral triangle with Schl\"afli symbol $\{3\}$. The symbol of the tessellation  $\{4,3,4\}$ of the Euclidean 3-space means that the 3-dimensional cells are cubes with Schl\"afli symbol $\{4,3\}$ and that the link or verticial figure of each vertex in the tessellation  $\{4,3,4\}$ is an octahedron  with Schl\"afli symbol $\{3,4\}$.

\noi We can compute the dihedral angles of $\triangle(4,3,4)$ by
means of $\{4,3,4\}$. The dihedral angles at the edges $AC, CB$ and $DB$
are $\pi/2$ because the incident tetrahedra at these triangles are
orthogonal. The dihedral angles at $AB$ and $CD$ is $\pi/4$ because
there are eight tetrahedra around these edges in $\{4,3,4\}$. Finally,
the dihedral angle at $AD$ is $\pi/3$ because there are six tetrahedra
with a common vertex $D$ of the cube, two for each square at the
corner of the cube $\mathcal{C}$.

\noi The dihedral angle of $\Delta_{\mathfrak{L}}$ at the triangle $BCD$ is the angle subtended
by the hyperplane $\Pi$ and the hyperplane  $\Pi_{\frac{\ii}{2}}$. The intersection of the cells of the tessellation $\mathbf{Y}_{\mathfrak{L}}$ with the 2-dimensional hyperbolic plane $\{\q= x_0+x_1\ii \in \Hy \}$ is the classic tessellation of the fundamental domains of the action of the 
2-dimensional modular group $PSL(2,\mathbb{Z})$ on a hyperbolic plane. The invariant tesellation of the fundamental domains of $PSL(2,\mathbb{Z})$ is a regular tessellation whose tiles are congruent copies of a triangle with one vertex at the point at infinity and two vertices with angles of $\pi/3$. Then we can see
that the dihedral angle at $BCD$ is $\pi/3$. See the figure 14.

\noi In short, we list the 10 dihedral angles of $\Delta_{\mathfrak{L}}$:

$$\begin{tabular}{ccccc} $\angle BCD=\pi/3$, & $\angle
  AC\infty=\pi/2$, & $\angle BC\infty=\pi/2$, \\ $\angle
  AB\infty=\pi/4$, & $\angle BD\infty=\pi/2$, & $\angle
  ABC=\pi/2$, \\$\angle AD\infty=\pi/3$, & $\angle ABD=\pi/2$, &
  $\angle ACD=\pi/2$, \\$\angle CD\infty=\pi/4.$ & & \end{tabular}$$

\noi In order to understand what is the Coxeter group corresponding to
the symmetries of $\mathbf{Y}_{\Delta_{\mathfrak{L}}}$ we need to determine the diagram of the
Coxeter group. This is a graph with labeled edges which decode the information of $\mathbf{Y}_{\Delta_{\mathfrak{L}}}$ by means of the dihedral angles of $\Delta_{\mathfrak{L}}$. For this purpose, we consider the complete graph with 5 vertices and 10 edges with vertices corresponding to the five 3-faces 
$f_1,\dots,f_5,$ of the 4-dimensional simplex $\Delta_{\mathfrak{L}}$ and edges to the intersection of two faces in $\Delta_{\mathfrak{L}}$. Given an edge of this graph with end points $f_n$ and $f_m$, $n\neq m$, we
label this edge with the natural number $p$ if the dihedral angle between the faces $f_n$ and $f_m$ in $\Delta_{\mathfrak{L}}$ is $2\pi/p$.

\noi If we eliminate the edges labeled with
the number 2, we have a path graph with 5 vertices with the numbered
labels 3, 4, 3, 4; see \cite{C}. Therefore we can identify $\Delta_{\mathfrak{L}}$ with the Coxeter 4-simplex
$\triangle(3,4,3,4)$.

\begin{definition} Let $\Gamma_{\{3,4,3,4\}}$ be the hyperbolic Coxeter group generated by reflections on the sides of $\Delta_{\mathfrak{L}}$. This group is a hyperbolic Kleinian group.

\end{definition}

\noi The union of the $6 \times 8=48$ simplexes asymptotic at $\infty$ and isometric to $\Delta_{\mathfrak{L}}$ with bases in the cube $\mathcal{C}$ is
$\mathcal{P}$. The Lipschitz fundamental domain $\mathcal{P}_{\mathfrak{L}}$ is obtained as the union of $6 \times 2=12$ simplexes asymptotic at $\infty$ and isometric to $\Delta_{\mathfrak{L}}$ with bases in the two cubes $\mathcal{C}_1$ and $\mathcal{C}_2$. 

\noi The Hurwitz modular domain $\mathcal{P}_{\mathfrak{H}}$ is obtained as the union of $4$ simplexes asymptotic at $\infty$ and isometric to $\Delta_{\mathfrak{L}}$ since $PSL(2,\mathfrak{L})$ is a subgroup of index 3 of $PSL(2,\mathfrak{H})$. 

\noi Finally, applying $24\times 48=$ 1152 elements of the group $\Gamma_{\{3,4,3,4\}}$ to $\Delta_{\mathfrak{L}}$ we obtain an union of isometric copies of $\mathcal{P}$ that forms a
24-cell which is a cell of the regular hyperbolic honeycomb $\{3,4,3,4\}$. See figure 13. There are 24 octahedra in a 24-cell and there are 48 simplexes congruent to $\Delta_{\mathfrak{L}}$ in each one octahedron. The non-compact, right-angled hyperbolic 24-cell has finite volume equal to
$4\pi^{2}/3$, (see \cite{RT}).  

\noi We can summarize these previous results in the following theorem

\begin{theorem} The Coxeter group $\Gamma_{\{3,4,3,4\}}$ of finite covolume contains as subgroups the quaternionic modular groups $PSL(2,\mathfrak{L})$ and $PSL(2,\mathfrak{H})$. We have $PSL(2,\mathfrak{L}) \subset PSL(2,\mathfrak{H}) \subset \Gamma_{\{3,4,3,4\}}$. 

\noi We have the following indices:
\begin{itemize}
\item $[\Gamma_{\{3,4,3,4\}} : PSL(2,\mathfrak{H})]=4$, 
\item $[\Gamma_{\{3,4,3,4\}}:PSL(2,\mathfrak{L})]=12$ and
\item  $[PSL(2,\mathfrak{H}):PSL(2,\mathfrak{L})]=3$.
\end{itemize}

\end{theorem}

\subsection{Volumes of the fundamental domains} The Coxeter decomposition of the 24-cell implies that the group of symmetries (both orientation-preserving and orientation-reversing) of the 24-cell is of order $24\times48=1152$. With the action of these 1152 symmetries the 
24-cell can be divided into 1152 congruent simplexes where each of them is congruent to a $\Delta_{\mathfrak{L}}$. One knows from \cite{RT} that the volume of the hyperbolic right-angled 24-cell is $4\pi^{2}/3$, therefore the volume of $\Delta_{\mathfrak{L}}$ is $(4\pi^{2}/3)$ divided by 1152, this is $(\pi^{2}/864)$. 
\noi Then, we have the following proposition

\begin{proposition} The volume of $\mathcal{P}_{\mathfrak{L}}$ is $12(\pi^{2}/864)=\pi^{2}/72$ and  the volume of $\mathcal{P}_{\mathfrak{H}}$ is $4(\pi^{2}/864)=\pi^{2}/216$.
\end{proposition}
         
\section{Algebraic presentations of the quaternionic modular groups.}

\begin{figure}
\centering
\includegraphics[scale=0.25]{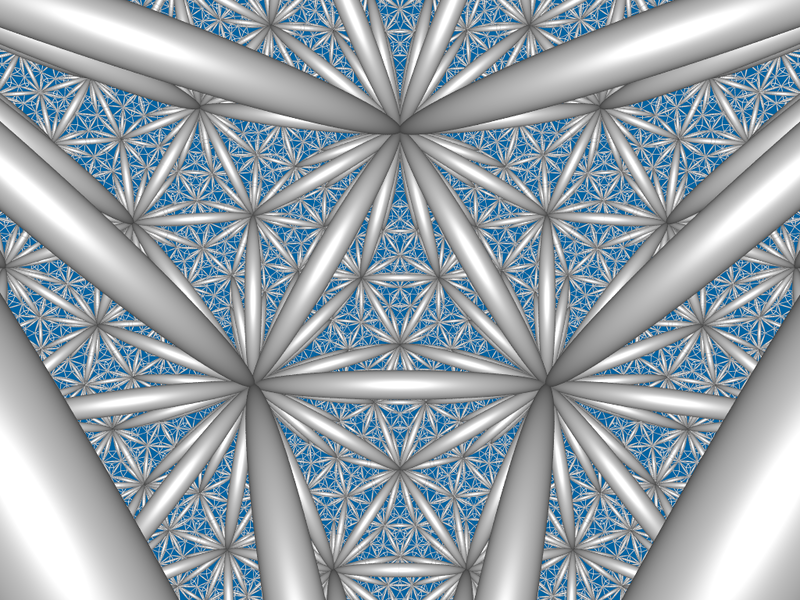}
\begin{center}
{{\bf Figure 9.} Schematic 3-dimensional version: the hyperbolic tessellation $\{3,4,4\}$ . 

This figure is courtesy of Roice Nelson \cite{RN} } 
\end{center}
\end{figure}

\begin{figure}
\centering
\includegraphics[scale=0.2]{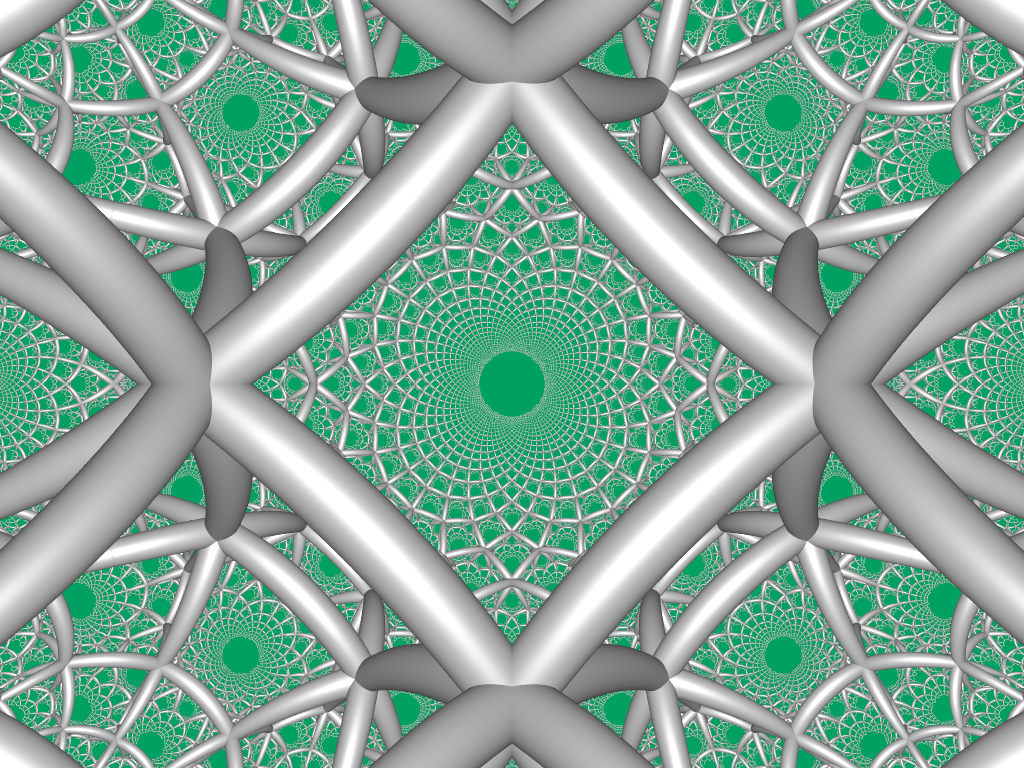}
\begin{center}
{{\bf Figure 10.} Schematic 3-dimensional version: the hyperbolic tessellation $\{4,4,3\}$\\ that its dual to $\{3,4,4\}$. 
This figure is courtesy of Roice Nelson \cite{RN}} .
\end{center}
\end{figure}

\noi We now give algebraic descriptions (following geometric approaches) of the quaternionic modular groups $PSL(2,\mathfrak{L})$ and $PSL(2,\mathfrak{H})$ in terms of  finite presentations by means of generators and relations.
For this purpose, we use the Cayley graphs of the groups that we construct from the action of the generators of $PSL(2,\mathfrak{L})$ and $PSL(2,\mathfrak{H})$ on the orbit of a regular point in the fundamental domain of $PSL(2,\mathfrak{L})$ and $PSL(2,\mathfrak{H})$. Each Cayley graph is a spiderweb of the 4-dimensional hyperbolic space and we can read the relations of the groups by considering reduced (or minimal) cycles of Cayley graph. We obtain the following presentations of the quaternionic modular groups.
      
\subsection{Algebraic presentation of the Lipschitz modular group $PSL(2,\mathfrak{L})$.}
   
We begin with the abstract presentation of the Lipschitz modular group $PSL(2,\mathfrak{L})$ generated by the three translations $\tau_{\mathbf{i}}, \tau_{\mathbf{j}}, \tau_{\mathbf{k}}$ and the inversion $T$.     
         
 \begin{theorem}\label{apL} The group $PSL(2,\mathfrak{L})$ 
 has the following finite presentation:
$$
PSL(2,\mathfrak{L})=
\left< T, \tau_{\mathbf{i}}, \tau_{\mathbf{j}}, \tau_{\mathbf{k}}\,  | \, \mathfrak{R}_{\mathfrak{L}} \right>,
$$ 
\noi where $\mathfrak{R}_{\mathfrak{L}}$ is the set of relations:

\begin{equation*}\label{R}
\mathrm{\mathfrak{R}_{\mathfrak{L}}:=}\quad\left\{\begin{array}{l}
 
  T^{2},\ [\tau_{\mathbf{i}}:\tau_{\mathbf{j}}],\ [\tau_{\mathbf{i}}:\tau_{\mathbf{k}}],
\ [\tau_{\mathbf{k}}:
\tau_{\mathbf{j}}],\\ \\ 
  \ (\tau_{\mathbf{i}}T)^{6},
\ (\tau_{\mathbf{j}}T)^{6},
\ (\tau_{\mathbf{k}}T)^{6}, \\ \\
  \ (\tau_{\mathbf{i}}\tau_{\mathbf{j}}T)^{4},
\ (\tau_{\mathbf{i}}\tau_{\mathbf{k}}T)^{4},
\ (\tau_{\mathbf{j}}\tau_{\mathbf{k}}T)^{4},\\ \\
\ (\tau_{\mathbf{i}}\tau_{\mathbf{j}}\tau_{\mathbf{k}}T)^{6},
\\ \\
  \ (\tau_{\mathbf{i}}T)^{3}\ (\tau_{\mathbf{j}}T)^{3}\ (\tau_{\mathbf{k}}T)^{3},
\ (\tau_{\mathbf{i}}T)^{3}\ (\tau_{\mathbf{k}}T)^{3}\ (\tau_{\mathbf{j}}T)^{3},
\\ \\
\ [(\tau_{\mathbf{i}}T)^{3}:T],[(\tau_{\mathbf{j}}T)^{3}:T],[(\tau_{\mathbf{k}}T)^{3}:T],
\\ \\
\ [(\tau_{\mathbf{i}}T)^{3}:\tau_{\mathbf{i}}],[(\tau_{\mathbf{j}}T)^{3}:\tau_{\mathbf{j}}],[(\tau_{\mathbf{k}}T)^{3}:\tau_{\mathbf{k}}],
\\ \\
\ (\tau_{\mathbf{u}}T)^{3}\tau_{\mathbf{w}}(\tau_{\mathbf{u}}T)^{3} \tau_{\mathbf{w}}, where \, \, \mathbf{u} \neq \mathbf{w} \, \, are\, \, units\, \, in\,\, the\,\, set \, \, \{\mathbf{i,j,k}\}.
 \end{array}
\right.
\end{equation*}

 \end{theorem}

\noi In this theorem $[\cdot:\cdot]$ denotes the commutator relation.

 \begin{proof}
 \noi We denote by $G_{\mathfrak{L}}$ the Cayley directed  graph of the action of $PSL(2,\mathfrak{L})$ on $\mathbf{H}_{\mathbb{H}}^1$ generated by the three translations $\tau_{\mathbf{i}}, \tau_{\mathbf{j}}, \tau_{\mathbf{k}}$ and the inversion $T$. For each element in $PSL(2,\mathfrak{L})$ we have exactly one vertex in the Cayley graph $G_{\mathfrak{L}}$. Then there is a bijective correspondence between the set of vertices of $G_{\mathfrak{L}}$ and the orbit of a point $\q$ in the interior of a cell congruent 
to $\mathcal{P}_{\mathfrak{L}}$ in the tessellation $Y_{\mathfrak{L}}$  (see the section \ref{COXDEC}). For example, the point $\q=\frac{1}{3}(\sqrt{3}+\ii+\jj+\kk)\in \mathcal{P}_{\mathfrak{H}}\subset \mathcal{P}_{\mathfrak{L}}$ is a regular point for the orbits of $PSL(2,\mathfrak{L})$ and $PSL(2,\mathfrak{H})$.
 
 \noi The edges of the Cayley graph $G_{\mathfrak{L}}$ are the oriented segments that join a vertex corresponding with
  an element $\gamma \in PSL(2,\mathfrak{L})$  with the vertex corresponding with the
  element $g  \gamma$, for each generator $g$ of $PSL(2,\mathfrak{L})$. 
  
\noi For each
  vertex in $G_{\mathfrak{L}}$ we have two oriented edges corresponding to each generator and its inverse. Then at each vertex of $G_{\mathfrak{L}}$ there are the same number of oriented edges that start and end at the vertex. Therefore     $G_{\mathfrak{L}}$ is a regular digraph. Moreover, the number of edges with a common vertex is twice the number of generators which, in our case, is 8.
 
  \noi As we have shown, the fundamental domain $P_{\mathfrak{L}}$ of the action of $PSL(2,\mathfrak{L})$ on $\mathbf{H}_{\mathbb{H}}^1$ is a
  pyramid with the apex at the infinity and base a pair of cubes $\mathcal{C}_1$ and $\mathcal{C}_2$ that are symmetric with respect to the center of the  cube $\mathcal{C}$.  We obtain that for each 3-dimensional face of $P_{\mathfrak{L}}$ there
  is an oriented edge of $G_{\mathfrak{L}}$. In fact, $G_{\mathfrak{L}}$ is the
  1-skeleton of the dual tessellation (in the Voronoi sense) in $\Hy$ of the tessellation $\mathbf{Y}_{\mathfrak{L}}$ whose cells are fundamental domains of $PSL(2,\mathfrak{L})$. 
  
 \noi  The dual tessellation of
  $\{3,4,3,4\}$ is the regular tessellation whose Schl\"afli symbol is
  $\{4,3,4,3\}$. Geometrically the cells of $\mathbf{Y}$ are horoballs whose boundary
  are 3-dimensional horospheres tessellated by the classic tessellation by cubes of $\mathbb{R}^3$ 
  and whose Schl\"afli symbol is $\{4,3,4\}$. The angle subtended at each pair of adjacent horoballs in a squared face of each cube is $\pi /6$.

\begin{figure}
\centering
\includegraphics[scale=0.33]{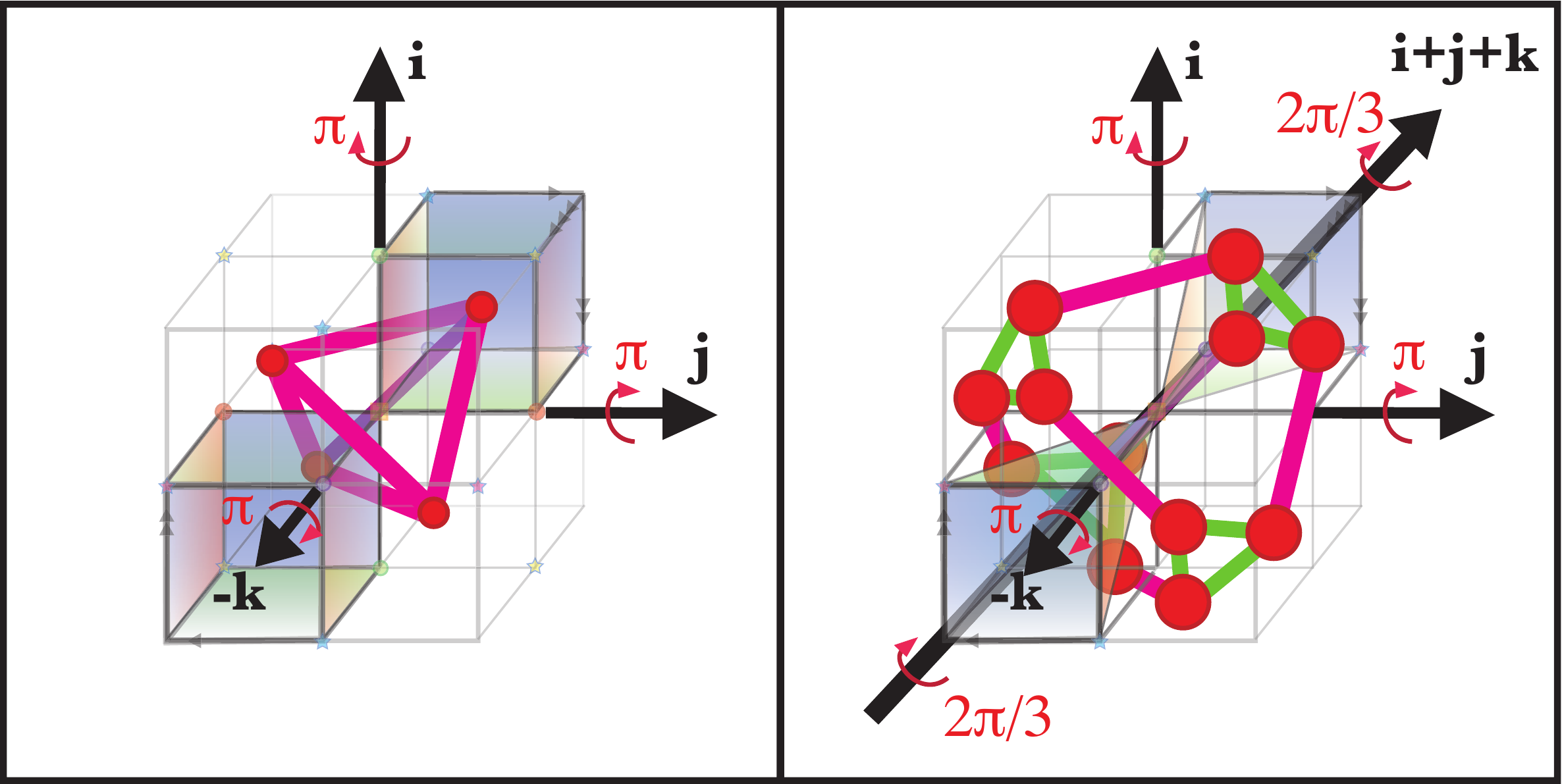}
\begin{center}
{{\bf Figure 11.} The unitary groups in the Cayley graphs of the quaternionic modular groups $PSL(2,\mathfrak{L})$ and $PSL(2,\mathfrak{H})$. The edges in red correspond to elements of order two and the edges in green correspond to elements of order three.}
\end{center}
\end{figure}

  \noi The generator $T$ is an involution; then it satisfies the relation
  $T^{2}=1$.
  
  \noi There are 3 free generators of the parabolic subgroup $\mathcal{T}_{\Im \mathbb{Z}(\mathbb{H})}$ of $PSL(2,\mathfrak{L})$:
  $\tau_{\mathbf{i}}, \tau_{\mathbf{j}}$ and $\tau_{\mathbf{k}}$ and they satisfy 3 relations:

$$[\tau_{\mathbf{i}}:\tau_{\mathbf{j}}]=\tau_{\mathbf{i}}\tau_{\mathbf{j}}\tau_{\mathbf{i}}^{-1}\tau_\mathbf{j}^{-1}=1,$$
$$ [\tau_{\mathbf{i}}:\tau_{\mathbf{k}}]=\tau_\mathbf{i}\tau_{\mathbf{k}}\tau_{\mathbf{i}}^{-1}\tau_{\mathbf{k}}^{-1}=1,$$
$$ [\tau_{\mathbf{k}}: \tau_{\mathbf{j}}]=\tau_{\mathbf{k}}\tau_{\mathbf{j}}\tau_{\mathbf{k}}^{-1}\tau_{\mathbf{j}}^{-1}=1.$$ 

\noi This group is isomorphic to $\mathbb{Z}\oplus\mathbb{Z}\oplus\mathbb{Z}$. The Cayley graph of $\mathcal{T}_{\Im \mathbb{Z}(\mathbb{H})}$ is the 1-skeleton of the classical regular cubic tessellation $\{4,3,4\}$ of $\mathbb{R}^3$ that we identify with the horosphere centered at the point at infinity $\Re(\q)=c$.
\noi Each of these relations is a cycle in $G_{\mathfrak{L}}$ of order
4. Geometrically these cycles are squares around the triangles that
are the 2-faces of the cells that all are congruent to the 24-cell $\{3,4,3\}$ of the
tessellation $\{3,4,3,4\}$. Unions of these one dimensional squares form
the 1-skeleton of cubes $\{4,3\}$ in $G_{\mathfrak{L}}$. Each cube is the link in $\{3,4,3\}$
of an ideal vertex of $\{3,4,3,4\}$ and is joined to another cube by 
the involution $T$. Then, in $G_{\mathfrak{L}}$ we have copies as non-labeled subgraphs
of the 1-skeleton of $\{4,3,4\}$ for each ideal vertex of
$\{3,4,3,4\}$.

\noi The Cayley graph $G_{\mathfrak{L}}$ is the 1-skeleton of the
tessellation which is dual to the tessellation  $\mathbf{Y}_{\mathfrak{L}}$ of $\mathbf{H}_{\mathbb{H}}^{1}$ whose cells are obtained as subsets of tetrahedrons  by the symmetric subdivision of the cells of $\{3, 4, 3, 4\}$.

\noi The 4-cells associated to $G_{\mathfrak{L}}$ in $\mathbf{H}_{\mathbb{H}}^{1}$ are of many types: the horoballs
whose Euclidean boundary is tessellated by $\{4,3,4\}$ and the compact
hyperbolic polytopes that are obtained from the truncation of the vertices
of $\mathcal{P}_\mathfrak{L}$. The relations of the presentation of $PSL(2,\mathfrak{L})$ are the reduced
cycles on the boundary of these polytopes. This give us seven new relations that involve  the involution $T$
and each one of the free generators $\tau_{\mathbf{i}},\tau_{\mathbf{j}},\tau_{\mathbf{k}}$:
$$ (\tau_{\mathbf{i}}T)^{6}=(\tau_{\mathbf{j}}T)^{6}= (\tau_{\mathbf{k}}T)^{6}=1,$$
$$  (\tau_{\mathbf{i}}\tau_{\mathbf{j}}T)^{4}= (\tau_{\mathbf{i}}\tau_{\mathbf{k}}T)^{4}= (\tau_{\mathbf{j}}\tau_{\mathbf{k}}T)^{4}= (\tau_{\mathbf{i}}\tau_{\mathbf{j}}\tau_{\mathbf{k}}T)^{6}=1.
$$

\noi The first six relations correspond to the cycles which are dodecagons in $G_{\mathfrak{L}}$ and which are obtained by
bitruncation of the triangular 2-faces of the regular tessellation $\{3,4,3,4\}$. They correspond to the isotropy groups of the square faces and the edges of $\mathcal{C}$. The last relation can be read from a polygon with 24 edges which corresponds to the isotropy group of the vertices of $\mathcal{C}$.

\noi Finally we consider the action of the unitary group $\mathcal{U}(\mathfrak{L})$. This group divides the pyramid $\mathcal{P}$ and the 3-dimensional cube $\mathcal{C}$ that is the base of $\mathcal{P}$ in eight pyramids and eight cubes. Then the fundamental domain $\mathcal{P}_{\mathfrak{L}}$ of $PSL(2,\mathfrak{L})$ is the union of two pyramids. We defined the fundamental domain $\mathcal{P}_\mathfrak{L}$ of $PSL(2,\mathfrak{L})$ as two pyramids whose bases are two cubes diametrically opposed by 1. 

\noi The unitary group $\mathcal{U}(\mathfrak{L})$ has order 4. Then we have four points in $\mathcal{P}$ as the orbit of a point in $\mathcal{P}_\mathfrak{L}$ by the action of $\mathcal{U}(\mathfrak{L})$. We have a representation of this group in the Cayley graph $G_{\mathfrak{L}}$ as a tetrahedron whose vertices are the involution elements $(\tau_{\mathbf{i}}  T)^3, (\tau_{\mathbf{j}}  T)^3$ and $(\tau_{\mathbf{k}}  T)^3$ in $PSL(2,\mathfrak{L})$ which correspond to $D_{\mathbf{i}},D_{\mathbf{j}},D_{\mathbf{k}}$ in $\mathcal{U}(\mathfrak{L})$, respectively . We have two more relations from the triangular sides of the tetrahedrons:
$$ \ (\tau_{\mathbf{i}}  T)^{3}\ (\tau_{\mathbf{j}}  T)^{3}\ (\tau_{\mathbf{k}}  T)^{3},
\ (\tau_{\mathbf{i}}  T)^{3}\ (\tau_{\mathbf{k}}  T)^{3}\ (\tau_{\mathbf{j}}  T)^{3}.$$

\noi Straight forward computations show the following six commuting relations:
 $$[D_{\mathbf{i}}:T],[D_{\mathbf{j}}:T],[D_{\mathbf{k}}:T],$$
 $$[D_{\mathbf{i}}:\tau_{\mathbf{i}}],[D_{\mathbf{j}}:\tau_{\mathbf{j}}],[D_{\mathbf{k}}:\tau_{\mathbf{k}}],$$
 and the following six anti-commuting relations:
 $$D_{\mathbf{u}}\tau_{\mathbf{w}}D_{\mathbf{u}} \tau_{\mathbf{w}}=1$$
 where $\mathbf{u}$ and $\mathbf{w}$ are two different units in the set $\{\mathbf{i,j,k}\}$.

\noi In this way we obtain the desired abstract presentation of $PSL(2,\mathfrak{L})$.
\end{proof} 
        
\subsection{Algebraic presentation of the Hurwitz modular group $PSL(2,\mathfrak{H})$.}

We give now an abstract presentation of the Hurwitz modular group $PSL(2,\mathfrak{H})$ generated by three free translations $\tau_i$, one inversion $T$ and the unitary Hurwitz group $\mathcal{U}(\mathfrak{H})$. We remember that $\mathcal{U}(\mathfrak{H})$ is a group of order 12 isomorphic to the tetrahedral group. Let 
$$
\omega_{\mathbf{1}}=\frac{1}{2}(1+\mathbf{i}+\mathbf{j}+\mathbf{k}), \, \, \omega_{\mathbf{i}}=\frac{1}{2}(1+\mathbf{i}-\mathbf{j}-\mathbf{k}), \, \,\omega_{\mathbf{j}}=\frac{1}{2}(1-\mathbf{i}+\mathbf{j}-\mathbf{k}), \, \,\omega_{\mathbf{k}}=\frac{1}{2}(1-\mathbf{i}-\mathbf{j}+\mathbf{k}).$$
\noi We remember that $D_{\mathbf{u}}=\left(\begin{array}{cc} \mathbf{u} & 0 \\ 0 & \mathbf{u}\end{array}\right)$, where $\mathbf{u}\in \mathfrak{H}_{u}$. We obtain the following presentation of $PSL(2,\mathfrak{H})$.

     \begin{theorem}

$$
PSL(2,\mathfrak{H})=
\left< T, \tau_{\mathbf{i}}, \tau_{\mathbf{j}}, \tau_{\mathbf{k}}, D_{\omega_{1}},D_{\omega_{\mathbf{i}}},D_{\omega_{\mathbf{j}}}, D_{\omega_{\mathbf{k}}} \,  | \, \mathfrak{R}_{\mathfrak{H}} \right>,
$$ 
\noi where $\mathfrak{R}_{\mathfrak{H}}$ is the set of relations:
  
  \begin{equation*}\label{R}
\mathrm{\mathfrak{R}_{\mathfrak{H}}:=}\quad\left\{\begin{array}{l}
 
  T^{2},\ [\tau_{\mathbf{i}}:\tau_{\mathbf{j}}],\ [\tau_{\mathbf{i}}:\tau_{\mathbf{k}}],
\ [\tau_{\mathbf{k}}:
\tau_{\mathbf{j}}],\\ \\ 
  \ (\tau_{\mathbf{i}}T)^{6},
\ (\tau_{\mathbf{j}}T)^{6},
\ (\tau_{\mathbf{k}}T)^{6}, \\ \\
  \ (\tau_{\mathbf{i}}\tau_{\mathbf{j}}T)^{4},
\ (\tau_{\mathbf{i}}\tau_{\mathbf{k}}T)^{4},
\ (\tau_{\mathbf{j}}\tau_{\mathbf{k}}T)^{4},\\ \\
\ (\tau_{\mathbf{i}}\tau_{\mathbf{j}}\tau_{\mathbf{k}}T)^{6},
\\ \\
  \ (\tau_{\mathbf{i}}T)^{3}\ (\tau_{\mathbf{j}}T)^{3}\ (\tau_{\mathbf{k}}T)^{3},
\ (\tau_{\mathbf{i}}T)^{3}\ (\tau_{\mathbf{k}}T)^{3}\ (\tau_{\mathbf{j}}T)^{3},
\\ \\
\ [(\tau_{\mathbf{i}}T)^{3}:T],[(\tau_{\mathbf{j}}T)^{3}:T],[(\tau_{\mathbf{k}}T)^{3}:T],
\\ \\
\ [(\tau_{\mathbf{i}}T)^{3}:\tau_{\mathbf{i}}],[(\tau_{\mathbf{j}}T)^{3}:\tau_{\mathbf{j}}],[(\tau_{\mathbf{k}}T)^{3}:\tau_{\mathbf{k}}],
\\ \\
\ D_{\mathbf{u}}\tau_{\mathbf{w}}D_{\mathbf{u}} \tau_{\mathbf{w}}, where \, \, \mathbf{u} \neq \mathbf{w} \, \, are\, \, units\, \, in\,\, the\,\, set \, \, \{\mathbf{i,j,k}\}
\\ \\
\ [D_{\omega_1}:T],[D_{\omega_{\mathbf{i}}}:T],[D_{\omega_{\mathbf{j}}}:T],[D_{\omega_{\mathbf{k}}}:T],
\\ \\
\ (D_{\omega_1})^{3}, (D_{\omega_{\mathbf{i}}})^{3},(D_{\omega_{\mathbf{j}}})^{3},(D_{\omega_{\mathbf{k}}})^{3},
\\ \\
\ D_{\omega_1} D_{\mathbf{i}} D_{\omega_{\mathbf{i}}} D_{\mathbf{j}} D_{\omega_{\mathbf{k}}} D_{\mathbf{k}},
\\ \\
\  D_{\mathbf{k}} D_{\omega_{\mathbf{k}}} D_{\mathbf{i}} D_{\omega_{\mathbf{j}}}^{-1} D_{\mathbf{j}} D_{\omega_1},
\\ \\
\  D_{\mathbf{i}} D_{\omega_{\mathbf{k}}} D_{\mathbf{j}} D_{\omega_{\mathbf{i}}} D_{\mathbf{k}} D_{\omega_{\mathbf{j}}},
\\ \\
\  D_{\mathbf{j}} D_{\omega_{\mathbf{j}}}^{-1} D_{\mathbf{k}} D_{\omega_{\mathbf{i}}} D_{\mathbf{i}} D_{\omega_1}.
 \end{array}
\right.
\end{equation*}

\end{theorem} 

 \begin{proof}

\noi We denote by $G_{\mathfrak{H}}$ the Cayley graph of the action of $PSL(2,\mathfrak{H})$ on $\mathbf{H}_{\mathbb{H}}^1$ generated by three translations $\tau_i$, one inversion $T$ and four matrices in the Hurwitz unitary group $U(\mathfrak{L})$ associated to rotations of order 3.

\noi As $PSL(2,\mathfrak{L})$ is a subgroup of $PSL(2,\mathfrak{H})$ and moreover all the generators of $PSL(2,\mathfrak{L})$ are generators of $PSL(2,\mathfrak{H})$ we have the same thirteen relations of $PSL(2,\mathfrak{L})$ that form the set $\mathfrak{R}_{\mathfrak{L}}$ in the theorem \ref{apL}. 

\noi For the last relations of $\mathfrak{R}_{\mathfrak{H}}$ we need to consider the action of the Hurwitz unitary group $\mathcal{U}(\mathfrak{H})$. This group divides the pyramid $\mathcal{P}$ and the cube $\mathcal{C}$ that is the base of $\mathcal{P}$ in 24 4-dimensional non-compact pyramids and 24 3-dimensional compact square pyramids in $\mathcal{C}$. The fundamental domain of $PSL(2,\mathfrak{H})$ is the union of two pyramids. We defined the fundamental domain $\mathcal{P}_\mathfrak{H}$ of $PSL(2,\mathfrak{H})$ as two 4-dimensional pyramids whose bases are two 3-dimensional square pyramids diametrically opposite in $\mathcal{C}$. 

\noi The unitary group $\mathcal{U}(\mathfrak{H})$ has order 12. Then we have twelve points in $\mathcal{P}$ as the orbit of a point in $\mathcal{P}_\mathfrak{H}$ by the action of $\mathcal{U}(\mathfrak{H})$. We have a representation of this group in the Cayley graph $G_{\mathfrak{H}}$ as a truncated tetrahedron whose edges are of two types: (1) the edges of the tetrahedron, these are associated to the involution elements in $\mathcal{U}(\mathfrak{L})\subset PSL(2,\mathfrak{L}) \subset PSL(2,\mathfrak{H})$  and (2) the edges of the triangular links of the vertices of the tetrahedron, these are associated to the elements of order 3 in $\mathcal{U}(\mathfrak{H})\subset PSL(2,\mathfrak{H})$. See figure 11. 

\noi We choose a set of seven generators of $\mathcal{U}(\mathfrak{H})$ as generators of $PSL(2,\mathfrak{H})$. First we choose $D_{\mathbf{i}},D_{\mathbf{j}},D_{\mathbf{k}} \in  
\mathcal{U}(\mathfrak{H})$. They generate $\mathcal{U}(\mathfrak{L})$ and they satisfy the relations: $D_{\mathbf{i}}=(\tau_{\mathbf{i}}  T)^3, D_{\mathbf{j}}=(\tau_{\mathbf{j}}  T)^3$ and $D_{\mathbf{k}}=(\tau_{\mathbf{k}}  T)^3$.

\noi Straight forward computations show the following seven commuting relations:
 $$[D_{\mathbf{i}}:\tau_{\mathbf{i}}],[D_{\mathbf{j}}:\tau_{\mathbf{j}}],[D_{\mathbf{k}}:\tau_{\mathbf{k}}],$$
 and the following six anti-commuting relations:
 $$D_{\mathbf{u}}\tau_{\mathbf{w}}D_{\mathbf{u}} \tau_{\mathbf{w}}=1$$
 where $\mathbf{u}$ and $\mathbf{w}$ are two different units in the set $\{\mathbf{i,j,k}\}$.

\noi Next, we choose $D_{\omega_{1}},D_{\omega_{\mathbf{i}}},D_{\omega_{\mathbf{j}}}$ and $D_{\omega_{\mathbf{k}}} \in \mathcal{U}(\mathfrak{H})$ where 
$$\omega_{\mathbf{1}}=\frac{1}{2}(1+\mathbf{i}+\mathbf{j}+\mathbf{k}), \, \, \omega_{\mathbf{i}}=\frac{1}{2}(1+\mathbf{i}-\mathbf{j}-\mathbf{k}), \, \, \omega_{\mathbf{j}}=\frac{1}{2}(1-\mathbf{i}+\mathbf{j}-\mathbf{k}), \, \, \omega_{\mathbf{k}}=\frac{1}{2}(1-\mathbf{i}-\mathbf{j}+\mathbf{k}).$$ 

\noi The quaternion $\omega_{1}$ is the vertex of $\mathcal{C}$ which has all of its coefficients positive. There are three square faces of $\mathcal{C}$ which have $\omega_{1}$ as a vertex. The three quaternions $\omega_{\mathbf{i}}, \omega_{\mathbf{j}},\omega_{\mathbf{k}}$ are the three opposite vertices of $\omega_{1}$ in these three square faces of $\mathcal{C}$. Moreover, the four quaternions $\omega_{1}, \omega_{\mathbf{i}}, \omega_{\mathbf{j}},\omega_{\mathbf{k}}$ are the vertices of a regular tetrahedron in $\Hy$. See figure 4 or 11.

\noi We have the following new relations: 

$$ (D_{\omega_1})^{3}=(D_{\omega_{\mathbf{i}}})^{3}=(D_{\omega_{\mathbf{j}}})^{3}=(D_{\omega_{\mathbf{k}}})^{3}=1.$$ 

\noi We have the four commuting relations:
$$
[D_{\omega_1}:T],[D_{\omega_{\mathbf{i}}}:T],[D_{\omega_{\mathbf{j}}}:T],[D_{\omega_{\mathbf{k}}}:T].$$

\noi Finally we have the last four relations, one from each of the four hexagonal sides of the truncated tetrahedron:

$$  D_{\omega_1} D_{\mathbf{i}} D_{\omega_{\mathbf{i}}} D_{\mathbf{j}} D_{\omega_{\mathbf{k}}} D_{\mathbf{k}} = D_{\mathbf{k}} D_{\omega_{\mathbf{k}}} D_{\mathbf{i}} D_{\omega_{\mathbf{j}}}^{-1} D_{\mathbf{j}} D_{\omega_1}=1
$$
$$
  D_{\mathbf{i}} D_{\omega_{\mathbf{k}}} D_{\mathbf{j}} D_{\omega_{\mathbf{i}}} D_{\mathbf{k}} D_{\omega_{\mathbf{j}}}= D_{\mathbf{j}} D_{\omega_{\mathbf{j}}}^{-1} D_{\mathbf{k}} D_{\omega_{\mathbf{i}}} D_{\mathbf{i}} D_{\omega_1}=1.
$$

 \noi In this way we obtain the desired abstract presentation of $PSL(2,\mathfrak{H})$.
\end{proof} 

\noi The presentations given above are not very efficient and it is possible to reduce the number of generators and relations. However with less generators the relations become more complicated. 

\noi For instance, for the Hurwitz group we can obtain a simpler presentation with three generators by considering the isometries of $\Hy$ corresponding to the three matrices: $T=\left(\begin{array}{cc} 0 & 1 \\ 1 & 0\end{array}\right), \tau_{\mathbf{i}}=\left(\begin{array}{cc} 1 & \mathbf{i} \\ 0 & 1 \end{array}\right)$ and $D_{\omega_{\mathbf{1}}}=\left(\begin{array}{cc} \omega_{\mathbf{1}} & 0 \\ 0 & \omega_{\mathbf{1}} \end{array}\right)$.
\noi With these generators, using only $D_{\ii}:=(\tau_{\ii}T)^3$ and $D_{\omega_{1}}$, we have the presentation of the unitary Hurwitz group: 
$\mathcal{U}(\mathfrak{H})=\left< D_{\mathbf{i}}, D_{\omega_{1}}\,  | \, (D_{\mathbf{i}})^2= ( D_{\omega_{1}})^3= ( D_{\mathbf{i}} D_{\omega_{1}})^3\right>.$ This is the usual presentation of the tetrahedral group. Then we can write the other elements in $\mathcal{U}(\mathfrak{H})$ as combinations of $D_{\mathbf{i}}$ and $D_{\omega_{1}}$:
 $$
D_{\mathbf{k}}=D_{\omega_{1}}^2 D_{\mathbf{i}} D_{\omega_{1}},\, \,
D_{\mathbf{j}}=D_{\mathbf{i}}D_{\omega_{1}}^2 D_{\mathbf{i}} D_{\omega_{1}}, \,\, \quad \quad
D_{\omega_{\ii}}= D_{\mathbf{k}} D_{\omega_{1}},\, \,
D_{\omega_{\jj}}= D_{\mathbf{i}} D_{\omega_{1}}, \, \,
D_{\omega_{\kk}}= D_{\mathbf{j}} D_{\omega_{1}}.
 $$
 \noi The translations $\tau_{\jj}$ and $\tau_{\kk}$ are obtained using conjugation by $D_{\omega_{1}}$:
 $$
\tau_{\jj}= D_{\omega_{1}}\tau_{\ii}D_{\omega_{1}}^2,\,\,
\tau_{\kk}= D_{\omega_{1}}^2\tau_{\ii}D_{\omega_{1}}.
 $$

\section{The Lipschitz and Hurwitz quaternionic modular orbifolds.} 
 
\noi In this section we study the geometry of the quaternionic modular orbifolds associated to the action of the quaternionic modular groups $PSL(2,\mathfrak{L})$ and $PSL(2,\mathfrak{H})$. We describe the ends, the underlying spaces and the singular loci of the quaternionic modular orbifolds. Moreover we
give the local models of these singularities and the local isotropy groups. Finally we compute their orbifold Euler characteristic.

\begin{definition} Let $\mathcal{O}^4_{\mathfrak{L}}:=\Hy/PSL(2,\mathfrak{L})$ and $\mathcal{O}^4_{\mathfrak{H}}:=\Hy/PSL(2,\mathfrak{H})$ be the \emph{Lipschitz quaternionic modular orbifold} and the \emph{Hurwitz quaternionic modular orbifold}, respectively. 
\end{definition}

\noi These quaternionic modular orbifolds
are hyperbolic non-compact real 4-dimensional orbifolds of finite hyperbolic volume. Both have only one end and their singular loci has one connected component that accumulates to the cusp at infinity.  

\noi Moreover, these orbifolds are diffeomorphic to the quotient spaces of their fundamental domains $\mathcal{P}_{\mathfrak{L}}$ and $\mathcal{P}_{\mathfrak{H}}$ by the action of the modular groups $PSL(2,\mathfrak{L})$ and $PSL(2,\mathfrak{H})$ on their boundaries $\partial \mathcal{P}_{\mathfrak{L}}$ and $\partial \mathcal{P}_{\mathfrak{L}}$, respectively. Then they have the same volume as $\mathcal{P}_{\mathfrak{L}}$ and $\mathcal{P}_{\mathfrak{H}}$, respectively. By proposition 7.3 these volumes are $\pi^{2}/72$ and $\pi^{2}/216$, respectively. 

\noi Each of the quaternionic modular orbifolds $\mathcal{O}^4_{\mathfrak{L}}$ and $\mathcal{O}^4_{\mathfrak{H}}$ has only one end because $ \mathcal{P}_{\mathfrak{L}}$ and $ \mathcal{P}_{\mathfrak{H}}$ have each one end. We study the structure of the ends and we start by describing the sections of their ends and the thin and thick regions in the sense of Margulis thin-thick decomposition, see \cite{Rat} pages 654--665).

\subsection{The sections of the ends and the thin regions of $\mathcal{O}^4_{\mathfrak{L}}$ and $\mathcal{O}^4_{\mathfrak{H}}$.}

\noi For $r>1$ we denote by $\mathcal{E}^3_r$ the \emph{horosphere centered at the point at infinity} in $\Hy$ which consists of the set of points in $\Hy$ which have real part equal to $r$. Then $\mathcal{E}^3_r$ with the induced metric of $\Hy$ is isometric to the Euclidean 3-space. The affine modular groups $\mathcal{A}(\mathfrak{L})$ and $\mathcal{A}(\mathfrak{H})$ are the maximal subgroups of $PSL(2,\mathfrak{L})$ and $PSL(2,\mathfrak{H})$ that leave invariant each horospheres  $\mathcal{E}^3_r$ for any 
$r>0$. Moreover $\mathcal{A}(\mathfrak{L})$ and $\mathcal{A}(\mathfrak{H})$ are isomorphic to discrete subgroups of Euclidean orientation-preserving isometries of $\mathcal{E}^3_r$. 
 
\noi The fact that $PSL(2,\mathfrak{L})$ is a subgroup of index 3 of 
$PSL(2,\mathfrak{H})$ implies the existence of an epimorphism $\pi: \mathcal{A}(\mathfrak{H}) \to\mathcal{A}(\mathfrak{L})$ with kernel
$\Z/3\Z$. 
 
 \noi We call $\mathcal{E}_{r,\mathfrak{L}}^3$ and $\mathcal{E}_{r,\mathfrak{H}}^3$ the intersections of the fundamental domains of the quaternionic modular groups  $PSL(2,\mathfrak{L})$ and $PSL(2,\mathfrak{H})$ with $\mathcal{E}^3_r$, respectively. Then $\mathcal{E}_{r,\mathfrak{L}}^3$ and $\mathcal{E}_{r,\mathfrak{H}}^3$ are hyperbolic subsets of $\Hy$ with finite volume which are isometric to Euclidean 3-dimensional polyhedra. In the Lipschitz case it consists of a pair of cubes which are symmetric with respect to the point $r$, where $r\in\mathbb{R} 
\cap \partial{\mathcal {E}^3_{r,\mathcal{L}}}$. In the Hurwitz case it consists of a pair of square pyramids in $\mathcal{E}_{r,\mathfrak{H}}^3$ symmetric with respect to the point $r\in \R \subset \mathbb{H}$.  The orthogonal projections  into the ideal boundary of $\Hy$ of $\mathcal{E}_{r,\mathfrak{L}}^3$ and $\mathcal{E}_{r,\mathfrak{H}}^3$ are the same as the orthogonal projections of $\mathcal{P}_{\mathfrak{L}}$ and $\mathcal{P}_{\mathfrak{H}}$. There is covering map $\pi_ \mathcal{E}: \mathcal{E}_{r,\mathfrak{L}}^3 \to \mathcal{E}_{r,\mathfrak{H}}^3$ which is three to one.

\noi  Let $\mathcal{S}_{r,\mathfrak{L}}^{3}:=\mathcal{E}_r^3 / \mathcal{A}(\mathfrak{L})$ and $\mathcal{S}_{r,\mathfrak{H}}^{3}:=\mathcal{E}_r^3 / \mathcal{A}(\mathfrak{H})$. These are Euclidean 3-dimensional orbifolds of finite hyperbolic volume. A pair of fundamental domains for the actions of the corresponding affine groups on $\mathcal{E}_r^3$ are the polyhedra $\mathcal{E}_{r,\mathfrak{L}}^3$ and $\mathcal{E}_{r,\mathfrak{H}}^3$, respectively.

\noi The actions of the restrictions of  $PSL(2,\mathfrak{L})$ and $PSL(2,\mathfrak{H})$ on the boundaries $\partial \mathcal{E}_{r,\mathfrak{L}}^3$ and $\partial \mathcal{E}_{r,\mathfrak{H}}^3$, respectively give side-pairings of $\mathcal{E}_{r,\mathfrak{L}}^3$ and $\mathcal{E}_{r,\mathfrak{H}}^3$. The quotients of the side-pairing in $\mathcal{E}_{r,\mathfrak{L}}^3$ and $\mathcal{E}_{r,\mathfrak{H}}^3$ are  diffeomorphic to $\mathcal{S}_{r,\mathfrak{L}}^{3}$ and $\mathcal{S}_{r,\mathfrak{H}}^{3}$, respectively.

\noi There is an orbifold covering map $\pi_ \mathcal{S}: \mathcal{S}_{r,\mathfrak{L}}^{3} \to \mathcal{S}_{r,\mathfrak{H}}^{3}$ which is three to one.

\noi A convenient description of these Euclidean orbifolds is as follows: let $\mathbf{T}^3=\{(z_1,z_2,z_3) \in \mathbb{C}^{3} \ : \ |z_1|=|z_2|=|z_3|=1\}=\mathbb{S}^1\times\mathbb{S}^1\times\mathbb{S}^1$ be the 3-torus with its standard flat metric. The group of orientation--preserving isometries of $\mathbf{T}^3$ generated by the transformations $ F_T, F_{\omega}, F_{\ii}, F_{\jj},F_{\kk}$ given by the formulas: $F_T(z_1,z_2,z_3)=(\overline{z_1},\overline{z_2},\overline{z_3})$, $F_{\omega}(z_1,z_2,z_3)=(z_2,z_3,z_1)$, $F_\ii(z_1,z_2,z_3)=(z_1,\overline{z_2},\overline{z_3})$, $F_{\jj}(z_1,z_2,z_3)=(\overline{z_1},z_2,\overline{z_3})$ and $F_{\kk}:=F_{\jj} F_{\ii}$, is isomorphic to the group $\hat{\mathcal{U}}(\mathfrak{H})$ generated by $T$ and $\mathcal{U}(\mathfrak{H})$ (see definition 4.6 and proposition 4.7).  

\noi The group $\hat{\mathcal{U}}(\mathfrak{H})$ has as subgroups $\hat{\mathcal{U}}(\mathfrak{L})$, $\mathcal{U}(\mathfrak{H})$ and $\mathcal{U}(\mathfrak{L})$. These subgroups are generated by the sets of transformations $\{ F_T, F_{\ii}, F_{\jj},F_{\kk}\},$ $\{ F_{\omega}, F_{\ii}, F_{\jj},F_{\kk}\},$ $\{ F_{\omega}, F_{\ii}, F_{\jj},F_{\kk}\}$ and $\{ F_{\ii}, F_{\jj},F_{\kk}\}$, respectively. 

\noi For $r>0$, $\mathbf{T}^3\times \{r\}/\langle F_{\ii}, F_{\jj},F_{\kk}\rangle$ is homeomorphic to the Euclidean 3-orbifold $\mathcal{S}^3_{r,\mathfrak{L}}$. As a topological space it is homeomorphic to the 3-sphere $\mathbb{S}^3$. On the other hand $\mathbf{T}^3\times \{0\}/\langle F_T,F_{\ii}, F_{\jj},F_{\kk}\rangle$ is homeomorphic to the closed 3-ball $\mathbf{B}^3$.

\noi Let $[(z_1,z_2,z_3)]$ denote
the equivalence class of orbits under the transformations $F_T, F_{\omega}, F_{\ii}, F_{\jj},F_{\kk}$.

\noi There exists a strong deformation retract of $\mathcal{O}^4_{\mathfrak{L}}$ and $\mathcal{O}^4_{\mathfrak{H}}$ to the Euclidean 3-orbifolds $\mathcal{S}^3_{2,\mathfrak{L}}$ and $\mathcal{S}^3_{2,\mathfrak{H}}$, repectively. In fact, as a topological
space $\mathcal{O}^4_{\mathfrak{L}}:=\mathbf{H}^{1}_{\mathbb{H}}/ PSL(2,\mathfrak{L})$ is homeomorphic to $\mathbf{T}^3\times[0,\infty)/\sim$, where  $\mathbf{T}^3=\{(z_1,z_2,z_3) \in \mathbb{C}^{3} \ : \
|z_1|=|z_2|=|z_3|=1\}$ and $\sim$ is the equivalence relation given by the orbits of the action of some groups of diffeomorphisms of  $\mathbf{T}^3$ generated by the set and subsets of elements $ F_T, F_{\ii}, F_{\jj},F_{\kk}$. Moreover $\mathbf{T}^3\times \{r\}/\Gamma$ is homeomorphic to $\mathbb{S}^3$ for $r>0$ and $\mathbf{T}^3\times \{0\}/\Gamma$ is homeomorphic to $\mathbf{B}^3$.

\noi The underlying spaces of the 3-dimensional Euclidean orbifolds $\mathcal{S}_{r,\mathfrak{L}}^{3}$ and $\mathcal{S}_{r,\mathfrak{H}}^{3}$ are homeomorphic to the 3-sphere $\mathbb{S}^3$ because they are obtained by pasting two 3-dimensional balls along their boundaries which are 2-dimensional spheres. 

\noi The singular loci  of the 3-dimensional Euclidean orbifolds $\mathcal{S}_{r,\mathfrak{L}}^{3}$ and $\mathcal{S}_{r,\mathfrak{H}}^{3}$ are the 1-skeletons of their fundamental domains divided by the actions of the corresponding groups. Thus, their singular loci  are the two graphs which are the 1-skeleton of a cube and the graph in the figure 12, respectively. All edges of the singular locus of $\mathcal{S}_{r,\mathfrak{L}}^{3}$ are labeled by 2. The labels of the edges of  the singular locus of $\mathcal{S}_{r,\mathfrak{H}}^{3}$ are showed in the figure 12.

\subsubsection{The singular locus of the Lipschitz cusp section.}

\noi All the isotropy groups of the vertices in the fundamental domain of $\mathcal{S}_{r,\mathfrak{L}}^{3}$ are isomorphic to $\mathcal{U}(\mathfrak{L})$. All the isotropy groups of points in the edges of the fundamental domain of $\mathcal{S}_{r,\mathfrak{L}}^{3}$ are isomorphic to $\Z/2\Z$.  The isotropy groups of the 6 open 2-dimensional faces and two open 3-dimensional faces are trivial. The orbifold  $\mathcal{S}_{r,\mathfrak{L}}^{3}$ has 8 vertices, 12 edges, 6 square faces and two cubic 3-dimensional faces.

\noi The orbifold Euler characteristic\footnote{For the definition of orbifold Euler characteristic we refer to
\cite{Hirh} and the appendix at the end of this paper. }
 of $\mathcal{S}_{r,\mathfrak{L}}^{3}$ is $8(\frac{1}{4})-12(\frac{1}{2})+6-2=0.$
 
\subsubsection{The singular locus of the Hurwitz cusp section.}

\begin{figure}
\centering
\includegraphics[scale=0.33]{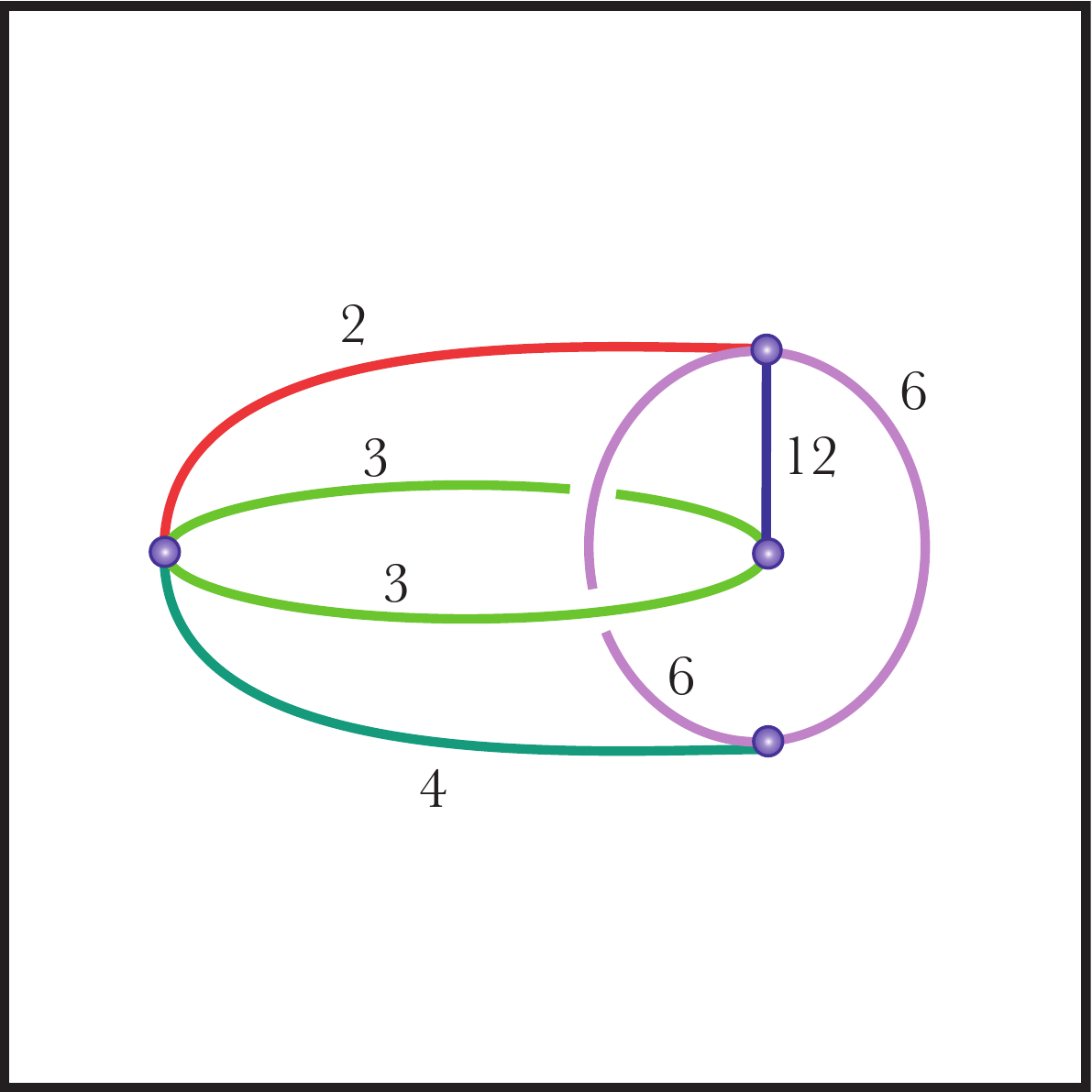}
\begin{center}
{{\bf Figure 12.} The singular locus of the Hurwitz cusp section $\mathcal{S}_{r,\mathfrak{H}}^{3}$.}
\end{center}
\end{figure}

\noi For  $\mathcal{S}_{r,\mathfrak{H}}^{3}$ there are vertices in the fundamental domain $\mathcal{E}_{r,\mathfrak{H}}^{3}$ of $\mathcal{S}_{r,\mathfrak{H}}^{3}$ with different isotropy groups : $\mathcal{U}(\mathfrak{L})$ of order four and $\mathcal{U}(\mathfrak{H})$ of order 12. Also the edges have three types of isotropy groups: the trivial group, the group $\Z/2\Z$ and $\Z/3\Z$. The center $r$ and the vertices $r+\frac{\ii+\jj+\kk}{2}$ and $r+\frac{-\ii-\jj-\kk}{2}$ of the cubes $\mathcal{E}_{r,\mathfrak{H}}^{3}$ have isotropy groups isomorphic to $\mathcal{U}(\mathfrak{H})$. The points $r+\frac{\kk}{2}$, $r-\frac{\kk}{2}$ and $r+\frac{\ii+\kk}{2}$, $r+\frac{-\ii-\kk}{2}$ and $r+\frac{\jj+\kk}{2}$ and $r+\frac{-\jj-\kk}{2}$ have isotropy groups isomorphic to $\mathcal{U}(\mathfrak{L})$. See the figure 12.  

\noi The points in the edges of $\mathcal{E}_{r,\mathfrak{H}}^{3}$ with have $r$ as a vertex and the points $r+\frac{\ii+\jj+\kk}{2}$ and $r+\frac{-\ii-\jj-\kk}{2}$ as second vertex have isotropy groups isomorphic to $\Z/3\Z$. The points in the edges of $\mathcal{E}_{r,\mathfrak{H}}^{3}$ with have $r$ as a vertex and the points $r+\frac{\jj-\kk}{2}$, $r+\frac{\ii-\kk}{2}$, $r+\frac{-\jj+\kk}{2}$, $r+\frac{-\ii+\kk}{2}$  as second vertex have trivial isotropy groups. All the points in the other edges of $\mathcal{E}_{r,\mathfrak{H}}^{3}$ have isotropy groups isomorphic to $\Z/2\Z$. The isotropy groups of the 5 open 2-dimensional faces and of the two open 3-dimensional faces are trivial. The orbifold  $\mathcal{S}_{r,\mathfrak{L}}^{3}$ has 4 vertices, 7 edges, 4 triangular faces and one square face and two 3-dimensional faces.

\noi The orbifold Euler characteristic of $\mathcal{S}_{r,\mathfrak{H}}^{3}$ is $2(\frac{1}{12}) +2(\frac{1}{4})-1-2(\frac{1}{3})-4(\frac{1}{2})+5-2=0$. 

\begin{remark} As it is expected, the orbifold Euler characteristics of both  $\mathcal{S}_{r,\mathfrak{L}}^{3}$
and  $\mathcal{S}_{r,\mathfrak{H}}^{3}$ vanishes since both orbifolds are compact and Euclidean.
\end{remark}
\subsubsection{The structure of the ends.}

\noi The family of Euclidean orbifolds $\mathcal{S}_{r,\mathfrak{H}}^{3}$ consists of orbifolds which are homothetic for all $r>1$. The Euclidean volume $V_{\mathfrak{H}}(r)$ decreases exponentially to 0 as $r\to\infty$.  The same is true for the family $\mathcal{S}_{r,\mathfrak{L}}^{3}$, $r>1$ and the corresponding volume $V_{\mathfrak{L}}(r)$, since
$V_{\mathfrak{L}}(r)=3V_{\mathfrak{H}}(r)$.

\noi The thin parts are open cylinders on the sections $\mathcal{S}_{r,\mathfrak{L}}^{3}$ and $\mathcal{S}_{r,\mathfrak{H}}^{3}$, respectively. More precisely, $\mathcal{S}_{r,\mathfrak{L}}^{3}$ and $\mathcal{S}_{r,\mathfrak{H}}^{3}$  separate $\mathcal{O}_{\mathfrak{L}}^{4}$ and $\mathcal{O}_{\mathfrak{H}}^{4}$, respectively, into two connected components with boundaries $\mathcal{S}_{r,\mathfrak{L}}^{3}$ and $\mathcal{S}_{r,\mathfrak{H}}^{3}$, respectively. One of the components is compact and the other is non-compact but with finite hyperbolic volume. Using Margulis notation the compact part is the \emph{thick} region and the non-compact part is the \emph{thin} region of the corresponding orbifolds. The thin regions are the non-compact orbifolds diffeomorphic to half-open cylinders $\mathcal{Z}_{r,\mathfrak{L}}^{4}:=\mathcal{S}_{r,\mathfrak{L}}^{3} \times [0,1)$ and $\mathcal{Z}_{r,\mathfrak{H}}^{4}:=\mathcal{S}_{r,\mathfrak{H}}^{3} \times [0,1)$, respectively. There is an orbifold cover $\pi_ \mathcal{Z}: \mathcal{Z}_{r,\mathfrak{L}}^{4} \to \mathcal{Z}_{r,\mathfrak{H}}^{4}$ which is three to one.

\subsection{The thick regions and underlying spaces of $\mathcal{O}^4_{\mathfrak{L}}$ and $\mathcal{O}^4_{\mathfrak{H}}$.}

\noi Each of the underlying spaces of the 4-dimensional orbifolds $\mathcal{O}^4_{\mathfrak{L}}$ and $\mathcal{O}^4_{\mathfrak{H}}$ has only one end. The sections of the ends are 3-dimensional Euclidean orbifolds $\mathcal{S}_{r,\mathfrak{L}}^{3}$ and $\mathcal{S}_{r,\mathfrak{H}}^{3}$ which each of their underlying spaces is homeomorphic to the 3-sphere $\mathbb{S}^3$. Then the thin regions of $\mathcal{O}^4_{\mathfrak{L}}$  and $\mathcal{O}^4_{\mathfrak{H}}$ are homeomorphic to the 4-ball  $\mathbf{D}^4$ minus one point  for example in its center. Moreover each of the thick regions of $\mathcal{O}^4_{\mathfrak{L}}$  and $\mathcal{O}^4_{\mathfrak{H}}$ is homeomorphic to the 4-ball  $\mathbf{D}^4$. 

\noi Then each of the underlying spaces of $\mathcal{O}^4_{\mathfrak{L}}$  and $\mathcal{O}^4_{\mathfrak{H}}$ is homeomorphic to the 4-sphere $\mathbb{S}^4$ minus one point thus each of the underlying spaces is homeomorphic to $\mathbb{R}^4$.

\subsection{The singular locus of $\mathcal{O}^4_{\mathfrak{L}}$ and $\mathcal{O}^4_{\mathfrak{H}}$.}

\noi  The 3-dimensional faces of $\mathcal{P}_{\mathfrak{L}}$ and $\mathcal{P}_{\mathfrak{H}}$  are identified in pairs by the action of the generators of the Lipschitz and Hurwitz modular groups, respectively. We denote by $\Sigma_{\mathfrak{L}}$ and $\Sigma_{\mathfrak{H}}$ the singular loci of $\mathcal{O}^{4}_{\mathfrak{L}}$ and $\mathcal{O}^{4}_{\mathfrak{H}}$, respectively.  They are the 2-dimensional skeletons of their fundamental domains $\mathcal{P}_{\mathfrak{L}}$ and $\mathcal{P}_{\mathfrak{H}}$. Then each singular locus is non-compact with one connected component. 

\noi The Lipschitz singular locus $\Sigma^2_{\mathfrak{L}}$ is the union of a 2-dimensional cube $\mathcal{C}_{\Sigma}$ which is obtained by identifying the boundaries of the cubes $\mathcal{C}_1$ and $\mathcal{C}_2$ in $\mathcal{C}\subset \Pi$  by the action of the group $\hat{\mathcal{U}}(\mathfrak{L})$ and the non-compact cone over its 1-skeleton of the 2-dimensional sides of $\mathcal{P}_{\mathfrak{L}}$ which are asymptotic to the point at infinity. 

\noi The Hurwitz singular locus $\Sigma^2_{\mathfrak{H}}$ is  the union of the 2-dimensional pyramid  $\mathcal{P}_{\Sigma}$ which is obtained by identifying  the boundaries of the union of the two pyramids $\mathcal{P}_1$ and $\mathcal{P}_2$ in $\mathcal{C}_1, \mathcal{C}_2 \subset \mathcal{C}\subset \Pi$ by the action of the group $\hat{\mathcal{U}}(\mathfrak{H})$ and the non-compact cone over its 1-skeleton of the 2-dimensional sides of $\mathcal{P}_{\mathfrak{H}}$ that are asymptotic to the point at infinity.

\subsection{Local models of the modular orbifolds singularities.}\label{local}

\noi In this section we study the local models of the isolated singularities of $\mathcal{O}^{4}_{\mathfrak{L}}$ and $\mathcal{O}^{4}_{\mathfrak{H}}$.

\noi The local models of the isolated singularities of $\mathcal{O}^{4}_{\mathfrak{L}}$ and $\mathcal{O}^{4}_{\mathfrak{H}}$ are obtained as quotients of a hyperbolic 4-ball  $\mathbf{B}^4$ by the
 action of a discrete subgroup of hyperbolic isometries which fix its center and its boundary which is a 3-sphere. 
\noi Let $\mathbb F_0\subset\mathbb F_{\mathbf B}$ be the subgroup of
 hyperbolic isometries of the hyperbolic ball which fix the
 center of $\mathbf B$. Then $\mathbb F_0$ is the group of orientation-preserving isometries of the 3-sphere $\mathbb{S}^3$. The group $\mathbb F_0$ is isomorphic to SO(4).

\noi Let $\mathbf B^4\subset {\mathbb H}$ denote, as before, the disk model of $\mathbf{H}^{1}_{\mathbb{H}}$.

\noi The ideal boundary of $\mathbf B^4$ is the unitary 3-sphere $\mathbb{S}^3 = \left \{ \mathbf{q}= \alpha +\beta \jj \in \Hy: \ \ \alpha,\beta\in\mathbf{C}, |\alpha|^2 + |\beta|^2 = 1\right \}$.

\noi Let $\mathrm{SU}(2) = \left \{ \begin{pmatrix} \alpha&-\overline{\beta}\\ \beta & \overline{\alpha} \end{pmatrix}: \ \ \alpha,\beta\in\mathbf{C}, |\alpha|^2 + |\beta|^2 = 1\right \}$ be  the \emph{unitary special group}. The 3-sphere is a Lie group isomorphic to SU(2). An element $\mathbf{q}= \alpha +\beta \jj \in \mathbb{S}^3$ corresponds to the element $\begin{pmatrix} \alpha&-\overline{\beta}\\ \beta & \overline{\alpha} \end{pmatrix}$ in SU(2).

\noi We define $$f_{(u,v)}:\mathbf{B}^4 \to \mathbf{B}^4;$$
 $$f_{(u,v)}(\mathbf{q}) \mapsto u\mathbf{q}v$$
 
\noi where $|u|=|v|=1$. We observe that $f_{(u,v)} \in $SO(4) and $(u,v)\in \mathbb{S}^{3}\times \mathbb{S}^{3}=\mathrm{SU}(2) \times \mathrm{SU}(2)$. We introduce $$\phi:\mathrm{SU}(2)\times \mathrm{SU}(2) \to \mathrm{SO}(4);$$
 $$(u,v) \mapsto f_{(u,v)}.$$
 
\noi Since SU(2) $\times$ SU(2) is simply connected but the fundamental group of SO(4) is $\mathbb{Z}/2\mathbb{Z}$ then the kernel of $\phi$ is the group with two elements consisting of $(1,1)$ and $(-1,-1)$.\\

 \noi There is an
 orthogonal action of $\mathbb{S}^3 \times \mathbb{S}^3$ on
 $\mathbb{S}^3$ given by $\mathbf{q} \mapsto q^{-1}_1\mathbf{q}q_2$,
 for a fixed pair $(q_1, q_2)\in \mathbb{S}^3 \times
 \mathbb{S}^3$. This defines a homomorphism of Lie groups
 $\mathbb{S}^3 \times \mathbb{S}^3 \to$ SO(4) with kernel
 $\mathbb{Z}/2\Z$ generated by $(-1,-1)$. Then SO(4) is isomorphic
 to the central product $\mathbb{S}^3 \times _{\mathbb{Z}/2\Z}
 \mathbb{S}^3$.  
  The finite subgroups of $SO(4)$ are, up to
 conjugation, exactly the finite subgroups of the central products of
 two binary polyhedral groups $G_1$ and $G_2$
$$G_1 \times _{\mathbb{Z}/2\Z} G_2 \subset \mathbb{S}^3 \times _{\mathbb{Z}/2\Z} \mathbb{S}^3.$$   

\noi The finite subgroups of $SU(2)$ have been classified by Felix
Klein in \cite{K} and they are the cyclic groups of order $n$ ($n>1$)),
the binary dihedral groups $ \langle 2,2,n\rangle$ of order 4$n$, the binary tetrahedral group $ \langle 2,3,3\rangle$ of order 24, the binary
octahedral group $ \langle 2,3,4\rangle$ of order 48 and the binary icosahedral group $ \langle 2,3,5\rangle$ of order 120. These are the
\emph{binary polyhedral groups}.  Let $\Gamma$ be a finite subgroup of
$\mathbb F_0$. Let $r>0$ and ${\mathbf B}^4_r$ be the hyperbolic ball
centered at the origin of radius $r$. The ball ${\mathbf B}^4_r$ is
invariant under the action of $\Gamma$. Let ${\mathcal
  O}^4(\Gamma,r)={\mathbf B^4}_r{\slash\Gamma}$. For every $r>0$ the orbifold ${\mathcal
  O}^4(\Gamma,r)={\mathbf B^4}_r{\slash\Gamma}$ is equivalent, up to rescaling the orbifold metric, to a fixed ${\mathbf B^4}_{\epsilon}{\slash\Gamma}$ for $\epsilon$ sufficiently small. Let ${\mathcal
  O}^4(\Gamma)={\mathbf B^4}_{\epsilon}{\slash\Gamma}$.
  
\begin{definition} Let $p$ and $q$ be two integers. Let  $\Gamma(p,q)\subset{SU(2)}$ be the abelian subgroup generated
by the map $T_{p,q}(\mathbf{z}_1,\mathbf{z}_2)=(e^{2\pi{i}/p}\mathbf{z}_1,
e^{2\pi{i}/q}\mathbf{z}_2)$. The group $\Gamma(p,q)$ is isomorphic to $\Z/p\Z\oplus\Z/q\Z$. Let us denote by ${\mathcal O}(p,q)$ the
orbifold ${\mathcal O}^4(\Gamma(p,q),\epsilon)={\mathbf
  B^4}_{\epsilon}{\slash\Gamma(p,q)}$. If
$\Gamma(G_1,G_2)\subset{SO(4)}$ is a finite subgroup isomorphic to
$G_1 \times _{\mathbb{Z}/2\Z} G_2$, where $G_1$ and $G_2$ are the binary polyhedral groups then we denote by $\mathcal {O}(
G_1,G_2)$ the orbifold $\mathbf{B}^4_{\epsilon}/\Gamma(G_1,G_2)$. If $G_k=\Z/p\Z$, where $k=1,2$, then we write $p$ in the place of $G_k$ in the notation 	$\mathcal {O}(G_1,G_2)$. If  $G_k=\Z/p\Z\oplus\Z/q\Z$, where $k=1,2$, then we write $(p,q)$ in the place of $G_k$ in the notation 	$\mathcal {O}(G_1,G_2)$.
\end{definition}

\subsection{Singular locus of the Lipschitz and Hurwitz modular orbifolds.} 
\noi We give the local groups and the local models of the singular points in the locus of the Lipschitz and Hurwitz modular orbifolds $\mathcal{O}^{4}_{\mathfrak{L}}$ and $\mathcal{O}^{4}_{\mathfrak{H}}$, respectively. We also give a presentation, a fundamental domain, a Cayley graph and a thorough study of its spherical 3-orbifold link for each group associated to singular points in the singular loci of $\mathcal{O}^{4}_{\mathfrak{L}}$ and $\mathcal{O}^{4}_{\mathfrak{H}}$, respectively.  

\noi We can describe orbifold stratifications of the set of singular points of $\mathcal{O}^{4}_{\mathfrak{L}}$ and $\mathcal{O}^{4}_{\mathfrak{H}}$ according to their isotropy groups. The lists of types of singular points and their corresponding isotropy groups can be divided into two types of strata: the compact and the non compact. Also in these two groups we can divide the strata according to the dimension of the corresponding stratum in the stratification. 
\noi We give a list of points in 11 strata in the Lipschitz singular locus $\Sigma^2_{\mathfrak{L}}$ of $\mathcal{O}^{4}_{\mathfrak{L}}$ and 15 strata in the Hurwitz singular locus $\Sigma^2_{\mathfrak{H}}$ of $\mathcal{O}^{4}_{\mathfrak{H}}$ which have isomorphic isotropy groups and denote these groups by $\Gamma_{k}$ and $\Lambda_{m}$, where $k=1,...,11;m=1,...,15$.    

\noi The isotropy group of a point in a non-compact stratum in the singular loci $\Sigma^{2}_{\mathfrak{L}}$ or $\Sigma^{2}_{\mathfrak{H}}$ is the isotropy group of the action of $\mathcal{A}(\mathfrak{L})$ and $\mathcal{A}(\mathfrak{H})$, respectively.
\noi The singular points of the orbifold $\mathcal{O}^{4}_{\mathfrak{L}}$ which are in compact strata are in the boundary of the cubes $\mathcal{C}_1$ and $\mathcal{C}_2$. The singular points of the orbifold $\mathcal{O}^{4}_{\mathfrak{H}}$ which are in compact strata are in the boundary of the squared pyramids $\mathcal{P}_1$ and $\mathcal{P}_2$. The strata can be characterized by the dimension of the corresponding stratum in the stratification and whether the stratum contains the point 1 or not.

\noi The point $\q=\frac{1}{3}(\sqrt{3}+\ii+\jj+\kk)\in \mathcal{P}_{\mathfrak{H}}\subset \mathcal{P}_{\mathfrak{L}}$ is a regular point for the orbits of $PSL(2,\mathfrak{L})$ and $PSL(2,\mathfrak{H})$. We had considered previously $\mathcal{P}_\mathfrak{L}$ and $\mathcal{P}_\mathfrak{H}$ as unions of the non-compact cones over the cubes $\mathcal{C}_1$ and $\mathcal{C}_2$ and the pyramids $\mathcal{P}_1$ and $\mathcal{P}_2$ with apices the point at infinity, respectively. However, to obtain the isotropy groups it is better to define new fundamental regions $\tilde{\mathcal{P}}_\mathfrak{L}$ and $\tilde{\mathcal{P}}_\mathfrak{H}$ as follows:

\begin{enumerate} 
\item Let $\tilde{\mathcal{P}}_\mathfrak{L}$ be the non-compact bicone over $\mathcal{C}_1$ with apices the ideal vertices at 0 and the point at infinity.

\item Let $\tilde{\mathcal{P}}_\mathfrak{H}$ be the non-compact bicone over  $\mathcal{P}_1$ with apices the ideal vertices at 0 and the point at infinity.
\end{enumerate}

\noi These are convex bicones over the cube $\mathcal{C}_1$ and the pyramid $\mathcal{P}_1$, each with two ends and they are fundamental domains for $PSL(2,\mathfrak{L})$ and $PSL(2,\mathfrak{H})$, respectively. 

\begin{remark}
The action of the groups $PSL(2,\mathfrak{L})$ and $PSL(2,\mathfrak{H})$ on their new  fundamental regions $\tilde{\mathcal{P}}_\mathfrak{L}$ and $\tilde{\mathcal{P}}_\mathfrak{H}$ is equivalent to the action of $G(3)$ on $\mathcal{P}$ described before in section 5.1. The groups $PSL(2,\mathfrak{L})$ and $PSL(2,\mathfrak{H})$ act on $\tilde{\mathcal{P}}_\mathfrak{L}$ and $\tilde{\mathcal{P}}_\mathfrak{H}$ by rotations around the 2-faces of the cube $\mathcal{C}_1$ and the pyramid $\mathcal{P}_1$, respectively.
\end{remark}

\noi For the strata in $\Sigma^{2}_{\mathfrak{L}}$ or $\Sigma^{2}_{\mathfrak{H}}$ that contain 1 it is easy to calculate the isotropy group as a subgroup of $\hat{\mathcal{U}}(\mathfrak{L})$ or $\hat{\mathcal{U}}(\mathfrak{H})$.  For the strata $\Sigma^{2}_{\mathfrak{L}}$ or $\Sigma^{2}_{\mathfrak{H}}$ which do not contain 1 we consider the orbit of 1 by the action of points in their isotropy groups. For each point $\pp$ in one of these strata 1 is a regular point for the action of its isotropy group on a 3-sphere $\mathbb{S}_{r_{0}}^3(\pp)$, where $r_{0}$ is the distance from $\pp$ to 1. Then the orbit of 1 is in correspondence 1-1 with fundamental regions of the correspondent isotropy group acting in a 3-sphere $\mathbb{S}^3$. 

\noi The fundamental regions on $\mathbb{S}^3$ of the isotropy groups of each stratum in the singular loci $\Sigma_{\mathfrak{L}}$ and $\Sigma_{\mathfrak{H}}$ are formed by two 4-simplices. Each 4-simplex has 5 3-dimensional faces. Therefore the isotropy groups have presentations with at most 4 generators.   

\subsection{The stratification of the Lipschitz singular locus.} 

\noi We give the list of 11 strata in the Lipschitz singular locus $\Sigma^{2}_{\mathfrak{L}}$. For each stratum we enlist its isotropy group $\Gamma_k$, $k=1,...11$, determine a fundamental domain of the action of its isotropy group on $\mathbb{S}^3$, and give a geometrical description of its spherical 3-orbifold (or spherical link). 

\noi In the following list we consider the canonical projection 
$\mathfrak{p}:\mathcal{C}\to\mathcal{C}_\Sigma$. It is important to see figure 8 in each case.

\noi {\bf Non compact strata.}
\begin{enumerate}
\item[$\Gamma_1$] \textbf{Eight 1-cells.} The 1-skeleton of the non-compact cone over the cubes $\mathcal{C}_1$ and $\mathcal{C}_2 \subset \mathcal{P}_{\mathfrak{L}}$ is a set of eight open lines in $\Sigma_{\mathfrak{L}}$ which are represented in $\mathcal{P}_{\mathfrak{L}}$ as 
\begin{enumerate}
\item[a)] the half-line $\{\q\,:\, \q=r, r\in \R, r>1\}$, 
\item[b)] the three lines $\mathbf{q}=r_1 \pm \ii/2, r_1 \pm \jj/2, r_1\pm \kk/2,$ where $r_1>\frac{\sqrt{3}}{2}$, 
\item[c)] the three lines $\mathbf{q}=r_2 \pm (\ii/2+\jj/2),r_2 \pm (\ii/2+\kk/2),r_2 \pm (\jj/2+\kk/2),$ where $r_2>\frac{\sqrt{2}}{2}$, and finally 
\item[d)] the line $\mathbf{q}=r_3 \pm (\ii/2+\jj/2+\kk/2),$ where $r_3>\frac{1}{2}$. 
\end{enumerate}
\noi These eight half-lines orthogonally project, under the natural projection $\mathcal{P}_\mathfrak{L}\to \mathcal{C}$ by geodesics asymptotic to the point at infinity, to the barycenter of the cube $\mathcal{C}\subset \mathcal{P}_{\mathfrak{L}}$, the barycenters of the square faces of $\mathcal{C}$, the half of its edges, and two of its vertices, respectively. These 8 open half-lines in $\mathcal{P}_{\mathfrak{L}}$ project to 8 open half-lines in $\mathcal{O}^{4}_{\mathfrak{L}}$. Their local isotropy groups are isomorphic to the group of order 4 isomorphic to $\mathcal{U}(\mathfrak{L})$. We define $\Gamma_1=\mathcal{U}(\mathfrak{L})=\Z/2\Z\times\Z/2\Z$ as the local isotropy group of the quaternions in these 8 open half-lines. The local model for the singular points in this stratum is isometric to the orbifold $\mathcal {O}(2,2)$.

\item[$\Gamma_2$] \textbf{Twelve 2-cells.} The 2-skeleton of the non-compact cone over the cubes $\mathcal{C}_1$ and $\mathcal{C}_2$ is a set of twelve triangles with one vertex at the point at infinity which are the quaternions in $\mathcal{P}_{\mathfrak{L}}$ that orthogonally projects over the quaternions which are the edges of $\mathcal{C}_1$ and $\mathcal{C}_2$. Their isotropy groups are isomorphic to the cyclic group of order 2 isomorphic to $\Z/2\Z$ and we define $\Gamma_2=\Z/2\Z$. The local model for the singular points in this stratum is isometric to the orbifold $\mathcal {O}(2)$.

\end{enumerate}

\noi {\bf Compact strata.}

\begin{itemize}
\item  {\bf 0-dimensional.}

\begin{enumerate}
\item[$\Gamma_3$] \textbf{One 0-cell.} The common vertex $v_1=1$ of $\mathcal{C}_1$ and $\mathcal{C}_2$ which is the barycenter of the cube $\mathcal{C}$. Its isotropy group is the abelian group $\Gamma_3=\Z/2\Z\times \mathcal{U}(\mathfrak{L})\cong\hat{\mathcal{U}}(\mathfrak{L})\cong (\Z/2\Z)^3$ of order 8 generated by the involution $T$ and the elements in the Lipschitz unitary group. If we take a round hyperbolic ball $\mathbf{B}_r(1)$ with center at $\q=1$ and small radius $r$ we obtain that its boundary $\mathbb{S}_r^3(1)$ intersects the tessellation $\mathbf{Y}_\mathfrak{L}$ of $PSL(2,\mathfrak{L})$ in a spherical tessellation by sixteen Coxeter spherical right-angled regular tetrahedra. These tetrahedra are the faces of a 4-dimensional regular convex polytope which is known as the 16-cell, it is the dual polytope of the hypercube known as the 8-cell. These polytopes are two of the six Platonic polytopes of dimension 4. The Cayley graph of $\Gamma_3$ is the 1-skeleton of a truncated octahedron in $\mathbb{S}_r^3(1)$.  The local model for the singular points in this stratum is isometric to the orbifold $\mathcal {O}(2,(2,2))$.

\item[$\Gamma_4$] \textbf{Three 0-cells.} The vertices $v_2,v_3, v_4$ of $\mathcal{C}_{\Sigma}$ which are the images under $\mathfrak{p}$ of the 6 barycenters of the square faces of the cube $\mathcal{C}$. The isotropy group of these vertices is isomorphic to the group $\Gamma_4= (\Z/2\Z \times\Z/3\Z) \times \Z/2\Z= \Z/6\Z\times \Z/2\Z$ of order 12. It is the group generated by $\tau_{\mathbf{u}}T$ and $TD_{\mathbf{v}}$ where $\mathbf{u},{\mathbf{v}}=\pm{\mathbf{i}}, \pm{\mathbf{j}}, \pm{\mathbf{k}}$ and ${\mathbf{u}}\neq{\mathbf{v}}$. The group $\Gamma_4$ leaves invariant two orthogonal hyperbolic planes meeting at the barycenter of the square face of the cube; one is the plane containing the face and $\Gamma_{4}$ acts on it as the group  of order 4 isomorphic to $\Z/2\Z \times\Z/2\Z$ (the group generated by  $(\tau_{\mathbf{u}}T)^3$ and $TD_{\mathbf{v}}$ restricted to this plane), and the other is its orthogonal complement, and $\Gamma_{4}$ acts on it as a rotation of order 3 (the group generated by  $(\tau_{\mathbf{u}}T)^2$ restricted to this plane). The 12 middle points of the edges are identified in groups of four points by translations with three singular vertices of $\mathcal{C}_{\Sigma}\subset \mathcal{O}^{4}_{\mathfrak{L}}$.  The local model for the singular points in this stratum is isometric to the orbifold $\mathcal {O}(2,6)$.

\item[$\Gamma_5$] \textbf{Three 0-cells.} The three vertices $v_5,v_6, v_7$ of $\mathcal{C}_{\Sigma}$ which are the images under $\mathfrak{p}$ of the 12 middle points of the edges of $\mathcal{C}$. The isotropy group of these vertices is isomorphic to the group $\Gamma_5=\Z/4\Z\times  (\Z/3\Z \times\Z/2\Z) =\Z/4\Z \times \Z/6\Z $  of order 24. It is the group generated by $TD_{\mathbf{v}}$, $\tau_{\mathbf{u}+\mathbf{v}}T$, $(\tau_{\mathbf{v}}T)^2$ and $(\tau_{\mathbf{u}}T)^2$ where $\mathbf{u},\mathbf{v}=\pm{\mathbf{i}}, \pm{\mathbf{j}}, \pm{\mathbf{k}}$, and $\mathbf{u}\neq\mathbf{v}$. The group $\Gamma_5$ leaves invariant two orthogonal hyperbolic planes meeting at the middle point of the edge of the cube; one is the hyperbolic plane whose ideal boundary is the line generated by $\mathbf{u}+\mathbf{v}$ and $\Gamma_{5}$ acts on it as  a rotation of order 4 (the group generated by  $\tau_{\mathbf{u}+\mathbf{v}}T$ restricted to this plane), and the other is its orthogonal complement, and $\Gamma_{5}$ acts as on it as a rotation of order 6 (the group generated by $D_{\mathbf{u}\mathbf{v}}T$  and $(\tau_{\mathbf{u}}T)^2$ restricted to this plane). The 12 middle points of the edges of the cube are identified in groups of four points by translations with three singular vertices of $\mathcal{C}_{\Sigma}\subset \mathcal{O}^{4}_{\mathfrak{L}}$.  The local model for the singular points in this stratum is isometric to the orbifold $\mathcal {O}(4,6)$.

\item[$\Gamma_6$] \textbf{One 0-cell.} The vertex $v_8$ of $\mathcal{C}_{\Sigma}$ which is the image under $\mathfrak{p}$ of the 8 vertices of $\mathcal{C}$. The isotropy group of this vertex is isomorphic to the group
  $\Gamma_{6}=(\Z/2\Z \times \Z/2\Z) \times  \langle 2,3,3\rangle$ of order 96 where $ \langle 2,3,3\rangle$ is the binary tetrahedral group of order 24. It is the group generated by $\tau_{\mathbf{i}+\mathbf{j}+\mathbf{k}}T$, $(\tau_{\mathbf{i}+\mathbf{j}}T)^2$, $(\tau_{\mathbf{i}+\mathbf{k}}T)^2$, $(\tau_{\mathbf{j}+\mathbf{k}}T)^2$, $(\tau_{\mathbf{u}} T)^2$ where
  $\mathbf{u}=\pm{\mathbf{i}},\pm{\mathbf{j}},\pm{\mathbf{k}}$.  The local model for the singular points in this stratum is isometric to the orbifold $\mathcal {O}((2,2),  \langle 2,3,3\rangle)$.

\end{enumerate}

\item  {\bf 1-dimensional.}
\begin{enumerate}

\item[$\Gamma_7$] \textbf{Three 1-cells.} The points of the three edges of $\mathcal{C}_{\Sigma}$ that are incident with the barycenter of $\mathcal{C}$ have isotropy group $\Gamma_{7}=\Z/2\Z \times \Z/2\Z$. If the point is contained in the edge of $\mathcal{C}_1$ which contains 1 and $\sqrt{3}+\mathbf{u}/2$, then $\Gamma_{7}$ is the group generated by
  $(\tau_{\mathbf{u}}  T)^{3}$ where $\mathbf{u}=\pm{\mathbf{i}},\pm{\mathbf{j}},\pm{\mathbf{k}}$. The group $\Gamma_{7}$ leaves invariant the hyperbolic 2-plane generated by 1 and $\mathbf{u}$ and the hyperbolic hyperplane $\Pi$. Moreover $\Gamma_{7}$ acts on it as a rotation of order 2 in $\Pi$. The
  points of these 6 edges in $\mathcal{C}_1$ and $\mathcal{C}_2$ are identified in groups of two and form 3
  singular edges in $\mathcal{O}^{4}_{\mathfrak{L}}$ incident with 1.  The local model for the singular points in this stratum is isometric to the orbifold $\mathcal {O}(2,2)$.
  
\item[$\Gamma_8$] \textbf{Three 1-cells.} The points of the six edges of $\mathcal{C}_{\Sigma}$ which have one vertex at the barycenter of the square faces of $\mathcal{C}$ and the other vertex is the middle point of an edge of the cube $\mathcal{C}$. Its
  isotropy group $\Gamma_{8}=\Z/2\Z \times \Z/3\Z$ is the group generated by
  $(\tau_{\mathbf{u}}  T)^{2}$ and $D_{\mathbf{v}}T$ where $\mathbf{u},\mathbf{v}=\pm{\mathbf{i}}\pm{\mathbf{j}}\pm{\mathbf{k}},\mathbf{u} \neq \mathbf{v}$. The group $\Gamma_{8}$ leaves invariant two orthogonal
  hyperbolic planes meeting at the vertex of the
  cube. $\Gamma_{8}$ acts on one plane as a rotation of order 6 and as a rotation of order 4 on the other plane. The
  points of the 6 edges of $\mathcal{C}$ are identified and form 3
  singular edges in $\mathcal{C}_{\Sigma} \subset\mathcal{O}^{4}_{\mathfrak{L}}$ which are incident with $1\in \mathcal{C}_{\Sigma} $.  The local model for the singular points in this stratum is isometric to the orbifold $\mathcal {O}(2,3)$.

\item[$\Gamma_9$] \textbf{Six 1-cells.} The points of the three edges of $\mathcal{C}_{\Sigma}$ which are incident with a vertex of $\mathcal{C}$. Its
  isotropy group $\Gamma_{9}=\Z/2\Z \times \Z/6\Z$ is the group generated by
  $(\tau_{\mathbf{u}}  T)^{3}$ where $\mathbf{u}=\pm{\mathbf{i}}\pm{\mathbf{j}}\pm{\mathbf{k}}$. The group $\Gamma_{9}$ leaves invariant two orthogonal
  hyperbolic planes meeting at the vertex of the
  cube. $\Gamma_{9}$ acts on one plane as a rotation of order 6 and as a rotation of order 2 in the other. The
  points of the 12 edges of $\mathcal{C}$ are identified and form six singular edges in $\mathcal{O}^{4}_{\mathfrak{L}}$.  The local model for the singular points in this stratum is isometric to the orbifold $\mathcal {O}(2,6)$.

\end{enumerate}

\item  {\bf 2-dimensional.}
\begin{enumerate}

\item[$\Gamma_{10}$] \textbf{Three 2-cells.} The points of the interior of a square face of $\mathcal{C}_{\Sigma}$ that is incident with the barycenter of $\mathcal{C}$. Its isotropy group $\Gamma_{10}=\Z/2\Z$ is the group generated
  by $ D_{\mathbf{u}}T$, where $\mathbf{u}=\ii$, $\jj$, $\kk$. The group
 $\Gamma_{10}$ acts on one plane  as  a rotation of order 2 around a hyperbolic 2-plane in $\Pi$. The points in the 6 squares
  are identified in groups of two with 3 singular 2-cells in
  $\mathcal{O}^{4}_{\mathfrak{L}}$.  The local model for the singular points in this stratum is isometric to the orbifold $\mathcal {O}(2)$.
  
\item[$\Gamma_{11}$] \textbf{Three 2-cells.} The points of the interior of a square face of $\mathcal{C}_{\Sigma}$ that is not incident with the barycenter of $\mathcal{C}$. Its isotropy group $\Gamma_{11}=\Z/3\Z$ is the group generated
  by $(\tau_{u}  T)^{2}$, where $u=\pm{\mathbf{i}}, \pm{\mathbf{j}}, \pm{\mathbf{k}}$. The group
 $\Gamma_{11}$ leaves invariant two orthogonal hyperbolic planes meeting at
  the point in the square face of the cube; one is the plane containing the face and $\Gamma_{11}$ acts on it as
  a rotation of order 3, and the other is its orthogonal complement
  and $\Gamma_{11}$ acts on it as the identity. The points in the 6 squares
  are identified in groups of two with 3 singular 2-cells in
  $\mathcal{O}^{4}_{\mathfrak{L}}$.   The local model for the singular points in this stratum is isometric to the orbifold $\mathcal {O}(3)$.
\end{enumerate}
\end{itemize}

\subsection{The stratification of the Hurwitz singular locus.} 

\noi Now we describe the orbifold stratification of the set of singular points of the Hurwitz modular orbifold $\mathcal{O}^{4}_{\mathfrak{H}}$ according to their isotropy groups. 

\noi We give the list of 15 strata in the Hurwitz singular locus $\Sigma^{2}_{\mathfrak{H}}$. For each stratum we enlist the isotropy group $\Lambda_k$, $k=1,...15$, determine a fundamental domain of its action on $\mathbb{S}^3$, and study in detail the corresponding spherical 3-orbifold (or spherical link). 

\noi In the following list we consider the canonical projection 
$\mathfrak{p}:\mathcal{P}_d\to\mathcal{P}_\Sigma$ where $d=1,2$.

\noi In the same way as in the case of the Lipschitz singular locus the list of types of singular points and their corresponding isotropy groups can be divided by the dimension and the compactness of the corresponding stratum in the stratification. The non-compact strata have the same isotropy groups that the respective Euclidean 3-orbifold which is the intersection of  $\mathcal{O}^{4}_{\mathfrak{H}}$ with a horosphere. The list is the following:

\noi {\bf Non compact strata.}
\begin{enumerate}

\item[$\Lambda_{1}, \, \Lambda_{2}$] \textbf{Five 1-cells.} The 1-skeleton of the non-compact cone over the pyramids  $\mathcal{P}_1$ and    $\mathcal{P}_2$ is a set of five open half-lines which are represented in $\mathcal{P}_{\mathfrak{H}}$ as 
\begin{enumerate}
\item[a)] the half-line $\{\q \in \Hy \,:\, \q=r, r\in \R, r>1\}$,  
\item[b)] the line $\mathbf{q}= r_1\pm \kk/2,$ where $r_1>\frac{\sqrt{3}}{2}$, 
\item[c)] the two lines $\mathbf{q}=r_2 \pm (\ii/2+\jj/2),r_2 \pm (\jj/2+\kk/2),$ where $r_2>\frac{\sqrt{2}}{2}$ and finally 
\item[d)] the line $\mathbf{q}=r_3 \pm (\ii/2+\jj/2+\kk/2),$ where $r_3>\frac{1}{2}$.
\end{enumerate} 
\noi These five half-lines orthogonally project, under the natural projection $\mathcal{P}_\mathfrak{H}\to \mathcal{C}$ by geodesics asymptotic to the point at infinity,  to the vertices of $\mathcal{P}_1$: the barycenter of the cube $\mathcal{C}$, the barycenters of any square face of $\mathcal{C}$, the half of two of its edges, and two of its vertices, respectively. These 5 open half-lines in $\mathcal{P}_{\mathfrak{H}}$ project to 4 open lines in $\mathcal{O}^{4}_{\mathfrak{H}}$. Their isotropy groups are isomorphic for a) and d) to the abelian group of order 8 isomorphic to $\mathcal{U}(\mathfrak{H})$ and for b) and c) to the dihedral group of order 4 isomorphic to $\mathcal{U}(\mathfrak{L})$. We obtain $\Lambda_1:=\mathcal{U}(\mathfrak{L})=\Z/2\Z\times\Z/2\Z$ and $\Lambda_2:=\mathcal{U}(\mathfrak{H})=\Z/2\Z\times\Z/2\Z\times\Z/2\Z$ as the isotropy groups of the quaternions in these 5 open half-lines. The local models for the singular points in these strata are isometric to the orbifolds $\mathcal {O}(2,2)$ and $\mathcal {O}(2,(2,2))$, respectively.

\item[$\Lambda_{3}, \, \Lambda_{4}$] \textbf{Eight 2-cells.}  The 2-skeleton of the non-compact cone over the pyramid $\mathcal{P}_1 \subset \mathcal{P}_{\mathfrak{H}}$ is a set of eight triangles with one vertex at infinity which are orthogonally projected over the quaternions which are the edges of $\mathcal{P}_1$ and $\mathcal{P}_2$. The isotropy groups of points in the five triangles with base the edges of the square base of $\mathcal{P}_{1}$ and the edge that joins 1 with the barycenter of a squared face of $\mathcal{C}$ are isomorphic to the cyclic group $\Z/2\Z$. For points in the diagonal edge of $\mathcal{P}_1$ which joins 1 with a vertex of $\mathcal{C}$ their isotropy groups are isomorphic to $\Z/3\Z$. For points in the two edges which joins 1 with middle points of the edges of $\mathcal{C}$ their isotropy groups are isomorphic to the trivial group, then these points are not singular. We define for these six strata $\Lambda_3=\Z/2\Z$ and $\Lambda_4=\Z/3\Z$.  The local models for the singular points in these strata are isometric to the orbifolds $\mathcal {O}(2)$ and $\mathcal {O}(3)$, respectively.
\end{enumerate}

\noi {\bf Compact strata.}

\begin{itemize}
\item  {\bf 0-dimensional.}

\begin{enumerate}
\item[$\Lambda_{5}$] \textbf{One 0-cell.}  The common vertex $v_1=1$ of $\mathcal{P}_1$ and $\mathcal{P}_2$ which is the barycenter of the cube $\mathcal{C}$. Its isotropy group is the abelian group $\Lambda_5= \hat{\mathcal{U}}(\mathfrak{H})$ of order 24 generated by the involution $T$ and the elements in the Hurwitz unitary group.  The local model for the singular points in this stratum is isometric to the orbifold $\mathcal {O}(2,  \langle 2,3,3\rangle)$.

\item[$\Lambda_{6}$] \textbf{One 0-cell.}    The vertex $v_2$ of $\mathcal{P}_{\Sigma}$ which is the image under $\mathfrak{p}$ of 2 opposite barycenters of the square faces of the cube $\mathcal{C}$. The isotropy group of this vertex is isomorphic to the group $\Lambda_6= (\Z/2\Z \times\Z/3\Z) \times \Z/2\Z= \Z/6\Z\times \Z/2\Z$ of order 12. It is the group generated by $\tau_{\mathbf{u}}T$ and $TD_{\mathbf{v}}$ where $\mathbf{u},{\mathbf{v}}=\pm{\mathbf{i}}, \pm{\mathbf{j}}, \pm{\mathbf{k}}$ and ${\mathbf{u}}\neq{\mathbf{v}}$.  The local model for the singular points in this stratum is isometric to the orbifold $\mathcal {O}(2, 6)$.

\item[$\Lambda_{7}$] \textbf{One 0-cell.}   The vertex $v_3$ of $\mathcal{P}_{\Sigma}$ which is the images under $\mathfrak{p}$ of the 12 middle points of the edges of $\mathcal{C}$. The isotropy group of this vertex is isomorphic to the group $\Lambda_7=\Z/4\Z\times  (\Z/3\Z \times\Z/2\Z) =\Z/4\Z \times \Z/6\Z $  of order 24. It is the group generated by $TD_{\mathbf{v}}$, $\tau_{\mathbf{u}+\mathbf{v}}T$, $(\tau_{\mathbf{v}}T)^2$ and $(\tau_{\mathbf{u}}T)^2$ where $\mathbf{u},\mathbf{v}=\pm{\mathbf{i}}, \pm{\mathbf{j}}, \pm{\mathbf{k}}$, and $\mathbf{u}\neq\mathbf{v}$.  The local model for the singular points in this stratum is isometric to the orbifold $\mathcal {O}(4, 6)$.

\item[$\Lambda_{8}$] \textbf{One 0-cell.}   The vertex $v_4$ of $\mathcal{P}_{\Sigma}$ which is the image under $\mathfrak{p}$ of the 8 vertices of $\mathcal{C}$. The isotropy group of this vertex is isomorphic to the group
  $\Lambda_{8}=(\Z/2\Z \times \Z/6\Z) \times  \langle 2,3,3\rangle$ of order 288 where $ \langle 2,3,3\rangle$ is the binary tetrahedral group of order 24. It is the group generated by $D_{\omega}$, $\tau_{\mathbf{i}+\mathbf{j}+\mathbf{k}}T$, $(\tau_{\mathbf{i}+\mathbf{j}}T)^2$, $(\tau_{\mathbf{i}+\mathbf{k}}T)^2$, $(\tau_{\mathbf{j}+\mathbf{k}}T)^2$, $(\tau_{\mathbf{u}} T)^2$ where
  $\mathbf{u}=\pm{\mathbf{i}},\pm{\mathbf{j}},\pm{\mathbf{k}}$.  The local model for the singular points in this stratum is isometric to the orbifold $\mathcal {O}((2, 6), \langle 2,3,3\rangle)$.
\end{enumerate}

\item  {\bf 1-dimensional.}
\begin{enumerate}

\item[$\Lambda_{9}$] \textbf{One 1-cell.}   The points of the edge of $\mathcal{P}_{\Sigma}$ which is incident with the barycenter of $\mathcal{C}$ and the barycenter of a square face of $\mathcal{C}$. Their isotropy groups are isomotphic to  $\Lambda_{9}=\Z/2 \Z \times \Z/2 \Z$ which is the group generated by
  $D_{i}T$ and $D_{j}T$.  The local model for the singular points in this stratum is isometric to the orbifold $\mathcal {O}(2,2)$.

\item[$\Lambda_{10}$] \textbf{Two 1-cells.}   The points of the two edges of $\mathcal{P}_{\Sigma}$ which are incident with the barycenter of $\mathcal{C}$ and the middle points of edges of $\mathcal{C}$. Their
  isotropy groups are isomorphic to $\Lambda_{10}=\Z/2 \Z$ which is the group generated by
  $D_{i}  T$.  The local model for the singular points in this stratum is isometric to the orbifold $\mathcal {O}(2)$.

 \item[$\Lambda_{11}$] \textbf{One 1-cell.}   The points of the edge of $\mathcal{P}_{\Sigma}$ which is incident with the barycenter of $\mathcal{C}$ and the vertex of $\mathcal{C}$. Their
  isotropy groups are isomorphic to $\Lambda_{11}=\Z/3 \Z $.  The local model for the singular points in this stratum is isometric to the orbifold $\mathcal {O}(3)$.

 \item[$\Lambda_{12}$] \textbf{Two 1-cells.}   The points of the two edges of $\mathcal{P}_{\Sigma}$ which are incident with the barycenter of a square face of $\mathcal{C}$ and a middle point of its edges. Their
  isotropy groups are isomorphic to $\Lambda_{12}=\Z/2 \Z \times \Z/3 \Z$.  The local model for the singular points in this stratum is isometric to the orbifold $\mathcal {O}(2, 3)$.

 \item[$\Lambda_{13}$] \textbf{One 1-cell.}   The points of the edge of $\mathcal{P}_{\Sigma}$ which are incident with a vertex of $\mathcal{C}$ and a middle point of its edges. Their
  isotropy groups are isomorphic to $\Lambda_{13}=\Z/2 \Z \times \Z/6 \Z$.  The local model for the singular points in this stratum is isometric to the orbifold $\mathcal {O}(2, 6)$.

\end{enumerate}

\item  {\bf 2-dimensional.}
\begin{enumerate}

\item[$\Lambda_{14}$] \textbf{One 2-cell.}   The isotropy groups of the points of the interior of the square face of $\mathcal{P}_{\Sigma}$ are isomorphic to $\Lambda_{14}=\Z/3\Z$.  The local model for the singular points in this stratum is isometric to the orbifold $\mathcal {O}(3)$.

\item[$\Lambda_{15}$] \textbf{Two 2-cell.}   The isotropy groups of the points in the interiors of the two triangle faces of $\mathcal{P}_{\Sigma}$ which contain 1 and the barycenter of a square face of $\mathcal{C}$ are isomorphic to $\Lambda_{15}=\Z/2\Z$.  The local model for the singular points in this stratum is isometric to the orbifold $\mathcal {O}(2)$.

  \end{enumerate}
\end{itemize}

\subsection{The Euler orbifold--characteristic of the Lipschitz and Hurwitz modular orbifolds.} 
\noi We use our previous computations on the order of the local groups of the strata in the singular loci of the Lipschitz and Hurwitz modular orbifolds to obtain the following:

\begin{theorem} The Euler orbifold--characteristic of the Lipschitz and Hurwitz modular orbifolds are

$$\chi ^{orb}(\mathcal{O}^{4}_{\mathfrak{L}})=\frac{1}{96} \qquad \, \rm{and}\,\qquad \chi ^{orb}(\mathcal{O}^{4}_{\mathfrak{L}})=\frac{1}{288},$$ respectively.
\end{theorem}
  
\begin{proof}  
   
\noi The Euler orbifold--characteristic of the Lipschitz and Hurwitz modular orbifolds can be computed by the alternate sums of the number of strata for each dimension in $\Sigma^{2}_{\mathfrak{L}}$ and in $\Sigma^{2}_{\mathfrak{H}}$ divided for the order of the isotropy group of a point in the stratum. 

\noi The Lipschitz modular orbifold has a stratification as CW complex with one vertex with isotropy group of order 8, another vertex with isotropy group of order 96, three vertices of order 12, and three vertices of order 24.
 It has three edges with isotropy group of order 4, six edges with isotropy group of order 6, three edges with isotropy group of order 12, and eight edges with isotropy group of order 4.  It has three 2-cells with isotropy group of order 2, three 2-cells with isotropy group of order 3  and twelve 2-cells with isotropy group of order 2. Finally it has six 3-cells and one 4-cell with isotropy groups of orders 1.

\noi The Hurwitz modular orbifold has a stratification as CW complex with one vertex with isotropy group of order 12, one vertex with isotropy group of order 288, and two vertices of order 24. It has one edge with isotropy group of order 2, one edge with isotropy group of order 3, three edges with isotropy group of order 4, two edges with isotropy group of order 6, and three edges with isotropy group of order 12. It has two 2-cells with trivial isotropy group, six 2-cells with isotropy group of order 2 and three 2-cells with isotropy group of order 3. Finally it has five 3-cells and one 4-cell with isotropy groups of orders 1.
\end{proof}
\begin{remark} Since there is an orbifold cover 
$\mathfrak{p}_{_{\mathfrak{L},\mathfrak{H}}}:\mathcal{O}^{4}_{\mathfrak{L}}\to\mathcal{O}^{4}_{\mathfrak{H}}$ of order 3 we obtain, as expected: $\chi^{\rm{orb}}(\mathcal{O}^{4}_{\mathfrak{L}})=3\chi^{\rm{orb}}(\mathcal{O}^{4}_{\mathfrak{H}})$.

\end{remark}

\noi The volume of an orbifold is the same as its fundamental domain. Then we had computed  
Vol($\mathcal{O}^{4}_{\mathfrak{L}})=3$Vol$(\mathcal{O}^{4}_{\mathfrak{H}})$ in the section 7.1. This is related to the Gauss-Bonnet-Euler Theorem for Orbifolds.
   
\section{Selberg's covers and examples of hyperbolic 4-manifolds.}\label{Selberg}

\noi In this section by a \emph{Selberg cover} we mean a covering space which is a manifold which corresponds to a torsion-free and finite-index subgroup. We have already remarked that the group $PSL(2, \mathfrak{L})$ is a
subgroup of the symmetries of the honeycomb $\{3,4,3,4\}$. This is a corollary of the results in section \ref{COXDEC}. Then the
fundamental domain $\mathcal{P}_{\mathfrak{L}}$ of the group $PSL(2,
\mathfrak{L})$ is commensurable with a hyperbolic regular right--angled
convex cell $\{3,4,3\}$ of the honeycomb $\{3,4,3,4\}$ (see also \cite{Rat}). In other words, there
is a finite subdivision of $\mathcal{P}_{\mathfrak{L}}$ and $\{3,4,3\}$
by congruent polyhedrons. The 24 vertices of $\{3,4,3\}$ are:
\begin{equation}\label{v24cell}
0, \infty, \pm \mathbf{i}, \pm \mathbf{j}, \pm \mathbf{k},  \pm \mathbf{i}\pm \mathbf{j}\pm \mathbf{k},  \frac{\pm \mathbf{i}\pm \mathbf{j}\pm \mathbf{k}}{2}.
\end{equation}

\noi Besides the points 0 and $\infty$ the other vertices are the vertices of regular
polyhedrons contained in concentrical spheres with center at the origin in the ideal boundary of
$\mathbf{H}_{\mathbb{H}}^1$: the units form six vertices whose convex
hull is an octahedron inscribed in the unit sphere. The rest are the
16 vertices of two cubes that are symmetric by the isometry $T$.

\noi We consider the parabolic group that we denote by $A(2,\mathfrak{L})$ generated by the twelve translations $\q\mapsto \q+\alpha,$ where $\alpha=\pm \mathbf{i}\pm \mathbf{j}, \pm \mathbf{i}\pm \mathbf{k}, \pm \mathbf{j}\pm \mathbf{k}$ (all possible signs are allowed). 

\noi The fundamental domain $\mathcal{P}_{{A}(2,\mathfrak{L})}$
of $A(2,\mathfrak{L})$ is a chimney (an ideal cone) with apex
the ideal point $\infty$ and whose base is a polyhedron with 12 faces,
one for each point (translation vector) in the generators. The
translation vectors form the convex hull of a cuboctahedron with 12
vertices in the middle points of the edges of the cube with vertices
$\pm \mathbf{i}\pm \mathbf{j}\pm \mathbf{k}$. To obtain the
fundamental domain $\mathcal{P}_{A(2,\mathfrak{L})}$ we
consider the Dirichlet domain of $A(2,\mathfrak{L})$, that is
the polyhedron bordered by the mediatrices (perpendicular bisectors)
of the 12 translation vectors and the origin 0. This is a rhombic
dodecahedron $\mathcal{R}_d$ whose faces consist of 12 isometric
rhombi. This is an equilateral convex polyhedron with two types of
vertices. The rhombic dodecahedron $\mathcal{R}_d$ has 6 vertices $\pm \mathbf{i}, \pm
\mathbf{j}, \pm \mathbf{k}$ adjacent to four rhombi and 8 vertices
$\frac{\pm \mathbf{i}\pm \mathbf{j}\pm \mathbf{k}}{2}$ adjacent to
three rhombi.

\noi The action of the parabolic group $A(2,\mathfrak{L})$
gives a uniform tessellation of the ideal boundary of
$\mathbf{H}_{\mathbb{H}}^1$ (with symmetry group transitive in cells, faces and
edges) that we identify with
$\mathbb{R}^3$. The cells are congruent rhomboidal dodecahedra. The
links of the vertices are of two types. In one case they are
tetrahedrons and in the other they are cubes.

\noi Following an approach which reminds the one used by Selberg (see \cite{S}) and 
taking into account 
the calculations just made for the tessellation of 
$\mathbf{H}^{1}_{\mathbb{H}}$, one is led  
to consider a singular (finite) covering for the orbifold $\mathcal{O}_{\mathfrak{L}}^4=\mathbf{H}^{1}_{\mathbb{H}}/ PSL(2,\mathfrak{L})$.

\noi The Selberg's theorem says that there exist smooth cover hyperbolic 4-manifolds of
the orbifolds $\mathcal{O}_{\mathfrak{L}}^4$ and $\mathcal{O}_{\mathfrak{H}}^4$. 
\begin{proposition} By proposition \rm{9.4} the minimal orders of Selberg covers of $O^4_{\mathfrak{L}}$ and $O^4_{\mathfrak{H}}$ are of orders 96 and 288, respectively. 
\end{proposition}

\noi In 1999,
J. Ratcliffe and T. Tschantz found 1171 noncompact hyperbolic
4-manifolds which have Euler characteristic 1 by side-pairings in a fundamental region $\{3,4,3\}$ of
the honeycomb $\{3,4,3,4\}$, see \cite{RT}. In 2004, D. Ivansi\'c
showed that the nonorientable 4-manifold numbered 1011 in the
Ratcliffe and Tschantz's list, this is, the 4-manifold $M_{1011}$ with
the biggest order of their symmetry groups, is the complement of five
Euclidean 2-torus in a closed 4-manifold with fundamental group
isomorphic to $\mathbb{Z}/{2\Z}$. Moreover, the orientable double cover
$\tilde{M}_{1011}$ is a complement of five 2-torus in the 4-sphere
\cite{I04}. In 2008 Ivansi\'c showed that 
this 4-sphere  has the same topology of  the  standard differentiable
4-sphere and not of an exotic 4-sphere, see \cite{I08}. In his doctoral thesis J.P. D\'iaz provides diagrams of this link to give an explicit model of the isotopy class of the link.

\subsection{One example by J. Ratcliffe and T. Tschantz of a hyperbolic 4-manifold} We now recall the beautiful construction of a complete, nonorientable hyperbolic 4-manifold of finite volume with six cusps whose cross sections  are 
$\mathbb{S}^1\times\mathbf{K}^2$, where $\mathbf{K}^2$ is the  Klein bottle. Let us consider the open unit ball in 
$\mathbb{H}$ with the Poincar\'e metric (see \ref{ISOD}). Let $\mathbf{C}_{24}$  denote the 24-cell whose vertices are the Hurwitz unit as seen before. There are 24 faces which are regular ideal hyperbolic octahedrons. See the figure 13. Given a face $F$ there is an opposite face $-F$ which is the face diametrically opposite to $F$ (the image under multiplication by -1). One identifies $F$ with $-F$ by a composition which consists of a reflection with respect to the hyperplane which contains $F$ followed by multiplication by -1. This composition is an orientation-reversing hyperbolic isometry which sends $\mathbf{C}_{24}$ onto a contiguous cell of the honeycomb determined by the 24-cell.  This pairing of each face with its opposite has the effect of creating a nonsingular, nonorientable, hyperbolic manifold with 6 cusps. We can take the orientable double covering.

\noi In 2005, Ratcliffe, Tschantz and Ivansi\'c showed that there are dozens of examples of non-orientables hyperbolic 4-manifolds from this list whose orientable double covers are complements of five or six Euclidean surfaces (tori and Klein bottles) in the 4-sphere, see \cite{IRT}. 

\noi Ratcliffe showed that three of these dozen complements can be
used to construct aespheric 4-manifolds that are homology spheres by
means of Dehn's fillings (see \cite{R}).

\begin{figure}
\centering
\includegraphics[scale=0.30]{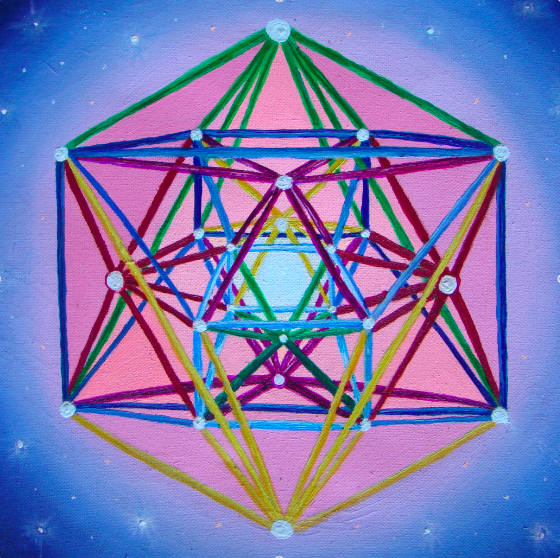}
\begin{center}
{{\bf Figure 13.} The 24-cell $\{3,4,3\}$.}
\end{center}
\end{figure}

\section{Lorentz transformations}

\noi We are mostly interested in the half-space model and in the
hyperboloid model of Lorentz and Minkowski hyperbolic models. Thus we will study the
Cayley transformations that give us isometries of the hyperbolic
models. In particular we study the representation of
$PSL(2,\mathbb{H})$ as Lorentz transformations. A
\textit{Lorentz-Minkowski matrix} $M$ 
is any $5\times 5$ matrix such that
$M^{t}JM=J$ where $M^{t}$ is the transpose matrix of
$M$ and $J$ is the matrix
$$J=\left(\begin{array}{ccccc}-1 & 0 & 0 & 0 & 0 \\0 & 1 & 0 & 0 & 0 \\0 & 0 & 1 & 0 & 0 \\ 0 & 0 & 0 & 1 & 0 \\0 & 0 & 0 & 0 & 1\end{array}\right).$$ 
We observe that the  determinant of any 
Lorentz-Minkowski matrix $M$ 
is $\pm 1$.

\noi We now describe  two 4-hyperbolic models as subsets in
$\mathbb{R}^5$: 
the hyperboloid model 
\begin{equation}\label{Lor}
\mathbf{Lor}:=\left\{(x_{0},x_{1},x_{2},x_{3},x_{4})\in\mathbb{R}^5\ :\ x_{0}>0,\ x_{1}^2+x_{2}^2+x_{3}^2+x_{4}^2=-1+x_0^2 \right\} 
\end{equation}
and the half-space model
$$\mathbf{H}^+:=\left\{(1,x_{1},x_{2},x_{3},x_{4})\in\mathbb{R}^5\ :\
x_{4}>0 \right\}.$$ 
\noi Each of these models has its corresponding
complete metric of constant curvature -1 and one can pass from one to
the other by explicit projections called \textit{Cayley
  transformations} (see \cite{Lev}).   

\noi Indeed, if $(x_{0},x_{1},x_{2},x_{3},x_{4})\in \mathbf{Lor}$ then 
\begin{equation}\label{Cayley1}
\left(1,\frac{x_{1}}{x_{0}+x_{4}},\frac{x_{2}}{x_{0}+x_{4}},\frac{x_{3}}{x_{0}+x_{4}}, \frac{1}{x_{0}+x_{4}}\right)\in \mathbf{H}^+.
\end{equation}
In fact, $x_0+x_4$ is positive since, $x_4^2-x_0^2=(x_4-x_0)(x_4+x_0)=
-(1+x_1^2+x_2^2+x_3^2)<0$ and hence either $x_0+x_4>0$
or $x_4-x_0>0$, but in this second case this is equivalent to
$x_4>x_0$ and since $x_0$ is positive, then $x_0+x_4$ is positive.
In order to prove that the function $\Phi_{ \mathbf{Lor},  \mathbf{H}^+}:  \mathbf{Lor}\to \mathbf{H}^+$
\[ \Phi_{ \mathbf{Lor},  \mathbf{H}^+}\left(x_0, x_{1},x_{2},x_{3}, x_{4}\right)=
\left(1,\frac{x_{1}}{x_{0}+x_{4}},\frac{x_{2}}{x_{0}+x_{4}},\frac{x_{3}}{x_{0}+x_{4}}, \frac{1}{x_{0}+x_{4}}\right)
\]
is a  one-to-one function, we show that it is invertible.

\noi Therefore, given $(1,\mathbf{y})=(1,y_{1},y_{2},y_{3},y_{4})\in
\mathbf{H}^+$, if
$|\mathbf{y}|^2=y_1^2+y_2^2+y_3^2+y_4^2$,  then it is readily seen that the inverse of $\Phi_{ \mathbf{Lor},  \mathbf{H}^+}$ is given by the formula:
\begin{equation}\label{Cayley2}
\Phi_{ \mathbf{Lor},  \mathbf{H}^+}^{-1}((1,y_{1},y_{2},y_{3},y_{4}))=\left(\frac{1+|y|^2}{2y_4},\frac{y_1}{y_4},\frac{y_2}{y_4},\frac{y_3}{y_4},\frac{1-|y|^2}{2y_4}\right)
\in \mathbf{Lor}
\end{equation}
because
\[\Phi_{ \mathbf{Lor},  \mathbf{H}^+}\left(\frac{1+|y|^2}{2y_4},\frac{y_1}{y_4},\frac{y_2}{y_4},\frac{y_3}{y_4},\frac{1-|y|^2}{2y_4}\right)=(1,y_{1},y_{2},y_{3},y_{4}).\]

\noi For a matrix $M$ the condition $M^{t}JM=J$ is equivalent to (BG)
 conditions, therefore

\begin{proposition}
Any Lorentz-Minkovski matrix is in one to one correspondence with a
matrix of $PSL(2,\mathbb{H})$ which satisfies (BG) conditions, i.e. the orientation-preserving isometry group of $\Hy$ which we denote by $\mathcal{M}_{\Hy}$.  
\end{proposition}

\noi  We also recall that the isometry group
of $\mathbf{Lor}$ is the group of Lorentz-Minkowski transformations
and that the subgroup of $PSL(2,\mathbb{H})$ of matrices which satisfy (BG)
conditions is isomorphic to the group of isometries of $\mathbf{H}^+$.

\noi For the  matrix associated to the general translation
$\tau_{x,y,z}=\left(\begin{array}{cc}1 &
  x\mathbf{i}+y\mathbf{j}+z\mathbf{k} \\0 & 1 \end{array}\right)$
we have the following Lorentz representation
$$\mathfrak{T}(x,y,z):=\left(\begin{array}{ccccc}1+\frac{(x^2+y^2+z^2)}{2} & x & y & z &  \frac{(x^2+y^2+z^2)}{2} 
\\x & 1 & 0 & 0 & x 
\\y & 0 & 1 & 0 & y 
\\z & 0 & 0 & 1 & z 
\\-\frac{(x^2+y^2+z^2)}{2} & -x & -y & -z & 1-\frac{(x^2+y^2+z^2)}{2}\end{array}\right)$$\\

\noi We notice that $\mathfrak{T}(x,y,z)\mathfrak{T}(x',y',z')=\mathfrak{T}(x+x',y+y',z+z')$.\\

\noi The Lorentz transformation corresponding to $T=\left(\begin{array}{cc}0 & 1  \\1 & 0 \end{array}\right)$ is the matrix:
$$\left(\begin{array}{ccccc}1 & 0 & 0 & 0 & 0 \\0 & -1 & 0 & 0 & 0 \\0 & 0 & -1 & 0 & 0 \\ 0 & 0 & 0 & -1 & 0 \\0 & 0 & 0 & 0 & -1\end{array}\right)=-J$$

\noi For $R=\left(\begin{array}{cc}\sqrt{r} & 0  \\0 & \frac{1}{\sqrt{r}} \end{array}\right)$ we have the Lorentz representation
$$D(r):=\left(\begin{array}{ccccc}\frac{1+r^2}{2r} & 0 & 0 & 0 & \frac{1-r^2}{2r} \\0 & 1 & 0 & 0 & 0 \\0 & 0 & 1 & 0 & 0 \\ 0 & 0 & 0 & 1 & 0 \\\frac{1-r^2}{2r} & 0 & 0 & 0 & \frac{1+r^2}{2r}\end{array}\right)$$

\begin{definition}\label{Lorso}  Finally if $B \in SO(4)$ we define 
\begin{equation}\label{LORSO}
\widehat{B}:=\left(
\begin{array}{c|c}
  1 & 0 \cdots 0 \\ \hline
  0 & \raisebox{-15pt}{{\huge\mbox{{$B$}}}} \\[-4ex]
  \vdots & \\[-0.5ex]
  0 &
\end{array}
\right).
\end{equation}
\end{definition}

\begin{remark} The set of matrices of the form $\widehat{B}$ can be viewed
  as the
  subgroup of $SO_+(4,1)$ of matrices 
that fix the point $(1,0,0,0,0)$; this subgroup is isomorphic to $SO(4)$.
\end{remark}

\subsection{Iwasawa decomposition for the Lorentz group $SO_+(4,1)$.} 
\noi Let  $SO_+(4,1)$ denote the Lorentz-Minkowski group. This is the group of orientation preserving isometries of the branch of the hyperboloid $\mathbf{Lor}$ in \ref{Lor}. Using the decomposition given in \ref{Iwasawa} we have the following:

\begin{proposition}[\bf Iwasawa decomposition] For a given $M \in SO_+(4,1)$ there exist unique real numbers $r,x,y,z \in \mathbb{R}, r>0$, and a unique matrix $B\in SO(4)$ such that 
\begin{equation} M=D(r) \mathfrak{T}(x,y,z) \widehat{B}.
\end{equation}
\end{proposition}

\subsection{The congruence group $\Gamma(2,\mathfrak{L})$ in the Lorentz group.}
\noi The congruence group $\Gamma(2,\mathfrak{L})$ corresponds, via the
Cayley transformations (\ref{Cayley1}) and (\ref{Cayley2}), to the
subgroup $\mathfrak{G}(2,\mathfrak{L})\subset{SO_+(4,1)}$ of Lorentz
transformations generated by translations $\mathfrak{T}(x,y,z)$ with $x,y,
z\in\Z$ such that $x+y+z\equiv0$ (mod 2), and the matrix $-J$.

\noi Let $SO_+(4,1,\Z)\subset{SO_+(4,1)}$ be the subgroup of Lorentz
matrices with integer coefficients.  Since in the case of the
congruence group $\Gamma(2,\mathfrak{L})$ the translation matrices
$\mathfrak{T}(x,y,z)$ have integer coefficients, we have:

\begin{proposition} The congruence group $\mathfrak{G}(2,\mathfrak{L})$ is
  generated by matrices with integer coefficients, i.e, 
$\mathfrak{G}(2,\mathfrak{L})\subset{SO_+(4,1,\Z)}$.
\end{proposition}
\begin{remark} Let $\mathbf{Lor}(\Z):=\mathbf{Lor}\cap{\Z^5}$ be the set of points with integer coefficients which are in the Lorentz-Minkowski hyperboloid, then
the orbit of $(1,0,0,0,0)$ under the action of $\mathfrak{G}(2,\mathfrak{L})$ is contained in $\mathbf{Lor}(\Z)$ thus this set is infinite and countable.
\end{remark}

\begin{theorem} The fundamental domain of $\Gamma(2,\mathfrak{L})$ in
  $\mathbf{H}^1_{\mathbb{H}}$ is a convex noncompact polytope of
  dimension four and of finite volume, with ideal vertices $v_1, v_2, v_3, v_4, v_5, v_6$ and $v_{\infty}$ where 
$\{v_1, v_2, v_3, v_4, v_5, v_6\}$ are the six vertices $\pm
  \mathbf{i},\pm \mathbf{j},\pm \mathbf{k},$ of a regular ideal
  octahedron and $v_{\infty}$ is the point at infinity and eight real
  vertices which are the vertices $(\frac{1}{2})(1\pm \mathbf{i} \pm
  \mathbf{j} \pm \mathbf{k})$ of a regular cube.

\end{theorem}
\noi The group $PSL(2,\mathfrak{L})$ leaves invariant the regular
hyperbolic honeycomb corresponding to the right-angle 24-cell with
Schl\"afli symbol $\{3,4,3,4\}$.  We recall that the 24-cell (or
{\em icositetrachoron}\footnote {It is also called an {\em octaplex} (short for
``octahedral complex"), {\em octacube}, or {\em polyoctahedron}, being constructed
of octahedral cells.})
is a convex regular 4-polytope, 
whose boundary is composed of 24 octahedral cells with six
meeting at each vertex, and three at each edge. Together they have 96
triangular faces, 96 edges, and 24 vertices. See the figure 13.
It is possible to give an (ideal) model of the 24-cell 
by considering the convex hull (of the images) of the 24 unitary Hurwitz numbers 
via the Cayley transform $\Psi(\mathbf{q})=(1+\mathbf{q})(1-\mathbf{q})^{-1}$.
The coordinates in $\mathbb{Z}^4$ of all unitary Hurwitz numbers 
can be obtained by considering the 8 quadruples
$(\pm 1,0,0,0)$, $(0,\pm 1,0,0)$, $(0,0,\pm 1,0)$ $(0,0,0,\pm 1)$
together with the 16 quadruples $(\pm 1/2,\pm 1/2,\pm 1/2,\pm 1/2)$.

\noi The congruence group $\Gamma(2,\mathfrak{L})$  leaves invariant the regular
tessellation generated by the 24-cell. Therefore (see  \cite{RT}) we have the following
\begin{theorem} The group $\mathfrak{G}(2,\mathfrak{L})$ contains a finite-index
  subgroup whose fundamental domain is the 24-cell. 
Therefore  as shown in \cite{RT}, the group ${\mathfrak{G}}(2,\mathfrak{L})$ contains 22
  subgroups which act freely on $\mathbf{Lor}$ and whose quotients are
  hyperbolic orientable and of finite volume.
\end{theorem}

\section{The Hurwitz modular group $PSL(2,\mathfrak{H})$ in the Lorentz model.}

\noi The algebra of the quaternions $\mathbb{H}$ is isomorphic to the real algebra of $4\times4$ matrices generated by $I_4, S_\mathbf{i},S_\mathbf{j}, S_\mathbf{k}$, where $I_4$ is the identity $4\times 4$ matrix and 

$$S_\mathbf{i} :=\left(\begin{array}{cccc} 0 & -1 & 0 & 0 \\ 1 & 0 & 0 & 0 \\  0 & 0 & 0 & -1 \\ 0 & 0 & 1 & 0\end{array}\right), \quad S_\mathbf{j} :=\left(\begin{array}{cccc}  0 & 0 & -1 & 0 \\ 0 & 0 & 0 & 1 \\  1 & 0 & 0 & 0 \\ 0 & -1 & 0 & 0\end{array}\right)$$

\noi and

$$S_\mathbf{k} :=\left(\begin{array}{cccc}  0 & 0 & 0 & -1 \\ 0 & 0 & -1 & 0 \\  0 & 1 & 0 & 0 \\ 1 & 0 & 0 & 0\end{array}\right).$$

\noi  This follows from the fact that $S_\mathbf{i}^2=S_\mathbf{j}^2=S_\mathbf{k}^2=-I_4$ and $S_\mathbf{i}S_\mathbf{j}=S_\mathbf{k}, S_\mathbf{j}S_\mathbf{k}=S_\mathbf{i}$ and  $S_\mathbf{k}S_\mathbf{i}=S_\mathbf{j}$. In particular the group of Hurwitz units $\mathcal{U}(\mathfrak{H})$ consists of the 24 special orthogonal matrices: 
$\pm I_4, \pm S_\mathbf{i},\pm S_\mathbf{j},\pm S_\mathbf{k},  \frac12({\pm I_4  \pm S_\mathbf{i} \pm S_\mathbf{j} \pm S_\mathbf{k}})$ (all possible 16 combinations of signs are allowed). We remark that this group is isomorphic to the binary tetrahedral group.
\begin{definition} Let $\mathcal{U}(\mathfrak{H},\mathbf{Lor})\subset{SO_+(4,1)}$ be the finite group of order 24 given,
using \ref{Lorso}, by the Lorentz matrices:
$$
\pm\widehat{I_4},\ \pm\widehat{S_\mathbf{i}},\ \pm\widehat{S_\mathbf{j}},\ \pm\widehat{S_\mathbf{k}},\
\frac12(\pm\widehat{I_4} \pm \widehat{S_\mathbf{i}}\pm\widehat{S_\mathbf{j}}\pm\widehat{S_\mathbf{k}}).
$$ 
 \end{definition}
 \noi  We remark that the inversion $T$ corresponds to $-\widehat{I_4} \in\mathcal{U}(\mathfrak{H},\mathbf{Lor})$ i.e. the matrix $-J$.
\begin{proposition} In the Lorentz model the group $PSL(2,\mathcal{H})$ corresponds to the subgroup of $SO_+(4,1)$,
denoted by $\Gamma_{\mathfrak{H}}$ generated by $\mathcal{U}(\mathfrak{H},\mathbf{Lor})$ and the translations $\mathfrak{T}(n,m,p)$ where $n,m,p \in \mathbb{Z}$ .
\end{proposition}

\noi Since $PSL(2,\mathfrak{L})\subset PSL(2,\mathfrak{H})$ we have a corresponding subgroup 
$\Gamma_{\mathfrak{L}}\subset{\Gamma_{\mathfrak{H}}}$ of the Lorentz group.

\noi The fundamental domain of $\Gamma_{\mathfrak{H}}$ is contained in the fundamental domain of $\Gamma_{\mathfrak{L}}$ and therefore as we seen before the group $PSL(2, \mathfrak{L})$ leaves invariant the hyperbolic honeycomb whose cell is the 24-cell.

\section{Poincar\'e extension to the hyperbolic fifth dimension.} 

\noi As we have seen before the quaternionic projective line
$\mathbf{P}^1_{\mathbb{H}}$ can be identified with the unit sphere
$\mathbb{S}^4$ in $\mathbb{R}^5$ and therefore $\mathbb S^4$ is the
boundary of the closed unit ball $\mathbf D^5\subset\mathbb R^5$. As
usual, we identify the interior of $\mathbf D^5$ with the real
hyperbolic 5-space $\mathbf{H}_\mathbb{R}^5$.  Since
$PSL(2,\mathbb{H})$ acts conformally on $\mathbb{S}^4
\cong \mathbf{P}^1_{\mathbb{H}}$, by Poincar\'e Extension Theorem each element
$\gamma\in{PSL(2,\mathbb H)}$ extends canonically to a conformal
diffeomorphism of $\mathbf{D}^{5}$ which restricted to
$\mathbf{H}_\mathbb{R}^5$ is an orientation preserving isometry $\tilde{\gamma}$ of the open
5-disk $\mathbf{B}^{5}$ with the Poincar\'e's metric. Reciprocally, any orientation preserving
isometry of $\mathbf{H}_\mathbb{R}^5$ extends canonically to the ideal
boundary $\mathbb{R}^{4}\cup \{\infty \}$ as an element of
$PSL(2,\mathbb H)$.  Thus the map $\gamma\mapsto\tilde{\gamma}$ is an
isomorphism and $PSL(2,\mathbb{H})=Isom_{+}\mathbf{H}_\mathbb{R}^5$.\\

\noi We can also refer to the upper half-space model
$\mathbf{H}_\mathbb{R}^5$ which can be regarded as
$\{(\mathbf{q},t)\,\,:\,\mathbf{q}\in\mathbb{H},
t>0\}$.  The action of $PSL(2,\mathbb{H})$ extends to the closed half-space
$\overline{\mathbf{H}}_\mathbb{R}^5:=\{(\mathbf{q},t)\,:\,\mathbf{q}\in\mathbb{H},
t\geq0\}$.

\noi We give the quaternionic Poincar\'e's extension theorem uses the notion of Dieudonn\'e determinant, let us briefly recall its definition and properties.

\begin{definition}
Let $\gamma=\left(\begin{array}{cc}a & b \\ c & d\end{array}\right) \in PSL(2,\mathbb{H})$.
The Dieudonn\'e determinant of $\gamma$ is defined by the formula:
$$
\rm{det}_{\mathbb{H}}(\gamma)=\sqrt{|a|^2|d|^2+|c|^2|b|^2-2\Re{c\overline{a}b\overline{d}}}$$
\end{definition}
\noi The Dieudonn\'e determinant has the following properties (\cite{BisGen}):
\begin{enumerate}
\item $\gamma=\left(\begin{array}{cc}a & b \\ c & d\end{array}\right)$ is in $PSL(2,\mathbb{H})$ if
and only if  $\rm{det}_{\mathbb{H}}(\gamma)\neq0$.
\item If $\gamma_1, \gamma_2\in{PSL(2,\mathbb{H})}$ then   $\rm{det}_{\mathbb{H}}(\gamma_1\gamma_2)=
 \rm{det}_{\mathbb{H}}(\gamma_1) \rm{det}_{\mathbb{H}}(\gamma_2)$.
\item If $c\neq0$  
the Dieudonn\'e determinant of $\gamma$ is equal to $|(cb-cac^{-1}d)|$ (see Lemma 2.3. in \cite{BisGen} page 7). 
If $c=0$,  $\det_{\mathbb{H}}(\gamma)=|a||d|$. If $a,b,c,d$ commute then $\det_{\mathbb{H}}(\gamma)=|ad-bc|$.
\end{enumerate}
\begin{theorem} 

\noi Let $\gamma=\left(\begin{array}{cc}a & b \\ c & d\end{array}\right) \in PSL(2,\mathbb{H})$, then 
its Poincar\'e extension is given explicitly in the upper half-space model $\mathbf{H}_\mathbb{R}^5$ as follows:

$$
\overline\gamma(\mathbf{q},t)=\left(\left(\frac{1}{|c\mathbf{q}+d|^2+|c|^2t^2}\right)((a\mathbf{q}+b)(\overline{\mathbf{q}}\overline{c}+\overline{d})+a\overline{c}t^2),
\frac{\rm{det}_{\mathbb{H}}(\gamma)\text{t}}{|c\mathbf{q}+d|^2+|c|^2t^2}\right).
$$
\end{theorem}

\begin{proof} 
\noi If  $\gamma(\q)$ is the affine transformation $\gamma(\q)=a\q+b$ ($a\neq0$) then its Poincar\'e's extension 
$\overline\gamma$ is given by the formula:
$$
\overline\gamma(\mathbf{q},t)=(a\q+b, |a|t)=(a\q+b,\rm{det}_{\mathbb{H}}(\gamma)\text{\emph{t}}).
$$

\noi If $c\neq0$ and $\gamma(\q)=(c\q+d)^{-1}$ its Poincar\'e extension is given by the formula
$$
\overline\gamma(\mathbf{q},t)=\left(\frac{\overline{(c\q+d)}}{|c\q+d|^2+|c|^2t^2}, \frac{|c|t}{|c\q+d|^2+|c|^2t^2}\right)=
$$
$$
=\left(\frac{\overline{(c\q+d)}}{|c\q+d|^2+|c|^2t^2}, \frac{\rm{det}_{\mathbb{H}}(\gamma)\emph{t}}{|c\mathbf{q}+d|^2+|c|^2t^2}\right)
$$

\noi  If $c\neq0$ we have:
$$
\gamma(\q)=(a\q + b)(c\q + d)^{-1} = ac^{-1} + (b-ac^{-1}d)(c\q + d)^{-1}.
$$
Hence $\gamma$ is a composition of affine transformations with the map $\q\mapsto\q^{-1}$. Applying the previous formulas we obtain:
$$
\overline\gamma(\mathbf{q},t)=\left(\frac{(b-ac^{-1}d)(\overline{\mathbf{q}}\overline{c}+\overline{d})}{|c\mathbf{q}+d|^2+|c|^2t^2}+a{c}^{-1},\dfrac{|(cb-cac^{-1}d)|t}{|c\mathbf{q}+d|^2+|c|^2t^2}\right)=
$$
$$
=\left(\frac{(b-ac^{-1}d)(\overline{\mathbf{q}}\overline{c}+\overline{d})}{|c\mathbf{q}+d|^2+|c|^2t^2}+
\frac{a{c}^{-1}[(c\mathbf{q}+d)(\overline{\mathbf{q}}\overline{c}+\overline{d}) +|c|^2t^2]}{|c\mathbf{q}+d|^2+|c|^2t^2},
\frac{\rm{det}_{\mathbb{H}}(\gamma)\emph{t}}{|c\mathbf{q}+d|^2+|c|^2t^2}\right)=
$$
$$
=\left(\frac{1}{|c\mathbf{q}+d|^2+|c|^2t^2}((a\mathbf{q}+b)(\overline{\mathbf{q}}\overline{c}+\overline{d})+a\overline{c}t^2),
\frac{\rm{det}_{\mathbb{H}}(\gamma)\emph{t}}{|c\mathbf{q}+d|^2+|c|^2t^2}\right)
$$

\noi A simply direct calculation
implies that the transformations corresponding to an affine transformation and $\q\mapsto\q^{-1}$ acts conformally on
$\overline{\mathbf{H}}_\mathbb{R}^5$ and therefore it acts by
isometries on the open half-space $\mathbf{H}_\mathbb{R}^5$. Since every extension is a finite composition of transformations of this type one has that the extensions act conformally and hence by isometries. \end{proof}
     
\begin{remark} The Poincar\'{e} extension
of a transformation when $t=0$ corresponds to the action of $\gamma$
 on $\mathbb{H}$.
 \end{remark}

\noi Therefore we may assume that any discrete subgroup of
$PSL(2,\mathbb{H})$ acts either conformally on $\mathbb{S}^4
\cong \mathbf{P}^1_{\mathbb{H}}$ or isometrically on $\mathbf{H}_\mathbb{R}^5$ and
rephrase Proposition \ref{pp}.

\begin{proposition} Let $\Gamma\subset{PSL(2,\mathbb{H})}$ be a
  discrete subgroup acting 
isometrically on 
$\mathbf{H}_\mathbb{R}^5$. Then this action is proper and
  discontinuous. Therefore $M:=\mathbf{H}_\mathbb{R}^5/\Gamma$ is a
  complete 5-dimensional hyperbolic orbifold. If in addition $\Gamma$
  acts freely then $M$ is a hyperbolic 5-manifold.
\end{proposition}

\subsection{Hyperbolic Kleinian 5-orbifolds and 5-manifolds.}
\noi Important examples of discrete subgroups of $PSL(2,\mathbb{H})$ are obtained when the coefficients are in the
ring $\mathbb{H}(\mathbb Z)$ of Lipschitz integers or the ring   $\mathbb{H}ur$ of Hurwitz integers i.e. $PSL(2,\mathbb{H}(\mathbb{Z}))$ and $PSL(2,\mathbb{H}ur(\mathbb{Z}))$ the subgroup of $PSL(2,\mathbb{H})$ whose entries are Hurwitz integers (\emph{not necessarily satisfying (BG) conditions}).
\noi We summarize our previous considerations in the following 
\begin{proposition} The groups $PSL(2, \mathbb{H}(\mathbb{Z}))$, $PSL(2,\mathbb{H}ur(\mathbb{Z}))$, $PSL(2,\mathfrak{L})$ and
  $PSL(2, \mathfrak{H})$ 
are discrete subgroups of
  $PSL(2,{\mathbb{H}})$.  In particular their Poincar\'e extensions
  act properly and discontinuously in
  $\mathbf{H}_\mathbb{R}^5$. 
 
\end{proposition} 
  
\begin{definition}  We can therefore define: 
$$\mathcal{O}^5_{\mathbb{H}(\mathbb{Z})}:=\mathbf{H}_\mathbb{R}^5/PSL(2,
  \mathbb{H}(\mathbb{Z})), \,\,\,\, \mathcal{O}^5_{\mathbb{H}ur(\mathbb{Z})}:=\mathbf{H}_\mathbb{R}^5/PSL(2,
  \mathbb{H}ur(\mathbb{Z})),$$
  $$ \mathcal{O}^5_{\mathfrak{L}}:=\mathbf{H}_\mathbb{R}^5/PSL(2,
  \mathfrak{L})\ \,\,\, \mathrm{and}\  \,\,\, \mathcal{O}^5_{\mathfrak{H}}:=\mathbf{H}_\mathbb{R}^5/PSL(2,\mathfrak{H});$$ these
  are real hyperbolic 5-dimensional orbifolds. 
 \end{definition} 
  
 \noi However the orbifolds $\mathcal{O}^5_{\mathbb{H(Z)}}$ and $\mathcal{O}^5_{\mathbb{H}ur(\mathbb{Z)}}$ are real 5-dimensional orbifolds of finite hyperbolic volume and the orbifolds $\mathcal{O}^5_{\mathfrak{L}}$  and $\mathcal{O}^5_{\mathfrak{H}}$ are real 5-dimensional orbifold of infinite hyperbolic volume.

\noi We identify $\mathbf{H}^5_{\mathbb{R}}$ with pairs
 $(\mathbf{q},t)$ where $\mathbf{q}\in\mathbb{H}$ and $t>0$. Let $\mathcal{P}^5$ be the 5-dimensional chimney which is the set: 
 $$\mathcal{P}^5:=\{(\mathbf{q},t)\in \mathbf{H}^5_{\mathbb{R}} : \mathbf{q}=x_0+x_1{\mathbf{i}}+x_2{\mathbf{j}}+x_3{\mathbf{k}}, \rm{where} -1/2\leq{x_n}\leq1/2,\, (n=0,1,2,3), \rm{and} \, |\mathbf{q}|^2+|t|^2\geq1   \}$$

\noi We can chose fundamental domains $\mathcal{P}^5_{\mathbb{H}(\mathbb{Z})}$ and $\mathcal{P}^5_{\mathbb{H}ur(\mathbb{Z})}$ for the actions of the Poincar\'e extensions of  $PSL(2, \mathbb{H}(\mathbb{Z}))$ and $PSL(2, \mathbb{H}ur(\mathbb{Z}))$ on $\mathbf{H}^5_{\mathbb{R}}$ which are subsets of $\mathcal{P}^5$ so that $\mathcal{P}^5_{\mathbb{H}(\mathbb{Z})}$ and $\mathcal{P}^5_{\mathbb{H}ur(\mathbb{Z})}$ have both finite hyperbolic volume.

\noi Some fundamental domains $\mathcal{P}^5_{\mathfrak{L}}$ and $\mathcal{P}^5_{\mathfrak{H}}$ for the actions of the Poincar\'e's extensions of  $PSL(2, \mathfrak{L})$ and $PSL(2, \mathfrak{H})$ on $\mathbf{H}^5_{\mathbb{R}}$ are the product of a line $\R$ by their respective fundamental domains in $\Hy$, these are $\mathcal{P}^4_{\mathfrak{L}} \times \R$ and $\mathcal{P}^4_{\mathfrak{H}} \times \R$. Then $\mathcal{P}_{\mathfrak{L}}$ and $\mathcal{P}_{\mathfrak{H}}$ have both infinite hyperbolic volume and one only end. Moreover, $\mathcal{O}^5_{\mathfrak{L}}$ and $\mathcal{O}^5_{\mathfrak{H}}$ are homeomorphic to $\mathcal{O}^4_{\mathfrak{L}} \times \R$ and $\mathcal{O}^4_{\mathfrak{H}} \times \R$, respectively. Then their underlying spaces are homeomorphic to $\R^5$ and their singular loci are the product of a line $\R$ by their respective singular loci $\Sigma_{\mathfrak{L}}$ and $\Sigma_{\mathfrak{H}}$ in $\mathcal{O}^4_{\mathfrak{L}}$ and $\mathcal{O}^4_{\mathfrak{H}}$, respectively.

\noi The orbifold Euler characteristic  of these orbifolds are:

$$\chi ^{orb}(\mathcal{O}^{5}_{\mathfrak{L}})=-\chi ^{orb}(\mathcal{O}^{4}_{\mathfrak{L}})=-\frac{1}{96} \qquad \, \rm{and}\,\qquad \chi ^{orb}(\mathcal{O}^{5}_{\mathfrak{H}})=-\chi ^{orb}(\mathcal{O}^{4}_{\mathfrak{H}})=-\frac{1}{288},$$
 respectively.

\noi The vertical ideal projection by geodesics asymptotic to the point at infinity of  $\mathcal{P}^5_{\mathbb{H}(\mathbb{Z})}$ onto the ideal boundary  identified with $\mathbb{R}^{4}$ is a subset of a hypercube $\{4,3,3\}$. Moreover the tessellation on the ideal boundary is commensurable with the classical self--dual and right--angled
tessellation $\{4,3,3,4\}$. The extension of $PSL(2, \mathbb{H}(\mathbb{Z}))$ to
$\mathbf{H}^{5}_{\mathbb{R}}$ is commensurable with the ideal hyperbolic tessellation $\{3,4,3,3,4\}$. This remarks imply:
\begin{proposition} 
The extension to
$\mathbf{H}^{5}_{\mathbb{R}}$ of $PSL(2,\mathbb{H}(\mathbb{Z}))$ is a subgroup of the group of symmetries of the ideal hyperbolic tessellation or honeycomb $\{3,4,3,3,4\}$.
\end{proposition}

\noi The cells of the hyperbolic tessellation $\{3,4,3,3,4\}$ are congruent copies of $\{3,4,3,3\}$ which is a regular Euclidean tessellation of $\mathbb{R}^4$. The ideal vertical figure is the regular self-dual Euclidean tessellation of $\mathbb{R}^4$ with symbol $\{4,3,3,4\}$.

\noi In the $\mathbb{H}ur$ case we obtain that the fundamental region has 24
faces and is the regular polytope $\{3,4,3\}$ and the tessellation in
$\mathbb{R}^{4}$ is $\{3,4,3,3\}$. 

\begin{proposition}
The extension to
$\mathbf{H}^{5}_{\mathbb{R}}$ of $PSL(2,\mathbb{H}ur(\mathbb{Z}))$ is a subgroup of the group of symmetries of the ideal self--dual hyperbolic tessellation or honeycomb $\{3,3,4,3,3\}$.
\end{proposition}

\section{Appendix: Orbifolds}

\subsection{Introduction to orbifolds}

For the sake of completeness, in this appendix we recall some basic facts about orbifolds (see also \cite{BMP}) .

\begin{definition}
A smooth $\mathrm{orbifold}$ modeled on the manifold $M$ is a
metrizable topological space $\mathcal{O}$ endowed with a collection
$\{(U_{i}, \tilde{U_{i}},\phi_{i},\Gamma_{i}\}_{i}$, called an {\it
  atlas}, where for each $i$,

\begin{enumerate}
\item $U_{i}$ is an open subset of $\mathcal{O}$, 
\item $\tilde{U_{i}}$ is an open subset of $M$, 
\item $\phi_{i}:\tilde{U_{i}} \to U_{i}$ is a continuous map (called a chart) and 
\item $\Gamma_{i}$ is a finite group of diffeomorphisms of $\tilde{U_{i}}$ satisfying the following conditions:
\begin{itemize}
\item The $\tilde{U_{i}}$'s cover $\mathcal{O}$.

\item Each $\phi_{i}$ factors through a homeomorphism between $\tilde{U_{i}}/\Gamma_{i}$ and $U_{i}$.

\item The charts are compatible in the following sense: for every
  $x\in \tilde{U_{i}}$ and $y\in \tilde{U_{j}}$ with
  $\phi_{i}(x)=\phi_{j}(y)$, there is a diffeomorphism $\psi$ between
  a neighborhood $V$ of $x$ and a neigborhood $W$ of $y$ such that
  $\phi_{j}(\psi(z))=\phi_{i}(z)$ for all $z\in V.$
\end{itemize}
\end{enumerate}
\end{definition} 

\begin{definition} 
The $\mathrm{local}$ $\mathrm{group}$ of $\mathcal{O}$ at a point
$x\in \mathcal{O}$ is the group $\Gamma_{x}$ defined as follows: let
$x\in U$ and $\phi:\tilde{U} \to U$ be a chart. Then $\Gamma_{x}$ is
the stabilizer of any point of $\phi^{-1}(x)$ under the action of
$\Gamma$.
\end{definition}

\noi It is well-defined up to isomorphism. Local groups are isomorphic
to subgroups of $O(n)$.

\begin{definition} 
If $\Gamma_{x}$ is trivial, we said that $x$ is $\mathrm{regular}$,
otherwise it is $\mathrm{singular}$. The $\mathrm{singular}$
$\mathrm{locus}$ of $\mathcal{O}$ is the set $\Sigma_{\mathcal{O}}$ of
singular points of $\mathcal{O}$.
\end{definition}

\begin{definition}
The $\mathrm{underlying}$ $\mathrm{space}$ of $\mathcal{O}$ is the
topological space $|\mathcal{O}|$ obtained by forgetting the orbifold
structure.
\end{definition}

\noi If $(M,g)$ is a smooth homogeneous space with respect to the
Riemannian metric $g$ and if $\Gamma$ is a discrete group of
isometries of $M$ acting properly and discontinuously on $M$, then the
orbit space $M/\Gamma$ is an orbifold and the canonical map $p:M\to
{M/\Gamma}$ provides us with a natural orbifold atlas. This is the
definition of a {\it good orbifold} modeled on the homogeneous
manifold $(M,g)$ in the sense of Thurston \cite{T}.

\noi There are many properties of the differential geometry and topology for manifolds which have counterparts for orbifolds. For example, the Euler characteristic of a manifold $M$ which is denoted by $\chi(M)$ or a CW complex defined as the integer number iqual to their alternating sum of the number of cells. A fundamental property of the Euler characteristic is that it is multiplicative under finite covers: If $M^{\prime} \to M$ is an $m-$fold cover, then
$$ \chi (M^{\prime})=m\chi (M).$$

\begin{definition} Let $\mathcal{O}$ be an orbifold with a decomposition as a CW complex so that the local group is constant on each open cell. Let $c$ be a cell, $\Gamma_c$ be the local group on $c$ and $|\Gamma_c |$ denote  its order. Then the Euler characteristic of the orbifold $\mathcal{O}$ is
$$\chi^{orb}(\mathcal{O}):=\Sigma_{c} \frac{(-1)^{\rm{dim} \, \it{c}}}{|\Gamma_c |}.$$
\end{definition}

The Euler characteristic of an orbifold is a rational number with a multiplicative property: If $M \to \mathcal{O}$ is an m-fold cover and $M$ is a manifold, then 
 $$ \chi (M)=m\chi^{orb} (\mathcal{O}).$$

\begin{remark} Most of the  orbifolds considered in this paper are good orbifolds 
and indeed they are mainly  modeled on 
\begin{itemize}
\item $\R^n$, the Euclidean $n-$space;
\item $\mathbb {S}^n$, the unitary sphere in $\R^{n+1}$ with the standard Riemannian metric;
\item $\mathbf H^n_{\R}$, the Siegel half space in $\R^{n}$, with the
  corresponding hyperbolic metric. In particular $\mathbf H_{\mathbb
    C}^1=\mathbf H_{\mathbb R}^2$ and $\mathbf H_{\mathbb H}^1=\mathbf
  H_{\mathbb R}^4$ (with the corresponding Poincar\'e metric) are the
  hyperbolic spaces which may be regarded as the half space of
  elements in $\mathbb{K}$ ($\mathbb{K}=\mathbb{C}$ or
  $\mathbb{K}=\mathbb{H}$) with positive real part or $\mathbf
  B_{\mathbb{K}} $ the unit ball in $\mathbb{K}$, with
  $\mathbb{K}=\mathbb{C},\mathbb{H}$
\end{itemize}
Correspondingly, $\R^n/\Gamma$ (where $\Gamma$ is a discrete subgroup
of isometries of $\R^n$ acting properly and discontinuously) are the
Euclidean orbifolds; similarly, $\mathbb S^n/\Gamma$ where $\Gamma$ is a finite
subgroup of $SO(n+1)$ are the spherical orbifolds and $\mathbf H^n/\Gamma$ are the hyperbolic orbifolds where $\Gamma$ is a finite subgroup of
$SO(n+1,1)$ and $\mathbf H^n_{\R}$ is the real hyperbolic $n-$space.
\end{remark}
\noi Taking into account the previous considerations, we have the following
\begin{proposition}\label{pp} If $\Gamma$
is a discrete subgroup of
isometries on $\mathbf{H}_\mathbb{H}^1$, then
$M:=\mathbf{H}_\mathbb{H}^1/\Gamma$ is a complete hyperbolic 4-real
orbifold.  If in addition $\Gamma$ acts freely, then $M$ is a
hyperbolic 4-real manifold.
\end{proposition}

\begin{definition} An orbifold of the form $M=\mathbf{H}_\mathbb{H}^1/\Gamma$ where $\Gamma$
is a discrete subgroup of
isometries  $\mathbf{H}_\mathbb{H}^1$
is called a {\it one dimensional quaternionic hyperbolic orbifold}.
\end{definition}

\subsection{Basic examples of orbifolds.}
For the sake of completeness we list here some classical examples of
orbifolds, whose constructions will be however useful in the sequel.

\begin{figure}
\centering
\includegraphics[scale=0.45]{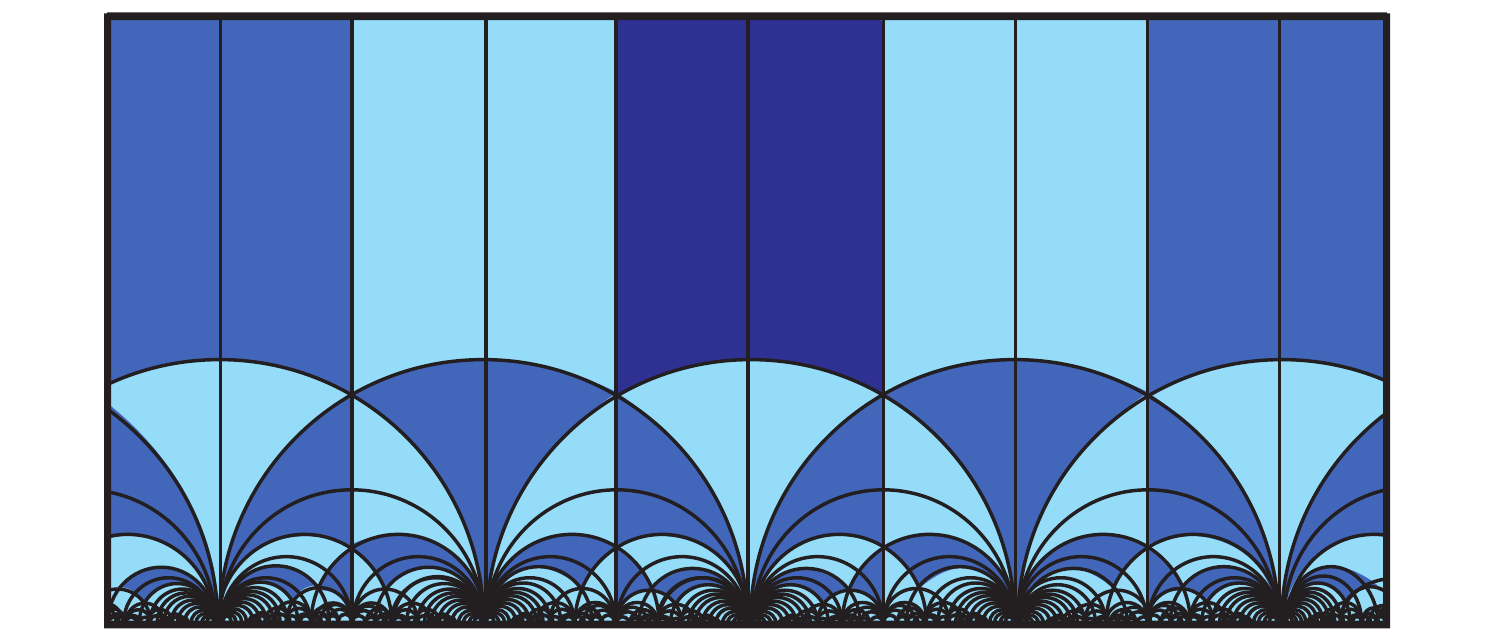}
\begin{center}{{\bf Figure 14.} A fundamental domain for the action of the modular group $PSL(2, \mathbb{Z})$ on the hyperbolic plane $\mathbf{H}^2_{\R}$
 .}
\end{center}
\end{figure}

\begin{enumerate}
\item {\bf The modular orbifold and the modular group $PSL(2, \mathbb{Z})$}
 We consider the action by isometries of the modular group $PSL(2, \mathbb{Z})$ on
the hyperbolic real plane $\textbf{H}^{2}_{\mathbb{R}}$. A fundamental domain is a triangle with one ideal point an two other vertices were the sides have an angle of $\pi/3$. This is the triangle with Coxeter notation $\triangle (3,3,\infty)$. See the figure . The modular group $PSL(2, \mathbb{Z})$ is a subgroup of the group of symmetries of the regular tessellation of $\textbf{H}^{2}$ whose tiles are congruent copies of the triangle $\triangle (3,3,\infty)$. We can describe the Cayley graph and a presentation of the group $PSL(2, \mathbb{Z})$ in terms of 2 generators  and 2 relations to obtain:

$$PSL(2, \mathbb{Z}) \, \,  = \,\, \langle a,b | a^2=(ab)^3=1\rangle \, \, =\,\, \Z/2\Z \ast \Z/3\Z.$$

\noi  The quotient
$\mathcal{O}:=\textbf{H}^{2}_{\mathbb{R}}/PSL(2,\mathbb{Z})$ has
underlying topological space the plane $\mathbb{R}^{2}$ (or $\C$) and $\Sigma_{\mathcal{O}}$
consists of 2 distinguished conical points. The local groups of the
singular points are $\mathbb{Z}/{2\Z}$ and $\mathbb{Z}/{3\Z}$ modeled on a
group of two and three elements, respectively, consisting of
hyperbolic rotations of angles $\pi$ and $2\pi/3$. The Euler characteristic of the orbifold $\mathcal{O}$ is -1/6. Thus a minimal Selberg cover is of order 6.

\item {\bf The orbifold associated to the Picard-Gauss group $PSL(2, \mathbb{Z}[\ii])$.}
 We consider the action by isometries of the modular group $PSL(2, \mathbb{Z}[\ii])$ on
the hyperbolic real space $\textbf{H}^{3}_{\mathbb{R}}$. The quotient
$\mathcal{O}:=\textbf{H}^{3}_{\mathbb{R}}/PSL(2,\mathbb{Z}[\ii])$ has
underlying space the 3-space $\mathbb{R}^{3}$. Its singular locus $\Sigma_{\mathcal{O}}$ is the 1-skeleton of a squared pyramid without the apex. The Euler characteristic of the orbifold $\mathcal{O}$ is 0.

\item \textbf{Carom or carambole billiard} For
each $n-$dimensional Euclidean space $\mathbb{R}^n$ there is a
self-dual, right-angled honeycomb whose cells are hypercubes and whose
Sch\"afli symbol is $\{4,3,...,3,4\}$. There is a subgroup $\Gamma$ of
their symmetries generated by reflections in the sides of their
hypercubical cells. The fundamental region of the action of $\Gamma$
in $\mathbb{R}^n$ is a cell. The quotient
$\mathcal{O}:=\mathbb{R}^n/\Gamma$ is an  Euclidean orbifold whose
shape is a solid hypercubical cell. The orbifold $\mathcal{O}$ has 
the $n$-ball as underlying space. Furthermore, 
$\Sigma_{\mathcal{O}}$ consists of the boundary of the
cell, this is the hypercube $\{4,3,...,3\}$ in the
$(n-1)-$sphere. There is a stratification of $\Sigma_{\mathcal{O}}$
and the local group of each singular point in the interior of the
sides of $\{4,3,...,3\}$ is the group with two elements
$\mathbb{Z}/{2\Z}$ whose representations consist of a
reflection. Others points in $\Sigma_{\mathcal{O}}$ are right-angled
corner points just as in the carom or carambole billiard table. The Euler characteristic of the orbifold $\mathcal{O}$ is 0. In general, the Euler characteristic of a compact Euclidean $n-$orbifold is 0.

\item \textbf{Real Kummer surface (a pillow)}. Let
$\mathbf{T}^{2}=\mathbb{S}^{1}\times \mathbb{S}^{1}=\{(z_1,z_2)
\in\C^2\, :\,\, |z_1|=|z_2|=1\}$ be the 2--torus.  Let $\tau$ be the
involution $\tau(z_1,z_2)=(\overline{z}_1,\overline{z}_2)$. Then $\tau$ has four
fixed points $(1,1), (1,-1), (-1,1), (-1,-1)$.  The quotient
$\mathcal{O}:=\mathbf{T}^{2}/\tau$ has 
the 2-sphere as underlying space 
and $\Sigma_{\mathcal{O}}$ consists of 4 distinguished conical
points. The local group of each singular point is the group with two
elements $\mathbb{Z}/2\Z$ whose representations consist of a rotation
of angle $\pi$. The Euler characteristic of the orbifold $\mathcal{O}$ is 0.

\item \textbf{Real Kummer $n$-orbifold} Let 
$$
\mathbf{T}^{n}=\mathbb{S}^{1}\times\cdots\times \mathbb{S}^{1}=\{(z_1,\ldots,z_n) \in\C^n\, :\,\,
|z_1|=\ldots=|z_n|=1\}
$$ be the $n$-torus.  \noi Let $\tau$ be the involution
$\tau(z_1,\cdots,z_n)=(\overline{z}_1,\cdots,\overline{z}_n)$. Then $\tau$ has
$2^n$ fixed points $(\pm 1,\cdots, \pm 1)$. Then the quotient
$\mathcal{O}:=\mathbf{T}^{n}/\tau$ is an orbifold.  The local group of each
singular point is the group with two elements $\mathbb{Z}/{2\Z}$ whose
representations consist of rotation $(z_1,\cdots,z_n)\mapsto
(-z_1,\cdots,-z_n)=(\overline{z}_1,\cdots,\overline{z}_n)$.  Thus at the
singular point the orbifold has a neighorhood homeomorphic to the cone
over $\mathbb P_{\R}^{n-1}$. The Euler characteristic of the orbifold $\mathcal{O}$ is 0. There is a figure of the Real Kummer 3-orbifold in section 6.

\item {\bf Real tori with conformal
  multiplication}. In analogy to the classical theory of abelian
manifolds with complex multiplication one can define a real torus with
conformal multiplication. Let ${{\mathbb T}^n}(\Lambda):={\mathbb
  T}^n/\Lambda$ be a real $n$-torus corresponding to the lattice
$\Lambda\subset{\mathbb R}^n$. Let $A:{\mathbb T}^n\to{\mathbb T}^n$
be an automorphism such that it lifts to a conformal isomorphism of
${\mathbb R}^n$. We say that ${\mathbb T}^n(\Lambda)$ is a torus with
{\it conformal multiplication}.  To admit a conformal automorphism
imposes symmetry restrictions on the lattice $\Lambda$. Since by a
classical theorem by Bieberbach (see \cite{Bie,Bie1,W}) the number of compact flat
manifolds and flat orbifolds of a given dimension is finite it follows
that the group $G$ of conformal automorphisms of a torus admiting
conformal multiplication is finite. Then ${\mathbb T}^n/G$ is a flat
orbifold. The Euler characteristic of this orbifold ${\mathbb T}^n/G$ is 0. 

\end{enumerate}
\bigskip

\begin{figure}
\centering
\includegraphics[scale=0.33]{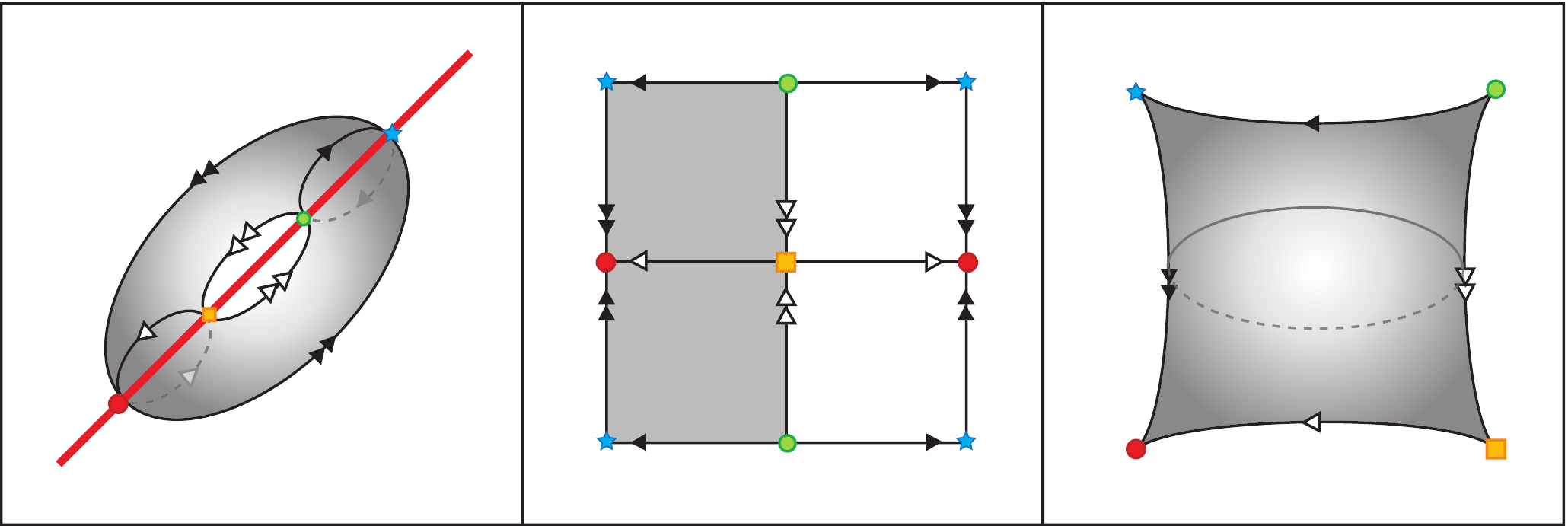}
\begin{center}{{\bf Figure 15.} The Kummer surface is a pillow, a 2-sphere with 4 conical points.}
\end{center}
\end{figure}

\noi {\bf Acknowledgments.} This research 
paper has been made possible  thanks to the financial support
generously given 
by the Italian institute
Gruppo Nazionale Strutture Algebriche, Geometriche e Applicazione (GNSAGA) of
the Istituto Nazionale di Alta Matematica (INdAM) ``F. Severi''; in particular
the first author received from GNSAGA a grant to visit the 
Dipartimento di Matematica e
Informatica (DiMaI) ``U. Dini'' (Florence) in Spring 2013 and 
the last author  was financially supported by GNSAGA to visit 
the Instituto de Matem\'aticas at  Universidad Nacional Aut\'onoma de M\'exico,
Unidad de Cuernavaca in Summer 2014.
Furthermore, the three authors are very grateful to Centro
Internazionale per la Ricerca Matematica (CIRM) in Trento and to 
the organizers of the Conference ``Complex Analysis and Geometry - XXI'' for
the invitation to deliver a talk in June 2013.
Finally special thanks go to the International Centre for Theoretical
Physics (ICTP) ``A. Salam'' of Trieste for the very nice hospitality
offered to the authors  in occasion of Schools/Workshops and Visits 
in the last two years.
The last author was also partially supported by 
Ministero Istruzione Universit\`a e Ricerca MIUR Progetto
di Ricerca di Interesse Nazionale PRIN â``Propriet\`a Geometriche delle Variet\`a Reali e Complesse''.

\noi The first and second named authors were financed by grant IN103914, PAPIIT, DGAPA,
Universidad Nacional Aut\'onoma de M\'exico.

\end{document}